\documentclass[11pt]{amsart}
\usepackage{amsmath,amsfonts,amssymb,amsthm,color,tikz,todonotes,xspace,euscript,nicefrac,mathrsfs,calligra}
\usepackage{tikz-cd}
\usepackage[all]{xy}
\usepackage[utf8]{inputenc}
\usepackage[russian,english]{babel}
\usepackage[normalem]{ulem}
\usepackage{bm}
\usepackage{mathtools}

\makeatletter
\def\@settitle{\begin{center}%
  \baselineskip14\p@\relax
    \bfseries \@title
  \end{center}%
}
\def\@setauthors{%
  \begingroup
  \def\thanks{\protect\thanks@warning}%
  \trivlist
  \centering\footnotesize \@topsep30\p@\relax
  \advance\@topsep by -\baselineskip
  \item\relax
  \author@andify\authors
  \def\\{\protect\linebreak}%
  {\authors}%
  \ifx\@empty\contribs
  \else
    ,\penalty-3 \space \@setcontribs
    \@closetoccontribs
  \fi
  \endtrivlist
  \endgroup
}
\makeatother
\usepackage{hyperref}
\newtheorem{Theorem}{Theorem}[section]
\newtheorem{Proposition}[Theorem]{Proposition}
\newtheorem{Lemma}[Theorem]{Lemma}
\newtheorem{Question}[Theorem]{Question}
\newtheorem{Corollary}[Theorem]{Corollary}

\newtheorem{Definition}[Theorem]{Definition}
\newtheorem{Example}[Theorem]{Example}
\newtheorem{Remark}[Theorem]{Remark}

\numberwithin{equation}{section}
\renewcommand{\theequation}{\arabic{section}.\arabic{equation}}

\tikzset{wei/.style={draw=red,double=red!40!white,double distance=1.5pt,thin}}
\newcounter{subeqn}
\renewcommand{\thesubeqn}{\theequation\alph{subeqn}}
\newcommand{\subeqn}{%
  \refstepcounter{subeqn}
  \tag{\thesubeqn}
}
\makeatletter
\@addtoreset{subeqn}{equation}
\newcommand{\newseq}{%
  \refstepcounter{equation}
}

\newcommand{\acom}[1]{\todo[inline,color=green!20]{ Alex: #1 }}

\newcommand{\ocom}[1]{\todo[inline,color=magenta!20]{Oded: #1}}
\newcommand{\bcom}[1]{\todo[inline,color=yellow!20]{Ben: #1}}

\newcommand{\bX}{\mathbb{X}}
\newcommand{\nc}{\newcommand}

\newcommand{\renc}{\renewcommand}
\nc{\bla}{{\boldsymbol{\la}}}
\nc{\mmod}{\operatorname{-mod}}
\nc{\h}{\mathfrak h}
\nc{\g}{\mathfrak g}
\renc{\tg}{\tilde{\mathfrak g}}
\nc{\ft}{\mathfrak t}
\nc{\fM}{\mathfrak M}
\nc{\bM}{\mathbf M}
\nc{\bR}{\mathbf R}
\nc{\Bm}{\mathbf m}
\nc{\bS}{\mathbf S}
\nc{\bT}{\mathbf T}
\nc{\bU}{\mathbf U}
\nc{\bV}{\mathbf V}
\nc{\bi}{\mathbf i}
\nc{\bp}{\mathbf p}
\nc{\barQ}{\bar{Q}}
\nc{\barP}{\bar{P}}
\nc{\barX}{\bar{X}}
\nc{\hsigma}{\hat{\sigma}}
\nc{\TL}{\tilde{\mathscr{T}}_{\mathcal{L}}}
\nc{\bs}{\mathbf s}
\renc{\C}{\mathbb C}
\nc{\Sym}{\operatorname{Sym}}
\nc{\acham}{\eta}
\nc{\tU}{\mathcal{U}}
\DeclareMathOperator{\Long}{
     \mathsf{Long}}
\nc{\PolKLR}{\mathsf{Pol}}
\nc{\BY}{\mathbf{Y}}
\nc{\longi}{{\boldsymbol{\ell}}}
\nc{\red}{\operatorname{red}}
\nc{\ind}{\operatorname{ind}}
\nc{\yz}{z}
\nc{\YZ}{Z}
\nc{\Z}{\mathbb Z}
\nc{\R}{\mathbb R}
\nc{\N}{\mathbb N}
\nc{\B}{\mathcal B}
\nc{\M}{\mathcal M}
\nc{\cE}{\mathcal E}
\nc{\cF}{\mathcal F}
\nc{\fB}{\mathfrak B}
\nc{\con}{\sim}
\nc{\pgl}{\mathfrak{pgl}}
\nc{\ev}{\mathsf{ev}}
\nc{\Hom}{\operatorname{Hom}}
\nc{\End}{\operatorname{End}}
\nc{\res}{\operatorname{res}}
\nc{\al}{\alpha}
\nc{\vp}{\varphi}
\nc{\Cth}{S_h}
\nc{\cO}{\mathcal{O}}
\nc{\fg}{\mathfrak{g}}
\nc{\one}{\mathbf{1}}
\nc{\bb}{\mathbf{b}}
\nc{\ext}{\operatorname{Ext}}
\nc{\out}{\operatorname{out}}
\nc{\FY}{FY}
\nc{\ep}{\epsilon}
\nc{\bz}{{\mathbf z}}
\nc{\inn}{\operatorname{in}}
\nc{\BK}{{\reflectbox{\rm R}}}
\nc{\Bi}{\mathbf{i}}
\nc{\Ba}{\mathbf{a}}
\nc{\Bj}{\mathbf{j}}
\nc{\Bb}{\mathbf{b}}
\nc{\Bnu}{{\boldsymbol{\nu}}}
\nc{\tGamma}{\tilde{\Gamma}}
\nc{\tGammabR}{\tGamma_{\bR}}
\nc{\GammabR}{\Gamma_{\bR}}
\nc{\diam}{\diamond}
\nc{\la}{\lambda}
\nc{\Yml}{Y_\mu^\lambda}
\nc{\FYml}{FY_\mu^\lambda}
\nc{\bgam}{{\boldsymbol{\gamma}}}
\nc{\blam}{{\boldsymbol{\lambda}}}
\nc{\gr}{\operatorname{gr}}
\nc{\Spec}{\operatorname{Spec}}
\nc{\Stendhal}{Stendhal\xspace}

\nc{\Tsetlin}{\foreignlanguage{russian}{Цетлин}\xspace}

\nc{\GT}{Gelfand-Tsetlin\xspace}
\nc{\GTc}{\operatorname{GT}}
\nc{\MaxSpec}{\operatorname{MaxSpec}}
\nc{\Cartan}{\C[H_\bullet^{(\bullet)}]}
\nc{\tmetric}{\mathscr{\tilde T}}
\nc{\metric}{\mathscr{T}}
\nc{\pmmetric}{{}_{\pm}\mathscr{T}}
\nc{\pmetric}{{}_{+}\mathscr{T}}
\nc{\mmetric}{{}_{-}\mathscr{T}}
\nc{\Pol}{\mathsf{Pol}}
\nc{\hh}{h}
\nc{\wtmodY}{{Y^\la_\mu\operatorname{-wtmod}}}
\nc{\wtmodFY}{{FY^\la_\mu\operatorname{-wtmod}}}
\nc{\wtmodBK}{{\BK\operatorname{-wtmod}}}
\nc{\fdFY}{{FY^\la_\mu\operatorname{-mod}_{\operatorname{fd}}}}
\nc{\OFY}{{FY^\la_\mu{\text{-}\cO}}}
\nc{\fdY}{{Y^\la_\mu\operatorname{-mod}_{\operatorname{fd}}}}
\nc{\OY}{{Y^\la_\mu{\text{-}\cO}}}

\nc{\yMon}{\mathsf{a}}
\nc{\zMon}{\mathsf{b}}

\newcommand{\Irr}{\operatorname{Irr}}

\newcommand{\arxiv}[1]{\href{http://arxiv.org/abs/#1}{\tt arXiv:\nolinkurl{#1}}}
\renewcommand{\O}{\mathcal O}
\nc{\Gr}{\mathsf{Gr}}
\nc{\Grlmbar}{\Gr^{\overline{\lambda}}_\mu}
\nc{\excise}[1]{}
 \usetikzlibrary{decorations.pathreplacing,backgrounds,decorations.markings,babel}

\title[On category $\cO$ for affine Grassmannian slices]{On category $\cO$ for affine Grassmannian slices\\ and categorified tensor products}

\author[Kamnitzer]{Joel Kamnitzer}
\address{J.~Kamnitzer: Department of Mathematics, University of Toronto, Canada}
\email{jkamnitz@math.toronto.edu}

\author[Tingley]{Peter Tingley}
\address{P.~Tingley: Department of Mathematics and Statistics, Loyola University, Chicago, United States}
\email{ptingley@luc.edu}

\author[Webster]{Ben Webster}
\address{B.~Webster: Department of Pure Mathematics, University of Waterloo \&
Perimeter Institute for Theoretical Physics, Canada}
\email{ben.webster@uwaterloo.ca}

\author[Weekes]{Alex Weekes}
\address{A.~Weekes: Perimeter Institute for Theoretical Physics, Canada}
\email{aweekes@perimeterinstitute.ca}

\author[Yacobi]{Oded Yacobi}
\address{O.~Yacobi: School of Mathematics and Statistics, University of Sydney, Australia}
\email{oded.yacobi@sydney.edu.au}

\date{\today}

\begin{document}

\begin{abstract}
Truncated shifted Yangians are a family of algebras which naturally quantize slices in the affine Grassmannian.  These algebras depend on a choice of two weights $\la$ and $\mu$ for a Lie algebra $\mathfrak{g}$, which we will assume is simply-laced. In this paper, we relate the category $\mathcal{O}$ over truncated shifted Yangians to categorified tensor products: for a generic integral choice of parameters, category $\mathcal{O}$ is equivalent to a weight space in the categorification of a tensor product of fundamental representations defined by the third author using KLRW algebras. We also give a precise description of category $\mathcal{O}$ for arbitrary parameters using a new algebra which we call the parity KLRW algebra.  In particular, we confirm the conjecture of the authors that the highest weights of category $\mathcal{O}$ are in canonical bijection with a product monomial  crystal depending on the choice of parameters.

This work also has interesting applications to classical representation theory.  In particular, it allows us to give a  classification of simple Gelfand-Tsetlin modules of $U(\mathfrak{gl}_n)$ and its associated W-algebras. 
\end{abstract}

\maketitle
\tableofcontents
\section{Introduction}
\subsection{Two geometric models and symplectic duality}
Let $ G $ be a simply-laced semisimple group with Lie algebra $\g$.
There are two geometric constructions of the finite-dimensional irreducible representations of $ \g $.  First we have the geometric Satake correspondence (due to Lusztig \cite{Lu83}, Ginzburg \cite{G}, and Mirkovic-Vilonen \cite{MV}) which constructs representations using the affine Grassmannian $ \Gr = G^\vee((z))/G^\vee[[z]] $ of the Langlands dual group of $G$.  The second construction involves the cohomology of quiver varieties constructed using the Dynkin diagram of $ \g$ (due to Nakajima \cite{Nak3}).  One goal of this paper is to answer the following question.
\begin{Question}
What is the relationship between these two geometric constructions?
\end{Question}

In \cite{BLPWgco}, Braden, Licata, Proudfoot, and the third author proposed that the framework of symplectic duality could be used to explain the connection between these two geometric models. In particular, they conjectured that there should exist a Koszul duality between certain categories of modules over the quantizations of these varieties.

More precisely, let $ \lambda, \mu $ be a pair of dominant weights for $ \g $, such that $ \la - \mu = \sum_i m_i \alpha_i $, with $ m_i \in \N $.  Then we can consider the affine Grasmmannian slice
$$ \Gr^\lambda_\mu  := \overline{G^\vee[[z]] z^\lambda} \cap G^\vee_1[z^{-1}] z^\mu \subset \Gr.$$
By the geometric Satake correspondence, the intersection cohomology of $\Gr^\lambda_\mu$ is (non-canonically) isomorphic to $V(\lambda)_\mu$, the $\mu$-weight space of the irreducible representation of $\g$ of highest weight $\lambda$.

In \cite{KWWY}, 80\% of the authors proved that $ \Gr^\lambda_\mu $ is an affine Poisson variety.  We also introduced the truncated shifted Yangian $ Y^\la_\mu $, an algebra which quantizes $ \Gr^\la_\mu $.  In fact, there are a family of such algebras, $ Y^\la_\mu(\bR) $, depending on a parameter  $ \bR \in \prod \C^{\lambda_i} $, where $ \lambda = \sum \lambda_i  \varpi_i $ (Definition \ref{def:Ymula}).  In \cite{BFNslices}, the definitions of $ \Gr^\la_\mu$ and $ Y^\la_\mu $ were generalized to the case of non-dominant $ \mu $.

The algebra $ Y^\la_\mu $ contains a polynomial subalgebra and thus we can speak about weight modules for $ Y^\la_\mu $ (see Definition \ref{Def5.1}).  Category $\cO$ for this algebra consists of those weight modules whose weights are bounded above.  This category will be the main object of study in this paper.

On the other hand, associated to $ \la, \mu $, we have the Nakajima quiver variety $ \M(\la, \mu) $.  Its top cohomology is also isomorphic to $ V(\la)_\mu$.  This quiver variety is defined as the Hamiltonian reduction of the contagent bundle of a vector space of framed quiver representations by the group $ H = \prod_i GL_{m_i} $.

As a Hamiltonian reduction, the quiver variety admits a natural quantization.  The category $\cO$ of highest weight modules for this quantization has been studied extensively.  In particular, in \cite[Th. A']{Webqui}, the third author proved (building on work of Rouquier and Varagnolo-Vasserot) that it is Koszul dual to the category of modules over of a certain combinatorial/diagramatic algebra, called a KLRW algebra.

\subsection{Parity KLRW algebras}
KLRW algebras (also called tensor product categorifications or red-black algebras) were introduced by the third author in \cite{Webmerged}, following the foundational work of Khovanov-Lauda and Rouquier \cite{KLI}, \cite{Rou2KM}.  These are algebras of string diagrams, containing red and black strands, modulo some local relations.  These algebras are used to categorify tensor products of irreducible representation of $ \fg $.  More precisely, we have the KLRW algebra ${}_-T $, whose category of modules carries a categorical $ \fg $-action. The Grothendieck group of ${}_-T\operatorname{-mod}$ is isomorphic to a tensor product of fundamental representations $ \bigotimes_i V(\varpi_i)^{\otimes \la_i} $.

Let $\bR$ be an integral set of parameters (cf. Section \ref{subsection:Notation}).  In this paper, we introduce the parity KLRW algebra which is a subalgebra ${}_-P^\bR \subset {}_-T$ (Definition \ref{def:parityalgebra}).  The categorical $ \g $-action on ${}_-T\mmod$ preserves $ {}_- P^\bR\mmod $ (Lemma \ref{lem:categoricalactions}), and thus the category of $ {}_- P^\bR\mmod $ categorifies a subrepresentation of $ \bigotimes_i V(\varpi_i)^{\otimes \la_i} $.  For generic values of these parameters, this representation is the full tensor product, whereas for very special values, it is just the irreducible representation $V(\la) $.  As with all such categorifications, ${}_- P^\bR$ is a direct sum of algebras ${}_- P^\bR_\mu$, such that $ {}_- P^\bR_\mu\mmod $ corresponds to the $\mu$-weight space of the Grothendieck group.

The main result of this paper is the following:
\begin{Theorem} \label{th:intro2}\hfill
There is an equivalence of categories from category $\cO$ over $Y^\la_\mu (\bR)$ to $ {}_- P^\bR_\mu\mmod $.
\end{Theorem}


This theorem has an immediate corollary which establishes the categorical symplectic duality between affine Grassmannian slices and quiver varieties.

\begin{Corollary} \label{th:intro1}
For a generic integral choice of $\bR$, the  quadratic dual of category $\cO$ over $ Y^\la_\mu $ is the category $\cO$ for the quiver variety $\M(\la, \mu) $ (associated to a $\C^*$-action depending on $\bR$).  If $\la$ is a sum of minuscule weights, then both category $\cO$'s are Koszul, and they are Koszul dual to each other.
\end{Corollary}

\begin{proof}
 We will just sketch the argument, since this duality is not the main focus of the current paper and since very similar arguments are used in \cite{Webdual} to establish Koszul duality in more general settings.

 Let $\bR$ be a generic set of parameters (cf. \cite[Section 2.5]{KWWY}).  In this case, we have that ${}_- P^\bR={}_- T^\bR$, and by Theorem \ref{th:intro2}, the quadratic dual of the category $\cO$ of $Y^\la_\mu (\bR)$ is the quadratic dual of ${}_-T^\bR\mmod$.
We refer the reader to \cite{MOS} for the definition of quadratic dual; in down to earth terms, ${}_-T^\bR$ is Morita equivalent to a (very difficult to describe) positively graded algebra, and we mean the representations of the quadratic dual of that algebra.

By \cite[Theorem 12]{MOS} the quadratic dual of ${}_-T^\bR\mmod$ is equivalent to the category $LCP({}_-T^\bR)$ of linear complexes of projective objects in ${}_-T^\bR\mmod$.  By \cite[Theorem A']{Webqui} this latter category is equivalent to (a graded lift of) the category $\cO$ of  $\M(\lambda,\mu)$.
In the special case of minuscule weights, \cite[Theorem B]{Webqui} implies that ${}_-T^\bR\mmod$ is Koszul, and thus quadratic dual and Koszul dual coincide.  \end{proof}

\subsection{Highest weights and monomial crystals}
  The polynomial subalgebra of $Y^\la_\mu $ which we use to define category $\O$ is isomorphic to $ P^\Sigma $, where $P$ is a polynomial ring and $\Sigma$ is a product of symmetric groups acting on the variables.   Since $ P^\Sigma $ is a partially symmetrized polynomial ring, we can think of the weights of $Y^\la_\mu $ as collections of multisets $ \bS$.

  The algebra $Y^\la_\mu(\bR) $ is a quotient of the shifted Yangian $ Y_\mu $ (Definition \ref{def: shifted Yangian}).  The representation theory of the shifted Yangian is much simpler, since $Y_\mu$ admits a PBW basis.  In particular, $Y_\mu$ has Verma modules for any choice of multisets $ \bS $.  On the other hand, the Verma modules for $ Y^\la_\mu(\bR)$ are much more difficult to understand; there are only finitely many of them for each $ \bR $.

\begin{Question}
For which $ \bS $ does $ Y^\la_\mu(\bR)$ have a Verma module with highest weight $ \bS$?
\end{Question}

In \cite{KTWWY}, we formulated a conjectural answer using the product monomial crystal $ \B(\bR) $.  This is a $\g$-crystal whose elements are collections of rational monomials in variables $ \yMon_{i,k}$, where $i\in I$ and $k\in \Z$.  It depends in a subtle way on the parameters $ \bR $; for generic values it is isomorphic to the tensor product crystal $ \bigotimes_i \B(\varpi_i)^{\otimes \la_i} $, but for special values it is isomorphic to the irreducible crystal $ \B(\la) $ (cf. Theorem \ref{th:subcrystal}).  In this paper, we prove our conjecture from \cite{KTWWY}.

\begin{Theorem}[Corollary \ref{cor:moncrystalconj}]
\label{th:moncrystal}
Let $\bR$ be an integral set of parameters.  There is a map $ \bS \mapsto  \yMon_\bR \zMon_\bS^{-1} $ which gives a bijection between the possible highest weights for $Y^\la_\mu(\bR) $ and the product monomial crystal $ \B(\bR)$.
\end{Theorem}
As explained in \cite{KTWWY}, this theorem is motivated from Nakajima's equivariant version of the Hikita conjecture.  Thus in proving this theorem, we have proved a weak form of the equivariant Hikita conjecture for the symplectic dual pair $ \Gr^\la_\mu, \mathcal{M}(\la, \mu) $.

There is a crystal structure on $\Irr({}_-P^\bR\mmod)$, the set of equivalence classes of simple ${}_-P^\bR$-modules.  We prove that taking ``highest weight'' induces a crystal isomorphism $ \Irr({}_-P^\bR\mmod) \cong \B(\bR) $ (Theorem \ref{thm:highest-bijection}).  Given this identification of crystals, Theorem \ref{th:moncrystal} follows from Theorem \ref{th:intro2}.

\subsection{The categorical action}
By \textit{transport de structure}, Theorem \ref{th:intro2} defines a categorical $ \g $-action on the direct sum over all $ \mu $ of category $ \mathcal O $ for $ Y^\la_\mu $.  We note that this is an abelian action; the Chevalley generators $ \mathcal E_i, \mathcal F_i $ act by exact functors.  In this way, this category $\g$-action is similar to the famous Bernstein-Frenkel-Khovanov \cite{BFK} action of $ \mathfrak{sl}_2 $ on blocks of category $ \mathcal O $ for $ \mathfrak{sl}_n $.  In fact, our work can be seen as a direct generalization of their construction (though the link is not immediate due to differences in the definition of category $ \mathcal O $.)
  
  Of course, it would be preferable to describe this categorical action without using the equivalence from Theorem \ref{th:intro2}.  In a forthcoming paper \cite{KTWWY2}, we will construct quantum Hamiltonian reductions relating truncated shifted Yangians.  We will prove that these quantum Hamiltonian reductions give rise to the categorical $ \fg $ action via induction and restriction functors.  These functors generalize the Bezrukavnikov-Etingof \cite{BEind} induction and restriction functors for modules over Cherednik algebras.

\subsection{Coulomb branch algebras}
Given a group $ H $ and a representation $ V$, Braverman-Finkelberg-Nakajima \cite{BFN} defined the Coulomb branch of the 3d $\mathcal{N}=4$ supersymmetric gauge theory associated to the pair $ (H,V) $.  They also defined an algebra $ \mathcal A(H,V) $ quantizing this Coulomb branch (here we specialize $ \hbar = 1$).
In \cite{BFNslices}, it is proved that when $H = \prod GL_{m_i} $ and $ V  $ is the vector space of framed quiver representations corresponding to $ \la $ and $ \mu$, then the quantized Coulomb branch $\mathcal A(H,V) $ admits a homomorphism from a truncated shifted Yangian, and that this is an isomorphism  in finite ADE type when $\mu$ is dominant.  Both of these restrictions are removed in a recent paper of the fourth author: by \cite[Theorem A]{Weekes}, the quantum Coulomb branch for any simply-laced quiver gauge theory is a truncated shifted Yangian, as in Definition \ref{def:Ymula}.

In this paper, we assume that $ \g $ is a simply-laced Kac-Moody Lie algebra whose Dynkin diagram is bipartite (note that this generalizes the finite-dimensional simply-laced simple Lie algebras).  This is the setting in which we prove Theorems \ref{th:intro2} and \ref{th:moncrystal}.  These assumptions are mostly for the purposes of simplifying the combinatorics; Theorem \ref{th:intro2} can be generalized to the symmetrizable case by using weighted KLRW algebras \cite{WebwKLR}; we will prove this in future work \cite{KTWWY2}.

In \cite{Webdual}, the third author considered the case of arbitrary $ (H, V) $. He introduced a combinatorially defined algebra depending on $ (H, V) $ and proved an equivalence similar to Corollary \ref{th:intro1} in this context.  The present paper can be thought of a specialization of \cite{Webdual} to the quiver case.  Indeed our methods are similar; however, we emphasize that the present paper can be read independently of \cite{Webdual} and \cite{BFNslices, BFN} with no need to explicitly use the Coulomb machinery.

We remark that Braverman-Finkelberg-Nakajima have formulated a geometric Satake conjecture in the context of generalized affine Grassmannian slices for arbitrary symmetric Kac-Moody Lie algebras (\cite[Conjecture 3.25]{BFNslices}). In future work \cite{KTWWY2}, we plan to prove this conjecture using techniques similar to those of this paper.

\subsection{\GT modules}
\label{sec:GZ-intro}

In the case where $\la=N\omega_1$ is a multiple of the first fundamental weight, the algebras $Y^{\la}_\mu$ for $\mu$ dominant are exactly the W-algebras of $\mathfrak{gl}_N$.  The weight modules of $Y^{\la}_\mu$ in this case are the {\bf \GT modules} of the corresponding $W$-algebra.  
  
  The results of this paper give a classification of the simple \GT modules in terms of crystal combinatorics (cf. Cor. \ref{cor:GZ-bijection}), and a combinatorial description of representation theoretic quantities such as the weight multiplicities of simples.  We describe this application of our work in Section \ref{sec:gelf-zetl-modul}.  We will study in more detail the relationship of this approach to other works on \GT modules in the future.

\subsection{Outline of the proof of Theorem \ref{th:intro2}}
\label{sec:proofoverview}

The bulk of this paper is devoted to proving Theorem \ref{th:intro2}.  The proof proceeds by introducing several related algebras.  The following diagram gives an overview of the relationships between their module categories:
$$
\tikz[very thick]{\matrix[row sep=12mm,column sep=17mm,ampersand replacement=\&]{
\& \node(a){ $\TL^\bR\operatorname{-mod}_{\operatorname{nil}}$}; \& \node(c){$\BK(\bR)\operatorname{-wtmod}$}; \& \\
\node(e){${}_-P_\mu^\bR\operatorname{-mod}$}; \& \node(b){${}_-\tmetric_\mu^\bR\operatorname{-mod}$}; \&  \node(d){$FY^\la_\mu(\bR)\text{-}\cO^-$}; \& \node(f){$Y^\la_\mu(\bR)\text{-}\cO^-$}; \\
};
\draw[->] (e) -- node[above,midway]{\ref{lem:Morita}} (b);
\draw[->] (d) -- node[above,midway]{\ref{sec:flagYang}} (f);
\draw[->] (b) -- node[above,midway]{\ref{co:tilde-main}} (d);
\draw[->] (a) -- node[above,midway]{\ref{thm:bigequiv}} (c);
\draw[->] (a) --  (b);
\draw[->] (c) -- node[right,midway]{\ref{thm:flagYa}} (d);
}
$$
All the horizontal functors are equivalences, and the two vertical ones are induced by idempotent truncation.
At the rightmost node of the diagram is the category $\cO$ for $Y_\mu^\lambda(\bR)$, and recall our aim is to prove the equivalence ${}_-P^\bR_\mu\operatorname{-mod} \cong \OY^-$.

Our argument is structured as follows.  First we observe that the parity KLRW algebra is Morita equivalent to the \textit{metric KLRW algebra} ${}_-\tmetric_\mu^\bR$ (Definition \ref{def:metricKLR}).  Roughly speaking, ${}_-\tmetric_\mu^\bR$  consists of KLRW diagrams where the red strands  carry additional data coming from $ \bR $, which we call longitudes, and the black strands also admit longitudes in such a way that the longitudes are weakly increasing as we move from left to right.  We can think of the longitude as the ``x''-coordinate of a strand.

We also also enlarge $Y_\mu^\lambda(\bR)$ to a  Morita equivalent algebra $FY_\mu^\lambda(\bR)$, called the \textit{flag Yangian} (Definition \ref{def:flagyang}).  The Morita equivalence restricts to an equivalence between their category $\cO$'s.

Thus it suffices to prove the equivalence ${}_-\tmetric_\mu^\bR\operatorname{-mod}\cong \OFY^-$, which we do by relating the metric KLRW algebra and the flag Yangian to yet larger algebras.  The linchpin is the \textit{KLR-Yangian algebra} $\BK(\bR)$ (Definition \ref{def:Ya}), which is a version of the extended BFN category from \cite{Webdual}.  It consists of KLR-like diagrams drawn on a cylinder, with additional relations along the ``seam'' of the cylinder.  We then deduce Theorem \ref{th:intro2}  using an equivalence between $\BK(\bR)$-weight modules and modules over $\TL^\bR$, the \textit{coarse metric KLRW algebra} (Definition \ref{def:coarselongs}).  This latter algebra is a further generalization of the metric KLRW algebra, in which we weaken the weakly increasing condition on longitudes.

\subsection*{Acknowledgement}
We thank Volodymyr Mazorchuk for pointing out the connection of our results to \GT modules.
J.K. was supported by NSERC through a Discovery Grant.
B.W. was supported in part by the NSF under Grant
  DMS-1151473, by NSERC through a Discovery Grant and by the Alfred P. Sloan Foundation through a Research Fellowship.
O.Y. was supported in part by the ARC through grants DP180102563 and DE15010116.  This research was supported in part by Perimeter Institute for Theoretical Physics. Research at Perimeter Institute is supported by the Government of Canada through the Department of Innovation, Science and Economic Development Canada and by the Province of Ontario through the Ministry of Research, Innovation and Science.


\section{Background}


\subsection{Notation}
\label{subsection:Notation}
Fix a bipartite graph with vertex set $ I = I_{\bar{0}} \cup I_{\bar{1}} $.  We
call vertices in $I_{\bar{0}}$ {\bf even} and those in $I_{\bar{1}}$
{\bf odd}.  We write $ i \con j $ if $i $ and $ j $ are connected.  We orient the graph so that arrows always point
from even vertices to odd ones and write $ i \to j $ for the oriented edges.  We also fix a total order on $I$, such that all  even vertices come before all odd vertices.

Let $ \g $ be the derived simply-laced Kac-Moody Lie algebra whose Dynkin diagram is $ I$.  We fix a triangular decomposition $\g=\mathfrak{n}_-\oplus \mathfrak{h}\oplus \mathfrak{n}$. Since $\g=[\g,\g]$, we have that $\mathfrak{h}$ has a basis given by the simple coroots $\al_i^\vee$. As usual, we call the $\Z$-span $\Lambda^\vee$ of these the {\bf coroot lattice}, and its dual $\Lambda\subset \mathfrak{h}^* $  the {\bf weight lattice} of $ {\g} $.  The weight lattice has a basis of   {\bf fundamental weights} $\varpi_i$ dual to the simple coroot basis of $\Lambda^\vee$. For a dominant weight $ \lambda$, we have that $ \lambda = \sum_i \lambda_i \varpi_i $ for unique non-negative integers $\lambda_i$. We let $V(\lambda)$ be the irreducible representation of highest weight $\lambda$.   

We let $ \C^\lambda = \prod_i \C^{\lambda_i}/\Sigma_{\la_i} $, the set of all collections of multisets of sizes $ (\lambda_i)_{i \in I} $.  A point in $ \C^\lambda $ will be written as $\bR = (R_i)_{i \in I} $ where $ R_i $ is a multiset of size $ \lambda_i $ and it will be called a \textbf{set of parameters} of weight $ \lambda $.  Thus, we say that $i \in I $ and $k \in \Z$ have the {\bf
  same parity} if $i\in I_{\bar{k}}$.

We say that $ \bR $ is \textbf{integral}, if for all $ i$, all elements of $R_i$ are integers and have
the same parity as $i$.   So an integral set of parameters consists of a multiset of $ \lambda_i $ even integers for every $ i \in I_0 $ and $ \lambda_i $ odd integers for every $ i \in I_1 $.

\subsection{Definition of the monomial crystal}
\label{subsection:Definitions}
Recall that a {\bf crystal} for $ \g $ is a set $ \B $, along with a partial inverse permutations $ \tilde{e}_i, \tilde{f}_i : \B \rightarrow \B $, for all $ i \in I $, and a weight map $ wt : \B \rightarrow \Lambda $.  For each dominant weight $ \lambda $, there is a crystal $ \B(\lambda) $ corresponding to the irreducible representation $ V(\lambda) $.  We say that a crystal $ \B $ is {\bf normal} if it is the disjoint union of these crystals $ \B(\lambda) $ (for varying $ \lambda $).

The most basic operation on crystals is the tensor product of crystals.  We must be careful about this definition since different papers in the literature use different conventions.  

\begin{Definition}
Given crystals $\mathcal{B}_1,\mathcal{B}_2$, their tensor product is the Cartesian product of the underlying sets, with the weight function given by the sum, and 
\begin{align*}
\tilde{e}_i(b_1\otimes b_2)&=\begin{cases} \tilde{e}_ib_1\otimes b_2 & \varphi_i(b_1)\geq \epsilon_i(b_2)\\
b_1\otimes \tilde{e}_ib_2 & \varphi_i(b_1)< \epsilon_i(b_2)\end{cases}\\
\tilde{f}_i(b_1\otimes b_2)&=\begin{cases} \tilde{f}_ib_1\otimes b_2 & \varphi_i(b_1)> \epsilon_i(b_2)\\
b_1\otimes \tilde{f}_ib_2 & \varphi_i(b_1)\leq \epsilon_i(b_2)
\end{cases}
\end{align*}
Here $\varphi_i(b)=max\{ n\geq 0\; |\; \tilde{f}^n_ib \neq 0 \}$, and $\epsilon_i(b)=max\{ n\geq 0\; |\; \tilde{e}^n_ib \neq 0 \}$.
\end{Definition}
Note this is the {\it opposite} of the definition in \cite[\S 7]{LoWe};  thus, in these conventions, \cite[Th. 7.2]{LoWe} proves that the natural crystal structure on the simples of a tensor product categorification for $\bla=(\la^{(1)},\dots, \la^{(\ell)})$ is naturally isomorphic to $\B(\la^{(\ell)})\otimes \cdots \otimes \B(\la^{(1)}).$

Let $ \M $ denote the set of all monomials in the variables $ \yMon_{i,k}
$, for $k\in \Z$, such that $i, k $ have the same parity.
Let
$$
\zMon_{i,k} = \frac{\yMon_{i,k} \yMon_{i,k+2}}{\displaystyle\prod_{j  \con i} \yMon_{j,k+1}}
$$
For a variable $ \yMon_{i,k}$ or $ \zMon_{i,k}$ the second index $k$ is called the {\bf longitude} of the variable.

Given a monomial $ p = \prod_{i,k} \yMon_{i,k}^{d_{i,k}} $, let
\begin{equation*} wt(p) = \sum_{i,k} d_{i,k} \varpi_{i} \quad
\varepsilon_{i}^m(p) = - \sum_{l \le m} d_{i,l} \quad
\varphi_{i}^m(p) = \sum_{l \ge m} d_{i,l}
\end{equation*}
and
\begin{align*}
\varepsilon_{i}(p) = \max_m \varepsilon_{i}^m(p) \quad
  \varphi_{i}(p) = \max_m \varphi_{i}^m(p)
\end{align*}

We can define the Kashiwara operators on this set of monomials by the rules:
\begin{align*}
\tilde{e}_i(p) &= \begin{cases} 0 &\text{if $\varepsilon_i(p) = 0 $} \\
\zMon_{i,m}p& \parbox{4.5in}{for $ m $ minimal such that $ \varepsilon_{i}^m(p) =\varepsilon_{i}(p)>0$}
\end{cases} \\
\tilde{f}_i(p) &= \begin{cases} 0 & \text{if $\varphi_i(p) = 0 $} \\
\zMon_{i,m-2}^{-1}p
&\parbox{4.5in}{for $ m$ maximal such that $ \varphi_{i}^m(p) =\varphi_{i}(p)>0$}
\end{cases}
\end{align*}
The following result is due to Kashiwara \cite[Proposition 3.1]{Kash}.

\begin{Theorem}
\label{Thm: B normal crystal}
$\M $ is a normal crystal.
\end{Theorem}

\subsection{Product monomial crystals}
\label{Tsection}
For any $ c \in \Z $ and $ i \in I $ of the same parity, the monomial $ \yMon_{i,c} $ is clearly highest weight and we can consider the monomial subcrystal $ \B(\varpi_i,c) $ generated by $ \yMon_{i,c} $.  Since $ \yMon_{i,c} $ has weight $ \varpi_i $, we see that $ \B(\varpi_i,c) \cong \B(\varpi_{i}) $.  The fundamental monomial crystals for different $ c $ all look the same, they differ simply by translating the variables.


Given a dominant weight $\lambda $ and an integral set of parameters $ \bR $ of weight $\lambda $ as above, following \cite{KTWWY}, we define the product monomial crystal $ \B(\bR) $ by
$$
\B(\bR) = \prod_{i \in I, c \in R_i} \B(\varpi_i, c)
$$
In other words, for each parameter $ c \in R_i$, we form its monomial crystal $ \B(\varpi_i, c) $ and then take the product of all monomials appearing in all these crystals.  In \cite{KTWWY}, we proved the
 following result  as a consequence of the link between $ \B(\bR)$ and  graded quiver varieties.  It also follows as an immediate corollary of the link between the monomial crystal and the parity KLRW algebra, see Theorem \ref{thm:highest-bijection} of the current paper.
\begin{Theorem} \label{th:subcrystal}
$\B(\bR) $ is a subcrystal of $ \M $.  In particular it is a normal crystal.  Moreover, there exists embeddings $ \B(\lambda) \subseteq \B(\bR) \subseteq \otimes_i \B(\varpi_i)^{\otimes \lambda_i} $.
\end{Theorem}

Thus $ \B(\bR) $ is a crystal which depends on the set of parameters $ \bR$ and lies between the crystal of the irreducible representation and the crystal of the corresponding tensor product of fundamental representations (for an example see Section 2.6 in \cite{KTWWY}).

\subsection{Collections of multisets and monomials}
\label{subsection: Collections of multisets and monomials}
Given a collection of multisets $ \bS = ( S_i )_{i \in I} $, we can define
$$ \yMon_\bS = \prod_{i \in I, k \in S_i} \yMon_{i,k}, \quad \zMon_\bS = \prod_{i \in I, k \in S_i} \zMon_{i,k}.$$
Here we still require that the elements of $S_i$ are integers of the same parity as $i$, but we don't put any restriction on the cardinalities of the $S_i$ so this is not necessarily a set of parameters of weight $\lambda$.  

From the definition of the monomial crystal, it is easy to see that every monomial $ p $ in $ \B(\bR) $ is of the form
\begin{equation} \label{eq:CSmonomial} p = \yMon_{\bR} \zMon_\bS^{-1} = \prod_{i,k \in R_i} \yMon_{i,k} \prod_{i,k \in S_i} \frac{\prod_{ j \con i} \yMon_{j,k+1}}{\yMon_{i,k} \yMon_{i,k+2}}
\end{equation}
 for some collection of multisets $ \bS $.  Thus an alternative combinatorics for labeling elements of the monomial crystal are these collections of multisets $ \bS $.

\begin{Remark}
\label{Remark: uniqueness of monomial factorization}
For $p\in \B(\bR)$, $\bS$ is uniquely determined.  In fact for any tuples of multisets $\bS$ and  $\bS'$, $\zMon_\bS = \zMon_{\bS'}$ implies $\bS = \bS'$ .
\end{Remark}


\section{Variations on KLRW algebras}
\subsection{Recollection on KLRW algebras}
\label{sec:KLRW}

Consider a sequence $\blam=(\lambda^{(1)},...,\lambda^{(\ell)})$ of dominant integral weights.
We recall now the third author's construction of the \textbf{KLRW algebra} $\tilde{T}^\blam$ (alias tensor product algebra, alias red-black algebra, alias Webster algebra).

The algebra $\tilde{T}^\blam$ is defined as a span of Stendhal diagrams, modulo some local relations.
More precisely, a Stendhal diagram \cite[Definition 4.1]{Webmerged} defining an element of $\tilde{T}^\blam$ is a collection of finitely many curves in $\R\times [0,1]$, where each curve is either red and labeled by one of the $\lambda^{(j)}$, or is black and labeled by $i\in I$.  Each curve has one endpoint on $ \R \times \{0\} $ and one on $ \R \times \{1\} $. The black curves can also be decorated with finitely many dots.  The diagrams are considered up to isotopy, and must be locally of the form
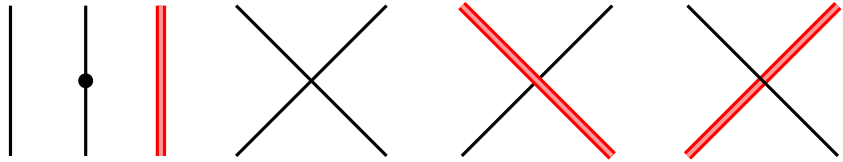
\begin{equation*}
\begin{tikzpicture}
  \draw[very thick] (0,0) +(-1,-1) -- +(1,1);
  \draw[very thick](0,0) +(1,-1) -- +(-1,1);


  \draw[very thick](3,0) +(-1,-1) -- +(1,1);
  \draw[wei, very thick](3,0) +(1,-1) -- +(-1,1);

  \draw[wei,very thick](6,0) +(-1,-1) -- +(1,1);
  \draw [very thick](6,0) +(1,-1) -- +(-1,1);

 \draw[very thick](-4,0) +(0,-1) --  +(0,1);

  \draw[very thick](-3,0) +(0,-1) --  node
  [midway,circle,fill=black,inner sep=2pt]{}
  +(0,1);

   \draw[wei,very thick](-2,0) +(0,-1) --  +(0,1);
\end{tikzpicture}
\end{equation*}
Note that no red strands can ever cross.  In \cite{Webmerged} the
diagrams are oriented.  We  only consider downward oriented strands,
so we omit the orientation here.  The diagrams in $\tilde{T}^\blam$ each
have $\ell$ red strands, labeled by
$\lambda^{(1)},...,\lambda^{(\ell)}$ from left to right.  For
convenience, we let $y_k$ denote the dot on the $k$th black strand read from
left to right.  The diagrams satisfy the following relations:
 \begin{itemize}
  \item the KLR relations (\ref{first-QH}--\ref{triple-smart})  for
    $\Gamma$ with \[X_{ij}(u,v)= \begin{cases}
  1 & i \not \leftarrow j\\
 u-v & i\leftarrow j\\
\end{cases} \qquad Q_{ij}(u,v)=X_{ij}(u,v)X_{ji}(v,u)=
\begin{cases}
  1 & i \not \leftrightarrow j\\
 u-v & i\leftarrow j\\
v-u & i\to j
\end{cases}
\] \newseq
\begin{equation*}\subeqn\label{first-QH}
    \begin{tikzpicture}[scale=.9,baseline]
      \draw[very thick](-4,0) +(-1,-1) -- +(1,1) node[below,at start]
      {$i$}; \draw[very thick](-4,0) +(1,-1) -- +(-1,1) node[below,at
      start] {$j$}; \fill (-4.5,.5) circle (3pt);
      \node at (-2,0){=}; \draw[very thick](0,0) +(-1,-1) -- +(1,1)
      node[below,at start] {$i$}; \draw[very thick](0,0) +(1,-1) --
      +(-1,1) node[below,at start] {$j$}; \fill (.5,-.5) circle (3pt);
      \node at (4,0){unless $i=j$};
    \end{tikzpicture}
  \end{equation*}
\begin{equation*}\subeqn\label{second-QH}
    \begin{tikzpicture}[scale=.9,baseline]
      \draw[very thick](-4,0) +(-1,-1) -- +(1,1) node[below,at start]
      {$i$}; \draw[very thick](-4,0) +(1,-1) -- +(-1,1) node[below,at
      start] {$j$}; \fill (-3.5,.5) circle (3pt);
      \node at (-2,0){=}; \draw[very thick](0,0) +(-1,-1) -- +(1,1)
      node[below,at start] {$i$}; \draw[very thick](0,0) +(1,-1) --
      +(-1,1) node[below,at start] {$j$}; \fill (-.5,-.5) circle (3pt);
      \node at (4,0){unless $i=j$};
    \end{tikzpicture}
  \end{equation*}
\begin{equation*}\subeqn\label{nilHecke-1}
    \begin{tikzpicture}[scale=.9,baseline]
      \draw[very thick](-4,0) +(-1,-1) -- +(1,1) node[below,at start]
      {$i$}; \draw[very thick](-4,0) +(1,-1) -- +(-1,1) node[below,at
      start] {$i$}; \fill (-4.5,.5) circle (3pt);
      \node at (-2,0){=}; \draw[very thick](0,0) +(-1,-1) -- +(1,1)
      node[below,at start] {$i$}; \draw[very thick](0,0) +(1,-1) --
      +(-1,1) node[below,at start] {$i$}; \fill (.5,-.5) circle (3pt);
      \node at (2,0){$+$}; \draw[very thick](4,0) +(-1,-1) -- +(-1,1)
      node[below,at start] {$i$}; \draw[very thick](4,0) +(0,-1) --
      +(0,1) node[below,at start] {$i$};
    \end{tikzpicture}
  \end{equation*}
 \begin{equation*}\subeqn\label{nilHecke-2}
    \begin{tikzpicture}[scale=.9,baseline]
      \draw[very thick](-4,0) +(-1,-1) -- +(1,1) node[below,at start]
      {$i$}; \draw[very thick](-4,0) +(1,-1) -- +(-1,1) node[below,at
      start] {$i$}; \fill (-4.5,-.5) circle (3pt);
      \node at (-2,0){=}; \draw[very thick](0,0) +(-1,-1) -- +(1,1)
      node[below,at start] {$i$}; \draw[very thick](0,0) +(1,-1) --
      +(-1,1) node[below,at start] {$i$}; \fill (.5,.5) circle (3pt);
      \node at (2,0){$+$}; \draw[very thick](4,0) +(-1,-1) -- +(-1,1)
      node[below,at start] {$i$}; \draw[very thick](4,0) +(0,-1) --
      +(0,1) node[below,at start] {$i$};
    \end{tikzpicture}
  \end{equation*}
  \begin{equation*}\subeqn\label{black-bigon}
    \begin{tikzpicture}[very thick,scale=.9,baseline]
      \draw (-2.8,0) +(0,-1) .. controls (-1.2,0) ..  +(0,1)
      node[below,at start]{$i$}; \draw (-1.2,0) +(0,-1) .. controls
      (-2.8,0) ..  +(0,1) node[below,at start]{$i$}; \node at (-.5,0)
      {=}; \node at (0.4,0) {$0$};
\node at (1.5,.05) {and};
    \end{tikzpicture}
\hspace{.4cm}
    \begin{tikzpicture}[very thick,scale=.9,baseline]

      \draw (-2.8,0) +(0,-1) .. controls (-1.2,0) ..  +(0,1)
      node[below,at start]{$i$}; \draw (-1.2,0) +(0,-1) .. controls
      (-2.8,0) ..  +(0,1) node[below,at start]{$j$}; \node at (-.5,0)
      {=};
\draw (1.8,0) +(0,-1) -- +(0,1) node[below,at start]{$j$};
      \draw (1,0) +(0,-1) -- +(0,1) node[below,at start]{$i$};
\node[inner xsep=10pt,fill=white,draw,inner ysep=8pt] at (1.4,0) {$Q_{ij}(y_1,y_2)$};
    \end{tikzpicture}
  \end{equation*}
 \begin{equation*}\subeqn\label{triple-dumb}
    \begin{tikzpicture}[very thick,scale=.9,baseline]
      \draw (-3,0) +(1,-1) -- +(-1,1) node[below,at start]{$k$}; \draw
      (-3,0) +(-1,-1) -- +(1,1) node[below,at start]{$i$}; \draw
      (-3,0) +(0,-1) .. controls (-4,0) ..  +(0,1) node[below,at
      start]{$j$}; \node at (-1,0) {=}; \draw (1,0) +(1,-1) -- +(-1,1)
      node[below,at start]{$k$}; \draw (1,0) +(-1,-1) -- +(1,1)
      node[below,at start]{$i$}; \draw (1,0) +(0,-1) .. controls
      (2,0) ..  +(0,1) node[below,at start]{$j$}; \node at (5,0)
      {unless $i=k\neq j$};
    \end{tikzpicture}
  \end{equation*}
\begin{equation*}\subeqn\label{triple-smart}
    \begin{tikzpicture}[very thick,scale=.9,baseline]
      \draw (-3,0) +(1,-1) -- +(-1,1) node[below,at start]{$i$}; \draw
      (-3,0) +(-1,-1) -- +(1,1) node[below,at start]{$i$}; \draw
      (-3,0) +(0,-1) .. controls (-4,0) ..  +(0,1) node[below,at
      start]{$j$}; \node at (-1,0) {=}; \draw (1,0) +(1,-1) -- +(-1,1)
      node[below,at start]{$i$}; \draw (1,0) +(-1,-1) -- +(1,1)
      node[below,at start]{$i$}; \draw (1,0) +(0,-1) .. controls
      (2,0) ..  +(0,1) node[below,at start]{$j$}; \node at (2.8,0)
      {$+$};        \draw (6.2,0)
      +(1,-1) -- +(1,1) node[below,at start]{$i$}; \draw (6.2,0)
      +(-1,-1) -- +(-1,1) node[below,at start]{$i$}; \draw (6.2,0)
      +(0,-1) -- +(0,1) node[below,at start]{$j$};
\node[inner ysep=8pt,inner xsep=5pt,fill=white,draw,scale=.8] at (6.2,0){$\displaystyle \frac{Q_{ij}(y_3,y_2)-Q_{ij}(y_1,y_2)}{y_3-y_1}$};
    \end{tikzpicture}
  \end{equation*}
\item  All black crossings and dots can pass through red lines.  For the latter two
  relations (\ref{dumb}--\ref{red-dot}), we also include their mirror images:
\newseq
  \begin{equation*}\subeqn
    \begin{tikzpicture}[very thick,baseline]\label{red-triple-correction}
      \draw (-3,0)  +(1,-1) -- +(-1,1) node[at start,below]{$i$};
      \draw (-3,0) +(-1,-1) -- +(1,1)node [at start,below]{$j$};
      \draw[wei] (-3,0)  +(0,-1) .. controls (-4,0) .. node[below, at start]{$\lambda$}  +(0,1);
      \node at (-1,0) {=};
      \draw (1,0)  +(1,-1) -- +(-1,1) node[at start,below]{$i$};
      \draw (1,0) +(-1,-1) -- +(1,1) node [at start,below]{$j$};
      \draw[wei] (1,0) +(0,-1) .. controls (2,0) ..  node[below, at start]{$\lambda$} +(0,1);
\node at (2.6,0) {$+ $};
      \draw (6,0)  +(1,-1) -- node
  [midway,circle,fill=black,inner sep=2pt]{} +(1,1) {} node[at start,below]{$i$};
      \draw (6,0) +(-1,-1) -- node
  [midway,circle,fill=black,inner sep=2pt]{} +(-1,1) {} node [at start,below]{$j$};
      \draw[wei] (6,0) +(0,-1) -- node[below, at start]{$\lambda$} +(0,1);
\node at (3.8,-.2){$\displaystyle \delta_{i,j}\sum_{a+b+1=\lambda_i}$} ;
\node at (5.3,0){$b$};
\node at (7.3,0){$a$};
 \end{tikzpicture}
  \end{equation*}
\begin{equation*}\subeqn\label{dumb}
    \begin{tikzpicture}[very thick,baseline=2.85cm]
      \draw[wei] (-3,3)  +(1,-1) -- +(-1,1);
      \draw (-3,3)  +(0,-1) .. controls (-4,3) ..  +(0,1);
      \draw (-3,3) +(-1,-1) -- +(1,1);
      \node at (-1,3) {=};
      \draw[wei] (1,3)  +(1,-1) -- +(-1,1);
  \draw (1,3)  +(0,-1) .. controls (2,3) ..  +(0,1);
      \draw (1,3) +(-1,-1) -- +(1,1);    \end{tikzpicture}
  \end{equation*}
\begin{equation*}\subeqn\label{red-dot}
    \begin{tikzpicture}[very thick,baseline]
  \draw(-3,0) +(-1,-1) -- +(1,1);
  \draw[wei](-3,0) +(1,-1) -- +(-1,1);
\fill (-3.5,-.5) circle (3pt);
\node at (-1,0) {=};
 \draw(1,0) +(-1,-1) -- +(1,1);
  \draw[wei](1,0) +(1,-1) -- +(-1,1);
\fill (1.5,.5) circle (3pt);
    \end{tikzpicture}
  \end{equation*}
\item  The ``cost'' of separating a red $\lambda$-strand and a black $i$-strand is adding $\lambda_i$ dots to the black strand.
  \begin{equation}\label{cost}
  \begin{tikzpicture}[very thick,baseline=1.6cm]
    \draw (-2.8,0)  +(0,-1) .. controls (-1.2,0) ..  +(0,1) node[below,at start]{$i$};
       \draw[wei] (-1.2,0)  +(0,-1) .. controls (-2.8,0) ..  +(0,1) node[below,at start]{$\lambda$};
           \node at (-.3,0) {=};
    \draw[wei] (2.8,0)  +(0,-1) -- +(0,1) node[below,at start]{$\lambda$};
       \draw (1.2,0)  +(0,-1) -- +(0,1) node[below,at start]{$i$};
       \fill (1.2,0) circle (3pt) node[left=3pt]{$\lambda_i$};
          \draw[wei] (-2.8,3)  +(0,-1) .. controls (-1.2,3) ..  +(0,1) node[below,at start]{$\lambda$};
  \draw (-1.2,3)  +(0,-1) .. controls (-2.8,3) ..  +(0,1) node[below,at start]{$i$};
           \node at (-.3,3) {=};
    \draw (2.8,3)  +(0,-1) -- +(0,1) node[below,at start]{$i$};
       \draw[wei] (1.2,3)  +(0,-1) -- +(0,1) node[below,at start]{$\lambda$};
       \fill (2.8,3) circle (3pt) node[right=3pt]{$\lambda_i$};
  \end{tikzpicture}
\end{equation}
  \end{itemize}

The algebra $\tilde{T}^\bla$ is graded with degree given by
\cite[Def. 4.4]{Webmerged}:
\[
  \deg\tikz[baseline,very thick,scale=1.5]{\draw[-] (.2,.3) --
    (-.2,-.1) node[at end,below, scale=.8]{$i$}; \draw[wei] (.2,-.1) --
    (-.2,.3) node[at start,below,scale=.8]{$\la$};} =\langle\al_i,\la\rangle\qquad \deg\tikz[baseline,very
  thick,-,scale=1.5]{\draw (0,.3) -- (0,-.1) node[at
    end,below,scale=.8]{$i$} node[midway,circle,fill=black,inner
    sep=2pt]{};}=-2
 \qquad  \deg\tikz[baseline,very thick,scale=1.5]{\draw[-] (.2,.3) --
    (-.2,-.1) node[at end,below, scale=.8]{$i$}; \draw[-] (.2,-.1) --
    (-.2,.3) node[at start,below,scale=.8]{$j$};}
  =-\langle\al_i,\al_j\rangle  \]

The identity of $\tilde{T}^\blam$ is a sum of
idempotents corresponding to sequences of red and black strands.  More carefully, we associate an
idempotent $e(\Bi,\kappa)$ to a sequence $\Bi=(i_1,\dots, i_n)$, which labels the black
strands in order, and a weakly increasing function $\kappa\colon [1,\ell]\to [0,n]$
such that the $k$th red strand is between the $\kappa(k)$th and
$(\kappa(k)+1)$-st black strands.  We interpret $\kappa(k)=0$ to mean
this red strand is left of all black strands, and $\kappa(k)=n$ to mean it is
right of all black strands.  Recall that in $\tilde{T}^\blam$ the red strands are always labeled $\lambda^{(1)},...,\lambda^{(\ell)}$ from left to right, so the pair $(\Bi,\kappa)$ uniquely determines the idempotent.


As in \cite[Lemma 4.12]{Webmerged}, $\tilde{T}^\blam$ acts on
\begin{equation*}
\label{eq:Pol}
\PolKLR=
\bigoplus_{(\Bi,\kappa)} \PolKLR(\Bi,\kappa), \quad \text{where}\quad \PolKLR(\Bi,\kappa)=\C[Y_1(\Bi,\kappa),\dots Y_n(\Bi,\kappa)]
\end{equation*}
The idempotent
 $e(\Bi,\kappa)$ acts by projection onto $\PolKLR(\Bi,\kappa)$, and \newseq
  \begin{equation*}\label{eq:T-action1}\subeqn
\begin{tikzpicture}[scale=.4,baseline]
\draw[wei] (-1,-1) -- (1,1) node[at start,below]{
$\lambda$} node[at end,above]{
$\lambda$};
\draw[very thick] (1,-1) -- (-1,1) node[at start,below]{
$i$} node[at end,above]{
$i$};
\node at (3.7,0) {$\bullet\: f=Y_k^{\lambda_i}\cdot  f$ } ;
\end{tikzpicture}\qquad \qquad
\begin{tikzpicture}[scale=.4,baseline]
\draw[wei] (1,-1) -- (-1,1) node[at start,below]{
$\lambda$} node[at end,above]{
$\lambda$};
\draw[very thick] (-1,-1) -- (1,1) node[at start,below]{
$i$} node[at end,above]{
$i$};
\node at (2.7,0) {$\bullet\: f=f$ } ;
\end{tikzpicture}\qquad \qquad
\begin{tikzpicture}[scale=.4,baseline]
\draw[very thick] (1,-1) -- (1,1) node[at start,below]{
$i$} node[at end,above]{
$i$} node[circle,midway,fill,inner sep=2pt]{};
\node at (3.8,0) {$\bullet\: f= Y_k\cdot f$ } ;
\end{tikzpicture}
  \end{equation*}
\begin{equation*}\label{eq:T-action2}\subeqn
  \begin{tikzpicture}[scale=.4]
\draw[very thick] (1,-1) -- (-1,1) node[at start,below]{
$j$} node[at end,above]{
$j$};
\draw[very thick] (-1,-1) -- (1,1) node[at start,below]{
$i$} node[at end,above]{
$i$};
\node at (7.5,0) {$\bullet\: f=\begin{cases}
 X_{ij}(Y_{k+1},Y_{k}) f^{s_k} & i\neq j\\
 \displaystyle \frac{f^{s_k}- f}{Y_{k+1}-Y_{k}} & i=j
\end{cases}$ } ;
\end{tikzpicture}
\end{equation*}
In the equations above, the black $i$-strand is always the $k$th black strand from left to right. 

\subsection{Some statements about categorification} Set $\lambda=\sum_j\lambda^{(j)}$.  For $\mu \in \Lambda$, let $\tilde{T}^\blam_{\mu}$ be the subalgebra of $\tilde{T}^\blam$ where
the sum of the roots associated to the black strands is $\lambda-\mu$.
We have a non-unital algebra homomorphism $\iota_{\mu}^+:\tilde{T}^\blam_{\mu}
\to \tilde{T}^\blam_{\mu-\alpha_i}$, which adds a black $i$-strand to the
right of the diagram, and its mirror image $\iota^-_\mu$ adding a
strand at the left.

Let $\tilde{T}^\blam_\mu\mathrm{-mod}$ be the category of finitely generated $\tilde{T}^\blam_\mu$-modules.  Consider the functors
\begin{align*}
\cF_i^+: \tilde{T}^\blam_\mu\mathrm{-mod} \to \tilde{T}^\blam_{\mu-\alpha_i}\mathrm{-mod} \\
\cF_i^-: \tilde{T}^\blam_\mu\mathrm{-mod} \to \tilde{T}^\blam_{\mu-\alpha_i}\mathrm{-mod}
\end{align*}
given by induction along the maps $\iota_{\mu}^\pm$. More specifically, $\cF_i^\pm(M)=M\otimes_{\tilde{T}^\blam_\mu}\tilde{T}^\blam_{\mu-\alpha_i}$.  Let $\cE_i^\pm$ be the adjoint functors of $\cF_i^\pm$ given by restriction along the
maps $\iota_{\mu}^\pm$.  Note that setting $e_i^+$ (respectively $e_i^-$) to be the sum of idempotents whose rightmost (respectively leftmost) strand is a black $i$-strand, then $\cE_i^+(M)=e_i^+M$ and $\cE_i^-(M)=e_i^-M$.

Let $K^0(\tilde{T}^\blam)$ denote the Grothendieck group of finitely generated $\tilde{T}^\blam$-modules, and set $K^\C(-):=K^0(-)\otimes_{\Z}\C$.  Let $V(\blam)=V({\lambda^{(1)}})\otimes \cdots \otimes V({\lambda^{(\ell)}})$.
By \cite[Prop. 2.18]{Webweb}, we have:
\begin{Proposition}
\label{prop:tildeTaction}
We have an isomorphism of $(U(\mathfrak{n}_-),U(\mathfrak{n}_-))$-bimodules
$$
\mathcal{G}\colon U(\blam):=
U(\mathfrak{n}_-) \otimes V({\blam}) \to K^\C(\tilde{T}^\blam)
$$
which maps $K^\C(\tilde{T}^\blam_\mu)$ to the $\mu$-weight space of
$U(\blam)$.  Under this isomorphism, the functors
$\cF_i^+,\cF_i^-$ categorify the actions:
\begin{align*}
  f_i^+\cdot (u\otimes v_1\otimes \cdots \otimes v_\ell )
  & = f_iu\otimes  v_1\otimes \cdots
    \otimes v_\ell \\
&\qquad  +u\otimes f_iv_1\otimes \cdots \otimes v_\ell+\cdots  +u\otimes v_1\otimes
    \cdots\otimes  f_iv_\ell  \\
 (u\otimes v_1\otimes \cdots \otimes v_\ell)  \cdot f_i^-
  & = uf_i \otimes v_1\otimes \cdots
    v_\ell
\end{align*}
\end{Proposition}

This isomorphism is uniquely fixed by the action property above, and
the fact that the functor $\mathcal{I}_\la$ induced by adding a red
strand at the far right (i.e. like $\cF_i^+$, but with a different
color) is induced by \[I_{\la}(u\otimes v_1\otimes \cdots \otimes
v_{\ell})\mapsto u\otimes v_1\otimes \cdots \otimes
v_{\ell}\otimes v_{\la,\operatorname{high}}\] where
$v_{\la,\operatorname{high}}$ is the highest weight vector of
$V(\la)$.  We let $\tilde{p}^\kappa_{\Bi}=[\tilde{T}e(\Bi,\kappa)]$, following the notation of \cite[\S 4.7]{Webmerged}; this is easily calculated from the rules above. 
\begin{Remark}
  The association of $\cF_i^+$ with the left action and $\cF^-_i$ with
  the right here might seem a little peculiar, since they involve
  adding strands on the right and left sides of the diagram.  This is
  due to the dyslexia of the conventions in \cite{Webmerged};
  unfortunately, this is a sin of one of authors for which all of
  us, readers included, must do penance.
\end{Remark}

The algebra $\tilde{T}^\blam$ has a natural triple of quotients; we
let ${}_+T^\blam$ be the quotient by the two-sided ideal generated by
all idempotents where $\kappa(1)> 0$ (i.e. a black strand is to the left of all red strands), ${}_-T^\blam$ be the
quotient by the two-sided ideal generated by all idempotents where
$\kappa(\ell)< n$ (i.e. a black strand is to the right of all red strands), and ${}_0T^\blam$ the quotient by the sum of these
ideals.

Note that $\cF_i^+$ descends to a well-defined functor on ${}_+T^\blam\mmod$
and $\cF_i^-$ to one on ${}_-T^\blam\mmod$, still both defined by extension
of scalars along an inclusion of algebras.  Since the quotients
${}_+T^\blam$ and ${}_-T^\blam$ are finite dimensional, the adjoint functors $\cE_i^+$ and
$\cE_i^-$ preserve finitely generated modules, and less obviously, they
preserve the category of projective modules (this follows from the
Morita equivalence of \cite[Thm. 4.30]{Webmerged} by allowing us to
rewrite it as induction to a Morita equivalent algebra).

   These quotient algebras correspond to killing the action of
the augmentation ideal of $ U(\mathfrak{n}_-)$ for the left action,
right action, or both.

\begin{Proposition}\label{prop:Taction}
  We have isomorphisms of Grothendieck groups
\newseq
  \begin{align*}
K^\C({}_+T^\blam)&\cong U(\blam)/\mathfrak{n}_-U(\blam)\cong V({\blam})\subeqn\label{eq:tensor1}\\
K^\C({}_-T^\blam) &\cong U(\blam)/U(\blam)\mathfrak{n}_-\cong V({\blam'})\subeqn\label{eq:tensor2}\\
K^\C({}_0T^\blam) &\cong U(\blam)/(U(\blam)\mathfrak{n}_-+\mathfrak{n}_-U(\blam)) \subeqn\label{eq:tensor3} \cong V({\blam})/ \mathfrak{n}_-V({\blam}) \cong (V({\blam}) )^{\mathfrak{n}}
  \end{align*}
  In (\ref{eq:tensor2}), $\blam'$ denotes the list of dominant weight in reverse order $(\lambda^{(\ell)},...,\lambda^{(1)})$.
  The action of the functors $\cE_i^\pm,\cF_i^\pm$ on ${}_\pm T\mmod$
  categorify the action of the Chevalley generators $e_i, f_i$ on
  $V(\blam)$.
\end{Proposition}
 
\begin{proof}
  The isomorphism (\ref{eq:tensor1}) is a consequence of
  \cite[Th. B]{Webmerged}.  Note that the relations
  (\ref{first-QH}--\ref{cost}) are preserved by reflecting in a
  vertical axis and negating crossings $\psi_i$ of strands with the
  same labels. This gives an isomorphism of ${}_-T^\blam$ to ${}_+T^{w_0\cdot\blam}$ with
  the order of labels on red strands reversed.  Thus
  (\ref{eq:tensor2}) follows.

   The isomorphism (\ref{eq:tensor3}) is not explicitly discussed in
   \cite{Webmerged}, but is easily derived from the definitions there:
   the modules over ${}_0T^\blam$ are the same as those over ${}_+T^\blam$ killed
   by all $\cE_i^+$ for all $i$, and the number of simple modules
   killed by these functors is the same as the number of highest
   weight elements of the corresponding tensor product crystal by
   \cite[Th. 7.2]{LoWe}, so the classes of these simples span the
   space of highest weight vectors.  The description of the highest
   weight vectors as the quotient $V/\mathfrak{n}_-V$  is
   valid for any finite dimensional $\fg$-module $V$.
\end{proof}

Let $\tU$ be the 2-Kac-Moody algebra categorifiying $U(\fg)$, as defined by Khovanov-Lauda \cite{KLIII} and Rouquier \cite{Rou2KM}.  For details see Definition 2.4 in \cite{Webmerged}.  Note that in the
definition we use the same matrix of polynomials $Q_{i,j}(u,v)$ which
appears above.
\begin{Proposition}[\mbox{\cite[Th. B]{Webmerged}}]\label{prop:cat-action}
The functors $\cE_i^{\pm},\cF_i^{\pm}$ induce an action of $\tU$ on
the categories  ${}_\pm T^\blam\mathrm{-mod}$, which each categorify $V({\blam})$ and $V({\blam'})$.
\end{Proposition}If we instead consider the graded Grothendieck group of $\tilde{T},{}_\pm
  T$, we will arrive at the obvious quantum analogues of the objects
  under consideration.  
  This is significant because  these quantum analogues allow us to
  define a canonical basis of the spaces $U(\bla)$ and $V(\bla)$, and
  when the conventions are chosen correctly, the isomorphisms of
  Propositions \ref{prop:tildeTaction} and \ref{prop:Taction} send
  the classes of indecomposable projectives bijectively to canonical
  basis by the argument of \cite[Th. 8.7]{WebCB}; note that the statement of {\it loc.\ cit.} only covers $V(\bla)$, but that is derived as a consequence of the fact that it holds for $\tilde{T}^{\bla}$ and $U(\bla)$.  
  
  We can convert this into a more concrete statement about multiplicities with a little algebra:
  \begin{Lemma}\label{lem:idempotent-canonical} The dimension  $\mathsf{m}_{\Bi,\kappa,b}=\dim e(\Bi, \kappa)L_b$ of the image of the idempotent $e(\Bi,\kappa)$ in the unique simple quotient of the indecomposable projective $P_b$ such that $[P_b]=b$ for $b$ a canonical basis vector is the coefficient of $b$ when $[\tilde{T}e(\Bi, \kappa)]$ is expanded in the canonical basis.  That is:
  \[p_{\Bi}^{\kappa}=[\tilde{T}e(\Bi, \kappa)]=\sum_{b} \mathsf{m}_{\Bi,\kappa,b}\cdot b.\]
  \end{Lemma}
  \begin{proof}
  Of course, $\mathsf{m}_{\Bi,\kappa,b}=\dim \Hom(\tilde{T}e(\Bi, \kappa),L)$.  Since $\dim \Hom(P_{b'},L_b)=\delta_{b,b'}$, the result follows.
  \end{proof}
  
  Unfortunately, some real care is
  needed here about conventions.  It might seem strange that in (\ref{eq:tensor1}) and (\ref{eq:tensor2}) we used opposite orders of the tensor product, since representations of $\fg$ are a symmetric tensor category and thus $V(\bla)$ and $V(\bla')$ are canonically isomorphic.  However, this is no longer true after quantization.  
  
  This manifests on the level of categorification: reflecting in a vertical axis (and inserting a few signs) instead gives an isomorphism ${}_+T^{\bla}\cong {}_-T^{\bla'}$.  
In fact, by \cite[Th. 3.12]{LoWe}, this implies that ${}_+T^{\bla}\mmod$ is a tensor product categorification (in the sense of {\it loc.\ cit.}) for $V(\bla)$ while ${}_-T^{\bla}\mmod$ is for $V(\bla')$.  

\subsection{A KLRW algebra associated to $\bR$}
\label{section:TR}

 Fix a dominant integral weight $ \lambda$ and write $ \lambda = \sum_i
\lambda_i \varpi_i $. We fix also an integral set of parameters $\bR = \{   R_i \}_{i\in I}$ of weight $\lambda$.  It will be convenient at times to encode the parameters using multiplicity functions $\rho_i:\Z\to\N$ defined by
$$
\rho_i(q)=
\begin{cases}
\text{multiplicity of $2q$ in $R_i$, if $i$ is even},\\
\text{multiplicity of $2q+1$ in $R_i$, if $i$ is odd.}
\end{cases}
$$
Recalling that we've fixed a total order on $I$ (cf. Section \ref{subsection:Notation}), let $I=\{i_1<\cdots< i_r\}$.  Define the following sequences of fundamental weights:
\begin{equation}
\label{eq:varpiq}
\varpi_\bR(q):=(\varpi_{i_1}^{\rho_{i_1}(q)},...,\varpi_{i_r}^{\rho_{i_r}(q)})
\end{equation}
where  $\varpi_{i}^{k}$ denotes the sequence $\varpi_i,...,\varpi_i$ of length $k$.  Concatenate these together to form
$$
\varpi_\bR:=(...,\varpi_\bR(q-1),\varpi_\bR(q),...)
$$
Since $\bR$ is finite, clearly $\varpi_\bR$ is a finite sequence of fundamental weights: $$\varpi_\bR=(\varpi_{j_1},...,\varpi_{j_\ell})$$ for some $j_1,...,j_\ell \in I$.
We let $$\tilde{T}=\tilde{T}^\bR: =\tilde{T}^{\varpi_\bR}$$  Of course $\tilde{T}$ depends on $\bR$, but we fix this choice and usually suppress this from our notation. We warn the reader that $\tilde{T}$ here is different from the algebra with same name in \cite{Webmerged}.

Note that in $\tilde{T}$ all diagrams have red strands labeled by $j_1,...,j_\ell$ from left to right.  Let $r_1,...,r_\ell$ be the weakly increasing ordering of the elements of $\bR$ (with multiplicity).  We say $r_k$ is the {\bf longitude} of $j_k$.

\begin{Example}
Let $\fg=\mathfrak{sl}_4$, and $\la=\varpi_1+2\varpi_2+2\varpi_3$.  We order the nodes of the Dynkin diagram by $2<1<3$.  Consider an integral set of parameters $R_1=\{1\},R_2=\{0,4\}$, and $R_3=\{-1,1\}$.  Then
$$
\varpi_\bR=(\varpi_3,\varpi_2,\varpi_1,\varpi_3,\varpi_2).
$$
Now take $\Bi=(i_1,i_2,i_3)$ and $\kappa:1\mapsto0, 2\mapsto 1, 3\mapsto 1, 4\mapsto 3,5\mapsto 3$.  Then
\begin{equation*}
\begin{tikzpicture}[scale=.4,baseline]
\draw[wei] (0,-1) -- (0,1) node[at start,below]{
$3$} ;
\draw[very thick] (1,-1) -- (1,1) node[at start,below]{
$i_1$};
\draw[wei] (2,-1) -- (2,1) node[at start,below]{
$2$} ;
\draw[wei] (3,-1) -- (3,1) node[at start,below]{
$1$} ;
\draw[very thick] (4,-1) -- (4,1) node[at start,below]{
$i_2$};
\draw[very thick] (5,-1) -- (5,1) node[at start,below]{
$i_3$};
\draw[wei] (6,-1) -- (6,1) node[at start,below]{
$3$} ;
\draw[wei] (7,-1) -- (7,1) node[at start,below]{
$2$} ;
\node at (-2.5,0) {$e(\Bi,\kappa)=$} ;
\end{tikzpicture}
\end{equation*}
The longitudes of the red strands are $-1,0,1,1,4$ from left to right.
\end{Example}

\subsection{The parity KLRW algebra}
\label{sec:parityKLRW}
We now introduce a subalgebra of $\tilde{T}$ whose representation
theory describes the weight modules over  the truncated shifted
Yangian.  This will allow us to deduce some theorems about the highest weight theory of the latter algebra (cf. Theorem \ref{co:main}).

Consider an idempotent $e\in\tilde{T}$.  Since a strand of $e$ (of either color) is labeled by an element of $I$, we can talk about its parity.
We will define the {\bf parity distance} $\delta$ between the strands of $e$ as follows.
First for consecutive strands $p,p'$ we set:
$$
\delta(p,p')=
\begin{cases}
2  \text{ if } p \text{ and } p' \text{ have the same parity, }  p \text{ is black and } p' \text{ is red } \\
1 \text{ if } p \text{ and } p' \text{ have different parity} \\
0 \text{ otherwise}
\end{cases}
$$
%
%
%
We then extend to all pairs of strands by
sharp triangle
inequality $\delta(p,p')=\delta(p,p'')+\delta(p'',p')$, where $p$ is left of $p'$ and $p''$ is any strand between $p$ and $p'$.
\begin{Definition}
\label{def:parityalgebra}
  We say an idempotent $e$ is {\bf parity}, if for any pair of red strands $p,p'$ of $e$, we have
  $\delta (p,p')\leq |r-r'|$, where $r$ is the longitude of $p$ and $r'$ is the longitude of $p'$.  Let $e_{\bR}$ be the sum of all
  parity idempotents.  We call an indecomposable projective $\tilde{T}$-module $Q$ and its corresponding canonical basis vector {\bf parity} if it is a summand of $\tilde{T}e_{\bR}$.
  
  The  {\bf parity KLRW algebra } is the subalgebra $\tilde{P}=\tilde{P}^\bR=e_{\bR}\tilde{T}e_{\bR}$, and $\tilde{P}_\mu=\tilde{P}\cap\tilde{T}_\mu$.
\end{Definition}

Of course, 
the Grothendieck groups $K^0(\tilde{P}\mmod),K^0({}_{\pm}{P}\mmod)$ are naturally isomorphic to the span of parity canonical basis vectors in $U(\bla), V(\bla),V(\bla')$ as appropriate; we denote these subspaces by $U(\bR), V(\bR), V(\bR')$. 

Note that these algebras will be Morita equivalent if and only if
$\tilde{T}=\tilde{T}e_{\bR}\tilde{T}$, which
will hold if every simple module of $\tilde{T}$ is not killed by
$e_{\bR}$ (and thus, does not factor through $\tilde{T}/\tilde{T}e_{\bR}\tilde{T}$).  As for $\tilde{T}$, the algebra $\tilde{P}$ has a natural triple of quotients: we
let ${}_+P^\bR$ be the quotient by the two-sided ideal generated by
all idempotents where $\kappa(1)> 0$, ${}_-P^\bR$ be the
quotient by the two-sided ideal generated by all idempotents where
$\kappa(\ell)< n$, and ${}_0P^\bR$ the quotient by the sum of these
ideals.  If $\bR$ is clear from context we usually omit it.  Similarly, we also have algebras ${}_\pm P_\mu$.

\begin{Example}
\label{ex:parity}
If $|r-r'|\gg n+\ell$ for all $r,r'\in\bR$, then all idempotents with at most $n$ black strands are
  parity.  In particular, if $ht(\lambda-\mu)\leq n$ then $\tilde{P}_\mu=\tilde{T}_\mu$.

  On the
  other hand, if all red strands have the same parity, and we have
  $r_1=\cdots=r_\ell$, then for an idempotent $e(\Bi,\kappa)$ to be parity we must have $\kappa(1)=\cdots=
  \kappa(\ell)\in \{0,n \}$.
\end{Example}

  \begin{Example}
  \label{ex:Nomega1}
  Suppose $\lambda=N\varpi_i$ for some $i\in I$.  The set of parameters is a single multiset $\bR=\{c_1^{t_1},...,c_q^{t_q}\}$, where $c_1<\cdots<c_q$, $c_i$ occurs $t_i$ times in $\bR$, and $\sum t_i=N$.

 The parity algebra $\tilde{P}$ is isomorphic to the KLRW $\tilde{T}^{\boldsymbol{\la}}$, where $\boldsymbol{\la}=(t_1\varpi_i,...,t_q\varpi_i)$.  Indeed, the parity idempotents in this case are all of the form:
\begin{equation*}
\begin{tikzpicture}[scale=.4,baseline]
\draw[very thick] (-3.5,-1) -- (-3.5,1) ;
\draw[very thick] (-1.5,-1) -- (-1.5,1) ;
\draw[wei] (-.5,-1) -- (-.5,1) ;
\draw[wei] (1.5,-1) -- (1.5,1) ;
\draw[very thick] (2.5,-1) -- (2.5,1) ;
\draw[very thick] (4.5,-1) -- (4.5,1) ;
\draw[very thick] (7.5,-1) -- (7.5,1) ;
\draw[very thick] (9.5,-1) -- (9.5,1) ;
\draw[wei] (10.5,-1) -- (10.5,1) ;
\draw[wei] (12.5,-1) -- (12.5,1) ;
\draw[very thick] (13.5,-1) -- (13.5,1) ;
\draw[very thick] (15.5,-1) -- (15.5,1) ;
\draw[decoration={brace,mirror,raise=5pt},decorate]
  (-.5,-1) -- node[below=6pt] {$t_1$} (1.5,-1);
  \draw[decoration={brace,mirror,raise=5pt},decorate]
  (10.5,-1) -- node[below=6pt] {$t_q$} (12.5,-1);
\node at (-2.5,0) {$\cdots$} ;
\node at (1/2,0) {$\cdots$} ;
\node at (3.5,0) {$\cdots$} ;
\node at (6,0) {$\cdots$} ;
\node at (8.5,0) {$\cdots$} ;
\node at (11.5,0) {$\cdots$} ;
\node at (14.5,0) {$\cdots$} ;
\end{tikzpicture}
\end{equation*}
Here the red strands are grouped into $q$ ``packets'', and the black strands can appear  only in between the packets.

Under the isomorphism between $\tilde{P}$ and $\tilde{T}^{\boldsymbol{\la}}$, this idempotent maps to the one in $\tilde{T}^{\boldsymbol{\la}}$ where the $k$-th packet of red strands is collapsed to a single red strand labeled by $t_k\varpi_i$.  Note that in particular when $\fg=\mathfrak{sl}_2$, the parity algebra $\tilde{P}$ is isomorphic to the KLRW algebra.
\end{Example}

\begin{Example}
Consider the case where $\fg=\mathfrak{sl}_3$ and $R_1=\{1\}$ and
$R_2=\{2\}$.  In this case, the idempotent ${\color{red}
  1}\;1\;2\; {\color{red} 2}$ is {\bf not} parity, since the parity
distance between $  {\color{red} 1}$ and ${\color{red} 2}$ is $3$, but
$|r_1-r_2|=1$.
On the other hand $2\;1\;{\color{red}
  1}\;{\color{red} 2}$ 
is parity, since the parity distance between the two red strands
is 1.  (Here, and in Example \ref{example} below, we are using the obvious shorthand for idempotent diagrams: each number corresponds to a strand in the same order, the color of a number corresponds to the color of the strand, and the number itself is the label on the strand.)

We remark that the algebra $\tilde{T}$ has a one dimensional representation where $e({\color{red}
  1},1,2,{\color{red} 2})$ acts by the identity, and every other
homogeneous element of the algebra acts by 1.  This simple is killed
by every parity idempotent, which shows that $\tilde{P}$ and
$\tilde{T}$ are not Morita equivalent in this case.
\end{Example}

\begin{Lemma}
\label{lem:categoricalactions}
  The categorical actions on ${}_{\pm}T\mathrm{-mod}$ of Proposition
  \ref{prop:cat-action} are inherited by the subcategories ${}_{\pm}P\mathrm{-mod}$.
\end{Lemma}
\begin{proof}
  The maps $\iota^\pm$ add black strands that are not between any two
  red strands.  Thus, they do not change the parity distance between
  any two reds, and they send parity idempotents to parity
  idempotents.

Thus, if we have a ${}_+T$-module $M$ which is killed by all parity
idempotents, then the action of a parity idempotent $e$ on $\cE_iM$ is
given by that of the parity idempotent $\iota^+(e)$, which is still
parity and  thus acts by zero.

Every element of $\cF_iM$ lies in the span of $x\otimes m$ where
$m\in M$, and $x$ is a straight-line diagram pulling a black strand
with label $i$ from the far right at $y=0$ to some position at $y=1$,
leaving all other strands in place.  If the top of $x$ is a parity
idempotent, then the bottom is as well (since we can only decrease
parity distance by moving a strand to the far right), and thus
$x\otimes m=0$.  Thus, $\cF_iM$ is killed by parity idempotents.

Since ${}_{\pm}P\mmod$ is the quotient of ${}_{\pm}T\mmod$ by elements
killed by parity idempotents, the result follows.
\end{proof}
This immediately implies:
\begin{Corollary} The subspace $U(\bR)$ is  a $(U(\mathfrak{n_-}),U(\mathfrak{n_-}))$-subbimodule, and $V(\bR),V(\bR')$ are $U(\fg)$-submodules.  
\end{Corollary}
Note, these spaces can be naturally quantized to give based modules in the sense of Lusztig; we will not discuss this in any detail, since we will not use anything other than the existence of the canonical basis.  

Let $\Irr({}_-P)$ denote the set of isomorphism classes of simple
${}_-P$-modules.  This has a crystal structure with the Kashiwara
operators defined by
\[\tilde{e}_iL=\operatorname{hd} (\cE_i^-L)\qquad \tilde{f}_iL=\operatorname{hd}(\cF_i^-L).\]
This is a subcrystal of $\Irr({}_-T^{\bla})$ which is isomorphic to the corresponding tensor
product crystal by \cite[Thm. 7.2]{LoWe}.

Similarly, the set $\Irr(\tilde{P})$ of {\it nilpotent} simple modules
over $\tilde{P} $ carries a bicrystal structure
defined by \[\tilde{f}_iL=\operatorname{hd}(\cF_i^-L)\qquad
  \tilde{f}_i^*L=\operatorname{hd}(\cF_i^+L).\]
We need to impose the nilpotent condition since $
\tilde{P}$ is finite over its center, which is a polynomial ring;
thus, for any maximal ideal of the center, there are simples killed by
that maximal ideal, with the nilpotent simples corresponding to unique
graded maximal ideal.  These are also the simple modules which will be
relevant to the representation theory of the Yangian via Theorem \ref{co:tilde-main}.

Note that many other sources on KLR algebras, such as \cite{KLI},
have typically studied graded modules instead.  However, nilpotent
simple modules are in bijection with graded simple modules considered
up to grading shift:
\begin{Lemma}\label{lem:nil-grade}
A simple module over $\tilde{P}$ is nilpotent if and only if it is
gradable.  Every indecomposable projective $\tilde{P}$-module has a unique
nilpotent simple quotient.  
\end{Lemma}
\begin{proof}
  Since $\tilde{P}$ is finitely generated over a commutative
  subalgebra, its simple modules must be finite dimensional, so any
  element of non-zero degree must be nilpotent on any simple.

  On the other hand, the subalgebra $A$ of polynomials in the dots
  which is symmetric under permutations of strands and their labels
  is a central and $\tilde{P}$ is
  finitely generated over $A$.  As usual, any central subalgebra acts
  semi-simply on any simple module.  Thus, a simple is nilpotent if it factors through
  the quotient by the unique graded maximal ideal $\mathfrak{m}$ of this central
  subalgebra.  The result then follows from the fact that simple
  modules over finite dimensional graded algebras are always
  gradable uniquely up to grading shift.

  Similarly, if we have an indecomposable projective $Q$ over
  $\tilde{P}$, then $Q/\mathfrak{m}Q$ is a projective module over
  $\tilde{P}/\mathfrak{m}\tilde{P}$.  Let $L$ be a simple module in the
  cosocle of $Q/\mathfrak{m}Q$.  As discussed above, $L$ is gradeable,
  and since   $\tilde{P}$ has finite dimensional degree 0 part, $L$
  has a unique graded projective cover $Q'$.  The projective property
  induces a surjective map $Q\to Q'$, which must be an isomorphism by
  the indecomposability of $Q$.  Thus,
  $Q/\mathfrak{m}Q=Q'/\mathfrak{m}Q'$  is the projective cover of $L$
  as  a $\tilde{P}/\mathfrak{m}\tilde{P}$-module, and so $Q$ has no
  other nilpotent simple modules as quotients.
\end{proof}

\subsection{A crystal isomorphism}
\label{sec:crystalisom}
In this section we'll show that the crystals $\Irr({}_-P)$ and $\B(\bR)$ are isomorphic.

Presently, we describe how to construct an idempotent $e(\bS)\in {}_-P$  from a monomial $\yMon_{\bR}\zMon_{\bS}^{-1}\in \B(\bR)$.
First, we order the elements of $\bS$ so that $s_1\leq \cdots \leq s_n$.  Then, we define a sequence $\Bi=(i_1,...,i_n)$ by  $s_m\in S_{i_m}$ for all $m$.  (Note if elements occur in $\bS$ with multiplicity, there is not a unique such $\Bi$, but this choice won't affect the isomorphism class of $e(\bS)$.) We call $s_m$ the \textbf{longitude} of $i_m$.  (Recall that we've already defined the longitudes of red strands in Section \ref{section:TR}.)

The idempotent $e(\bS)$ is  given by interlacing according to longitude the $\ell$ red strands with the $n$ black strands.  Of course the red strands are labeled
$j_1,...,j_\ell$ from left to right, while the black strands are labeled by $i_1,...,i_n$ from left to right.
If the longitude of a red strand agrees with the longitude of a black one, then the red
strand goes to the left of the black one.  As mentioned above, we also have to make a choice of the order of black strands with the same
longitude, but reordering these gives idempotents which are isomorphic
by the obvious straight-line diagram.


Going in the other direction, given an idempotent $e\in {}_-P$ we construct $\bS$ as follows: for every black $i$-strand we add an element to $S_i$ equal to:
$$
 (\text{longitude of closest red strand to its right})-(\text{parity distance between these strands}).
 $$
 Note that such a red strand exists since we are working in ${}_-P$.  This construction  gives the unique $\bS$ such that $e(\bS)=e$ with $x(\bS)$ is maximal, where:

\begin{Definition}
\label{def:x(S)}
For $\yMon_\bR \zMon_\bS^{-1}\in \B(\bR)$ set
$
x(\bS)=\sum_{s\in S_i}s
$.
\end{Definition}

\begin{Example}
\label{example}
Consider $\g=\mathfrak{sl}_3$ and $\lambda=\varpi_1+\varpi_2$.   If we take $R_1=\{-1\}$ and
$R_2=\{2\}$, the elements of the product monomial crystal are
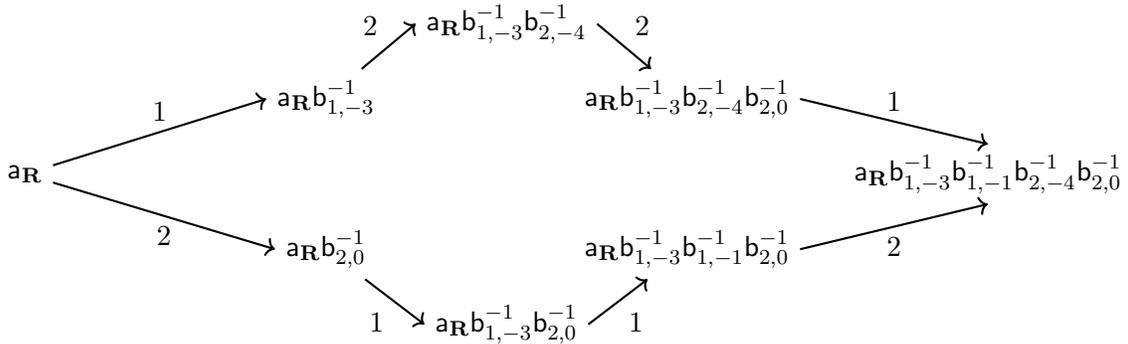
\[
\tikz[->,thick,xscale=1.6]{
 \node at (-4,0) (a){$\yMon_{\bR}$};
\node at (-1.5, 1) (b){$\yMon_{\bR}\zMon_{1,-3}^{-1}$};
   \node at (-1.5,-1) (c){$ \yMon_{\bR}\zMon_{2,0}^{-1}$};
  \node at (0,2) (d){$ \yMon_{\bR}\zMon_{1,-3}^{-1}\zMon_{2,-4}^{-1}$};
  \node at (0,-2) (e){$ \yMon_{\bR}\zMon_{1,-3}^{-1}\zMon_{2,0}^{-1}$};
  \node at (1.5,-1) (f){$\yMon_{\bR}\zMon_{1,-3}^{-1}\zMon_{1,-1}^{-1}\zMon_{2,0}^{-1}$};
  \node at (1.5,1) (g){$  \yMon_{\bR}\zMon_{1,-3}^{-1}\zMon_{2,-4}^{-1}\zMon_{2,0}^{-1}$};
  \node at (4,0) (h){$ \yMon_{\bR}\zMon_{1,-3}^{-1}\zMon_{1,-1}^{-1}\zMon_{2,-4}^{-1}\zMon_{2,0}^{-1}$};
\draw[->](a)--(b.west) node[midway,above]{$1$};
\draw[->](a)--(c.west) node[midway,below]{$2$};
\draw[->](c)--(e.west) node[midway,below left]{$1$};
\draw[->](b)--(d.west) node[midway,above left]{$2$};
\draw[->](e.east)--(f) node[midway,below right]{$1$};
\draw[->](d.east)--(g) node[midway,above right]{$2$};
\draw[->](g.east)--(h.north) node[midway,above]{$1$};
\draw[->](f.east)--(h.south) node[midway,below]{$2$};}
\]
The corresponding  idempotents are 
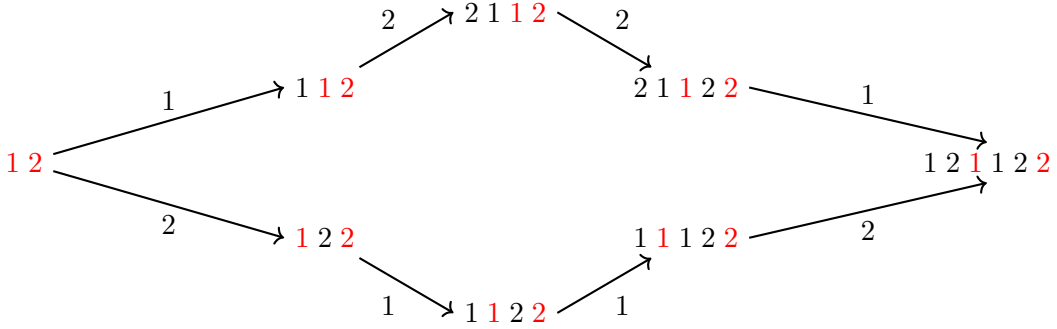
\[
\tikz[->,thick,xscale=1.6]{
 \node at (-4,0) (a){${\color{red}
  1}\;{\color{red} 2}$};
\node at (-1.5, 1) (b){$1\;{\color{red}
  1}\;{\color{red} 2}$};
   \node at (-1.5,-1) (c){$ {\color{red} 1}\; 2\; {\color{red} 2}$};
  \node at (0,2) (d){$2\; 1\; {\color{red}
  1}\; {\color{red} 2}$};
  \node at (0,-2) (e){$1\; {\color{red}
  1}\; 2\; {\color{red} 2}$};
  \node at (1.5,-1) (f){$1\; {\color{red}
  1}\; 1\; 2\; {\color{red} 2}$};
  \node at (1.5,1) (g){$ 2\; 1\; {\color{red}
  1}\; 2\; {\color{red} 2}$};
  \node at (4,0) (h){$ 1\; 2\; {\color{red}
  1}\; 1\; 2\; {\color{red} 2}$};
\draw[->](a)--(b.west) node[midway,above]{$1$};
\draw[->](a)--(c.west) node[midway,below]{$2$};
\draw[->](c)--(e.west) node[midway,below left]{$1$};
\draw[->](b)--(d.west) node[midway,above left]{$2$};
\draw[->](e.east)--(f) node[midway,below right]{$1$};
\draw[->](d.east)--(g) node[midway,above right]{$2$};
\draw[->](g.east)--(h.north) node[midway,above]{$1$};
\draw[->](f.east)--(h.south) node[midway,below]{$2$};}
\]
Note that all the idempotents appearing here are parity.  This is a special case of Theorem \ref{thm:highest-bijection} below.
\end{Example}

Given a simple ${}_-P$-module $L$, there is an $\bS$ with $x(\bS)$ maximal
such that $e(\bS)L\neq 0$.  We call this the {\bf highest
  weight}  of $L$.  It's
not obvious the highest weight is unique, but we will show this below.
\begin{Theorem}\label{thm:highest-bijection}
  Every simple ${}_-P$-module has a unique highest weight.  This induces a crystal isomorphism $\varphi: \Irr({}_-P) \to \B(\bR)$.  This bijection sends the simple modules that factor through
  ${}_0 {P}$ to the highest weight elements of $\B(\bR)$.
\end{Theorem}
We will prove this theorem after developing the representation theory of ${}_-P$ a bit further.
\begin{Definition}
  Given an idempotent $e(\bS)$, let $\Delta(\bS)$ be the
  quotient of ${}_-Pe(\bS)$ by
  $$
  \sum{}_-Pe(\bS'){}_-Pe(\bS)
  $$
  where the sum is over $\bS'$  such that $x(\bS')<x(\bS)$.
\end{Definition}

\begin{Lemma}\label{lem:highest-monomial}\hfill
\begin{enumerate}
\item  Suppose that $\Delta(\bS)$ is nonzero.  Then it has a unique simple quotient $L(\bS)$ which has $\bS$ as its unique highest weight.
\item There is an injective map
$\varphi:\Irr({}_-P) \to \B(\bR)$,
given by $L(\bS) \mapsto \yMon_\bR \zMon_\bS^{-1}$.
\end{enumerate}
\end{Lemma}
\newcommand{\ttop}{\mathrm{top}}
\newcommand{\bbot}{\mathrm{bot}}
\begin{proof}
  First we prove (1).  Assume that $x(\bS')=x(\bS)$.  We will
  first show that $e(\bS') \Delta(\bS)$ is zero unless $\bS'=\bS$.  Indeed, suppose $\bS'\neq\bS$ and we have a nonzero diagram $D\in e(\bS') \Delta(\bS)$.  Note that attached to a black strand $s$ in $D$ there are two pieces of additional data, namely a top longitude coming from $\bS'$, and a bottom longitude coming from $\bS$.  We'll denote these $\ttop(s)$ and $\bbot(s)$ respectively.

Since $\bS'\neq\bS$ there is some black strand $s$ in $D$ such that $\bbot(s)>\ttop(s)$.  Suppose this strand is labeled $i$.  Define $\bS''$ by
  $$
  S''_j=
  \begin{cases}
 S_i\setminus \{\bbot(s)\} \cup \{\ttop(s)\} \text{ if } j=i \\
 S_j \text{ otherwise}
  \end{cases}
  $$
Let $D'' \in e(\bS''){}_-Pe(\bS)$ be the diagram equal to $e(\bS)$ except for one black $i$-strand whose bottom longitude is $\bbot(s)$ and top longitude is $\ttop(s)$.  Then there is a diagram $D' \in  e(\bS'){}_-Pe(\bS'')$ such that in $e(\bS'){}_-Pe(\bS)$ we have
$$
D=D'D''+(\text{diagrams with fewer crossings})
$$
Since $x(\bS'')<x(\bS)$, we have that $D'D''=0$ in $\Delta(\bS)$.
%
Therefore, by induction, we may now
  assume that no strand in $D$ has bigger bottom longitude than top longitude.
  This is only possible if $\bS=\bS'$, contradicting our assumption.  This shows that $e(\bS') \Delta(\bS)=0$.

Note that  $e(\bS) \Delta(\bS)$ is a
  quotient of the subalgebra $A$ in $e(\bS) \tilde{P}e(\bS) $ where all strands preserve their
  longitude.  The algebra $A$ is isomorphic to a tensor product of affine nilHecke algebras.  Thus it is Morita equivalent to a tensor product of symmetric polynomials on alphabets corresponding to like-colored strands with a given longitude. The module $e(\bS) \Delta(\bS)$ is cyclic over $A$, and thus has a unique simple quotient (since the same is true of a finite dimensional graded quotient of a polynomial ring).  A standard argument as in \cite[Lem. 5.9]{Webmerged} shows that the same is true of $\Delta(\bS)$.

It remains to show that $L(\bS)$ has  $\bS$ as its unique highest weight.  But this is clear, since $e(\bS)L(\bS)\neq 0$, and $e(\bS')\Delta(\bS)=0$ for any $\bS'\neq\bS$ such that $x(\bS')\leq x(\bS)$.  The proof of (2) is now immediate since  every simple module is a quotient of some $\Delta(\bS)$.
\end{proof}

\begin{Proposition}\label{lem:crystal-map}
The map $\varphi:\Irr({}_-P)\to\B(\bR)$ is a morphism of $\fg$-crystals.
\end{Proposition}
\begin{proof}

Fix $i \in I$ and an idempotent $e(\bU) \in {}_-P$ such that $U_i=\emptyset$.  Let $e=\sum e(\bU')$, where the sum is over all idempotents such that $U_j'=U_j$ for $j\neq i$.

We wish to relate the simple ${}_-P$-modules on which this idempotent acts
non-trivially to modules over the  tensor product algebra for the
root
$\mathfrak{sl}_2$ subalgebra of $\fg$ corresponding to $i$.  We consider a
set of parameters for  $\mathfrak{sl}_2$ given by $\bR'=\{R_i\cup
\bigcup_{i\sim j} (U_j+1)\}$; note that since we are working with
$\mathfrak{sl}_2$ a set of parameters consists of a single multiset.  Let ${}_-P'$ be the $\mathfrak{sl}_2$ parity algebra
associated with  $\bR'$.  To remind
ourself which root subalgebra we're working with, we label the black and red
strands with $i$.

Let $A_\bU$ be the
subalgebra of $\tilde{P}$ spanned by diagrams where all red strands
and all black strands with labels other than $i$ are straight vertical in the position
of $e(\bU)$.
We have a natural map $A_{\bU}\to e {}_-P e$, which
makes ${}_-P e$ into a right $A_\bU$-module.
We have a ``redification'' map $\varrho_i:A_\bU \to {}_-P'$ defined on diagrams as follows:
\begin{itemize}
\item Black and red $i$-strands remain in their  positions.
\item For $j \sim i$, a dotless black $j$-strand maps to a red
  $i$-strand.
\item A dot on a black $j$-strand is sent to 0.
\item Remove any black $j$-strand if $j \not\sim i$ and $j \neq i$.
\item For $j\neq i$, remove any red $j$-strand.
\end{itemize}
Note that for a diagram in $A_\bU$, any crossing must involve a black $i$-strand.  Therefore the nilHecke relations (\ref{nilHecke-1}--\ref{black-bigon})
only involving one color are only relevant for the label $i$, and they
are sent to the same relation in ${}_-P'$.
When strands of more than one color are involved,  $\varrho_i$ sends the relations (\ref{first-QH}--\ref{second-QH}) to
(\ref{red-dot}), the relation (\ref{black-bigon}) to (\ref{cost}) and
the relations (\ref{triple-dumb}--\ref{triple-smart}) to (\ref{red-triple-correction}--\ref{dumb}).  Therefore $\varrho_i$ is indeed a map of algebras.

Thus, given a ${}_-P'$-module $M$, we can consider the associated
module $ {}_-{P} e\otimes_{A_\bU} M$.  Note that it is quite possible
that this tensor product is 0; we are free to ignore these cases.

We define a functor $\bX:{}_{-}P'\mmod\to {}_{-}P\mmod$ as follows: $\bX(M)$ is the
quotient of  $ {}_-P e\otimes_{A_\bU} M$ by the submodule generated
the image of all idempotents where the average longitude of the
black strands with labels other than $i$ is higher than $x(\bU)$.  The image $e
\mathbb{X}(M) $ is naturally an $A_\bU$-module, and
as
in the proof of Lemma \ref{lem:highest-monomial}, we see that $e
\mathbb{X}(M) $
is a quotient of $M$.  In particular, if $M$ is simple,
then applying \cite[Lem. 5.9]{Webmerged} shows that $\mathbb{X}(M)$
has a unique simple quotient $\mathbb{L}(M)$ (if it is non-zero).
One can also easily confirm that $\mathbb{L}$ commutes with the
functors $\EuScript{E}_{i},\EuScript{F}_{i}$ of the categorical
$\mathfrak{sl}_2$-action on both categories.  Thus, the induced map on
simples commutes with the categorical crystal operators for $e_{i}$
and $f_{i}$.  In particular, if a simple in the crystal is sent to
$0$, then its whole component is.

Let $\B(\bR')$ be the product monomial crystal for $\mathfrak{sl}_2$.  Note that we
have a crystal map $\mathsf{L}:\B(\bR')\to\B(\bR)$ given by $\yMon_{\bR'}\zMon_{\bT}^{-1} \mapsto \yMon_{\bR}\zMon_{\bS\cup
  \bT}^{-1}$.  Since the numbers $a_{i,k}$ are unchanged by this map, $\mathsf{L}$ is a map of $\mathfrak{sl}_2$ crystals. Thus we have a diagram of maps, where the left hand map
is dashed since it is only partially defined:
$$
\tikz[very thick]{\matrix[row sep=12mm,column sep=30mm,ampersand replacement=\&]{
\node(a){$\Irr({}_-P')$}; \& \node(c){$\B(\bR')$}; \\
\node(b){$\Irr({}_-P) $}; \&  \node(d){$\B(\bR)$};\\
};
\draw[dashed,->] (a) -- node[left,midway]{$\mathbb{L}$} (b);
\draw[->] (c) -- node[right,midway]{$\mathsf{L}$} (d);
\draw[->] (a) -- node[above,midway]{$\varphi'$} (c);
\draw[->] (b) -- node[below,midway]{$\varphi$} (d);
}
$$
This diagram commutes on objects where $\mathbb{L}$ is defined.
Both vertical maps are morphisms of $\mathfrak{sl}_2$-crystals, and
the maps $\varphi$ and $\varphi'$ are both injective.  Furthermore,
every element of $\Irr({}_-P) $ is in the image of  $\mathbb{L}$ for
some choice of $\bU$; in fact, this implies that its whole root string
under $\tilde{e}_i,\tilde{f}_i$ is.

Thus to show that $\varphi$ is a crystal map, it suffices to prove
this for $\varphi'$ for an arbitrary choice of $i$ and $\bU$.  That
is, it suffices to prove the result in the case where $\g=\mathfrak{sl}_2$.  In this case, the monomial crystal structure is simply a tensor product crystal, and ${}_-P$ is an $\mathfrak{sl}_2$ tensor product algebra (cf. Example \ref{ex:Nomega1}).  The crystals for these match by \cite[Th. 7.2]{LoWe}.
\end{proof}

%
\begin{proof}[Proof of Theorem \ref{thm:highest-bijection}]
To complete the proof of the bijection, we need to show that $\Delta(\bS)$ is non-trivial if and only if $\yMon_{\bR}\zMon_{\bS}^{-1}$ lies in $ \B(\bR) $.  We'll prove this by induction on $n=\sum_i|S_i|$.

For $n=0$ we have that $\bS=\emptyset$.  Since  $\Delta(\emptyset)=\C e(\emptyset)\neq0$ and $\yMon_\bR\in\B(\bR)$, the claim follows.

Now let $n>0$.  Suppose first that there is an element of $\bS$ which is less
than or equal to every element of $\bR$.  If $\Delta(\bS)$ is non-trivial then $L(\bS)$ is not killed
by $\cE_i$, for some $i$.  Thus, $L(\bS)=\tilde{f}_i  L(\bS')$ for some highest weight $\bS'$.  By
induction we can assume that $\yMon_\bR \zMon_{\bS'}^{-1}\in  \B(\bR) $, and hence by Proposition \ref{lem:crystal-map} $\yMon_\bR \zMon_\bS^{-1}=\tilde{f}_i (\yMon_\bR \zMon_{\bS'}^{-1})$ is in $\B(\bR)$ as well.
Conversely, if $\yMon_\bR \zMon_{\bS}^{-1} \in \B(\bR)$ then $\tilde{e}_i(\yMon_\bR \zMon_{\bS}^{-1})\neq 0$, and  reversing the argument above we get that $L(\bS)\neq0$.

It remains to consider the case where there is $r\in R_i$ strictly less than all
elements of $\bS$.  We can assume $r$ corresponds to the leftmost red
strand in $e(\bS)$.  Let $\bR'$ be the set of parameters of weight $\lambda -\varpi_i$ obtained from $\bR$ by removing $r$ from $R_i$.  Let ${}_-P'$ be the associated algebra.  We have a map ${}_-P' \to {}_-P$ given by adding a red $i$-strand on the left.  Let $e \in {}_-P$ be the sum of idempotents whose leftmost strand is red.  If $\Delta(\bS)$ is non-trivial then $eL(\bS)$ is a nontrivial ${}_-P'$-module, which has highest weight $\bS$.  Therefore some simple composition factor of this module has highest weight $\bS$.  By induction we then have that $\yMon_{\bR'} \zMon_{\bS}^{-1}\in
\B(\bR')$.  Now note that multiplication by $\yMon_{i,r}$ defines an injective set map $\iota:\B(\bR') \to \B(\bR)$.  Hence $\yMon_{\bR} \zMon_{\bS}^{-1}\in
\B(\bR)$.

Conversely, if $\yMon_{\bR} \zMon_{\bS}^{-1}\in
\B(\bR)$ then it's in the image of $\iota$.  Hence $\yMon_{\bR'} \zMon_{\bS}^{-1}\in\B(\bR')$.  By induction we then have that the irreducible ${}_-P'$-module $L'(\bS)\neq 0$.  By Frobenius reciprocity there exists a nonzero map ${}_-P\otimes_{{}_-P'}L'(\bS) \to L(\bS)$, and so $\Delta(\bS)\neq0$.

This completes the proof that the map $L(\bS) \mapsto \yMon_\bR \zMon_\bS^{-1}$ is a crystal isomorphism from the set of isomorphism classes of simple ${}_-P$-modules to $\B(\bR)$.

To complete the proof of the theorem, assume that a simple $L(\bS)$ is a highest weight element of the crystal.  This holds if and
only if
$\cE_i L(\bS)=0$ for all $i$, which holds if and only if $L(\bS)$ is killed by all
idempotents whose leftmost strand is black.  This is, of course, equivalent to
factoring through the quotient ${}_0P$.
\end{proof}

Note, this proves an interesting and non-trivial statement: there is a
natural embedding of the monomial product crystal in the tensor
product crystal for any ordering compatible with the longitudes.  This
embedding is uniquely characterized by the fact that adding a new
element of $\bR$ larger than all the longitudes of elements in $\bS$
matches tensoring with the highest weight element of a new tensor
factor.

We can extend this result to understand the sets of nilpotent simple
modules $\Irr(\tilde{T}^{\bR}) $ and
$\Irr(\tilde{P}^{\bR})$ as crystals.  In the case where 
$\bR=\emptyset$, we have that $\tilde{T}^{\emptyset}$ is the original
KLR algebra, and work of Lauda and Vazirani \cite[Th. 7.4]{LV}
identifies $\Irr(\tilde{T}^{\emptyset})$ with the crystal $B(\infty)$.

Recall from \cite[Def. 5.4]{Webmerged} that we have a partial standardization
functor
\[ \mathbb{S}^{\bR; \emptyset}\colon  {}_-T^\bR \otimes \tilde{T}^{\emptyset} \mmod \to
  \tilde{T}^\bR\mmod.\]
Note that since we are using ${}_-T$, we swap left and right
everywhere compared to \cite{Webmerged}, in particular in the
definition of standard modules.  This correctly accounts for the fact that we use the opposite convention for crystal tensor product from \cite{LoWe,Webmerged}.
As argued in \cite[Th. 5.8]{Webmerged}, the map \[h\colon 
   \Irr({}_-{T}^{\bR})\times \Irr(\tilde{T}^{\emptyset})\to
  \Irr(\tilde{T}^{\bR}) \qquad   h(L,L')\mapsto
\operatorname{hd}(\mathbb{S}^{\emptyset;\bR}(L\boxtimes L'))\] is a
bijection.  As discussed before, \cite{Webmerged} deals with graded
simple modules, but the same results hold for nilpotent modules by
Lemma \ref{lem:nil-grade}.
\begin{Lemma}
\label{lem:tildeP-crystal}
  The map $h$ induces a crystal isomorphism
  $\Irr(\tilde{P}^{\bR})\cong
\mathcal{B}(\bR)  \otimes \mathcal{B}(\infty)$.    
\end{Lemma}
\begin{proof}
We can make sense of parity algebras where we allow $\pm\infty$ as
longitudes (for these purposes, $\infty$ is both even and odd).  Note
that the parity distance between any two strands with longitude
$\infty$ is 0.

Fix a large integer $n$, and let $\bR_+$ be obtained from $\bR$ by
adding $n$ copies of $\infty $ to each $R_i$.  Note that we have a
surjective map $\tilde{P}^{\bR}\to {}_-{P}^{\bR_+}$ by adding $n$
strands labeled by each fundamental weight at the far right of the
diagram; no black strands are allowed between these new red strands,
since the parity distance between them is 0.  Let $I$ be the kernel of
this map.  We have a surjective map
\begin{equation}\label{eq:PR-proj}
\Irr(\tilde{P}^{\bR})\to
\Irr({P}^{\bR_+})\sqcup \{0\}
\end{equation}
defined by $L\mapsto L/IL$. 

Furthermore, this map is a bijection in any given weight $\mu$ for
sufficiently large $n$; thus these maps for all $n\geq 0$ unique fixed
the crystal structure on $\Irr(\tilde{P}^{\bR})$.

Straightforward modification of \cite[Th. 5.8]{Webmerged} shows that we have a natural bijection $\mathcal{B}(\bR_+)\cong \mathcal{B}(\bR)\otimes
\mathcal{B}(n\rho)$ compatible with $h$, and \cite[Th. 7.2]{LoWe} shows this is a crystal isomorphism.   Taking inverse limit as $n\to \infty$, we
obtain that $h$ is a crystal map as well. 
\end{proof}

\subsection{Metric KLRW algebras}
We'll now consider a different generalization of the KLRW
algebras of \cite[\S 4]{Webmerged}.
\begin{Definition}\label{def:long}
  Consider a sequence $\Bi\in I^n$ and a function $\kappa\colon
  [1,\ell]\to [0,n]$ fixing the
  position of the red strands.   A {\bf
  (metric) longitude}  compatible with this data is a sequence $\Ba=(a_1,\dots, a_n)\in \Z^n$ such that for all $k$,
\begin{enumerate}
\item $i_k$ and $a_k$ have the same parity,
\item $a_{k}\leq a_{k+1}$, and
\item $a_k\geq  r_p$ if and only if  $k> \kappa(p)$.
\end{enumerate}
\end{Definition}
The longitude conditions say roughly that  we can isotope the
idempotent $e(\Bi,\kappa)\in\tilde{T}$ to have a black strand of label $i_k$ at
$x=a_k$, and red strands at $x=r_p$; condition (2) implies that the
black strands are in the correct order, and condition (3) implies
that red and black strands are in the correct relative position. Note
that there are some border cases that keep this from being strictly
true: consecutive black strands can have the same longitude, but
cannot have the same $x$-value, and similarly, a red strand and the
black to its right (but not its left) can have the same longitude, but
not the same $x$-value.

We let $\Long$ be the set of possible combinations of data
$(\Bi,\kappa)$ as above and the longitudes $(a_1,\dots, a_n)$.   Note that the
information of $\Ba$ is enough to uniquely fix the choice
of $\kappa$.  We can also speak of the integer this attaches to a
single strand as its longitude; the longitude of a red strand is by convention the
corresponding element $r_p$.

\begin{Definition}\label{def:Stendhal}
  A {\bf metric \Stendhal diagram} is a Stendhal diagram together with a
  choice of a longitudes $\Ba_{\operatorname{top}}$ and $\Ba_{\operatorname{bottom}}$ on the data
  $(\Bi_{\operatorname{top}},\kappa_{\operatorname{top}})$ and
  $(\Bi_{\operatorname{bottom}},\kappa_{\operatorname{bottom}})$ at the top and bottom of the diagram.
\excise{ a collection of finitely many oriented curves in
  $\R\times [0,1]$ which are colored black or red. Each curve is
  labeled with $i\in I$ and decorated with finitely many
  dots.
  The diagram must be locally of the form \begin{equation*}
\begin{tikzpicture}
  \draw[very thick] (-4,0) +(-1,-1) -- +(1,1);
  \draw[very thick](-4,0) +(1,-1) -- +(-1,1);


  \draw[very thick](0,0) +(-1,-1) -- +(1,1);
  \draw[wei, very thick](0,0) +(1,-1) -- +(-1,1);

  \draw[wei,very thick](4,0) +(-1,-1) -- +(1,1);
  \draw [very thick](4,0) +(1,-1) -- +(-1,1);

  \draw[very thick](8,0) +(0,-1) --  node
  [midway,circle,fill=black,inner sep=2pt]{}
  +(0,1);
\end{tikzpicture}
\end{equation*}
with each curve oriented in the negative direction.  In
particular, no red strands can ever cross.}
\end{Definition}
As in \cite[Def. 4.2]{Webmerged}, we can define a product on the
formal span of all metric \Stendhal diagrams considered up to isotopy.  Note that in taking this
product, we must require that longitudes in addition to labels in
$I$ match, or else we take the product to be 0.  There are idempotent diagrams
in this algebra given by straight line diagrams with the same
longitudes at the top and bottom.  The isomorphism type of this
idempotent only depends on $\Bi$ and $\kappa$, since we can label the
top and bottom of this diagram with any compatible longitude.

\begin{Definition}
\label{def:metricKLR}
  The {\bf metric KLRW algebra} $\tmetric=\tmetric^\bR$ consists of finite spans of
  metric \Stendhal diagrams, modulo the local relations
  (\ref{first-QH}--\ref{cost}).  We let $\pmmetric$ denote its quotients by
  diagrams that violate on left and right respectively, and as usual $\pmmetric_\mu \subset \pmmetric$ denote the subalgebras where the sum of the roots labeling black strands is equal to $\lambda-\mu$.
\end{Definition}

\begin{Example}
\label{ex:metricdiag}
Let $I=\{x,y\}$, connected by a single edge $x\to y$. In this case, $\g=\mathfrak{sl}_3$, and we take $\lambda=\varpi_x+\varpi_y, R_x=\{-1\}$, and $R_y=\{4\}$.  Here is a typical  element of the metric KLRW algebra $\tmetric^\bR$.
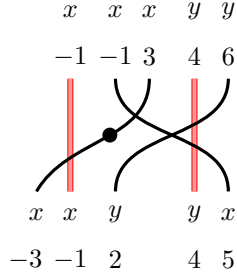
\[ \tikz[very thick,scale=1.5,baseline]{
    \draw[wei] (-.5,-.5) -- (-.5,.5);
    \draw[wei] (.6,-.5) -- (.6,.5);
    \draw (-.8,-.5) to [out=60,in=-90] (.2,.5);
    \draw (-.1,-.5) to [out=90,in=-90] (.9,.5);
    \draw (.9,-.5) to [out=90,in=-90] (-.1,.5);

  \draw (-.8,-.5)+(0,-.2) node {\small$x$};
  \draw (-.5,-.5)+(0,-.2) node {\small$x$};
  \draw (-.1,-.5)+(0,-.2) node {\small$y$};
  \draw (.6,-.5)+(0,-.2) node {\small$y$};
  \draw (.9,-.5)+(0,-.2) node {\small$x$};

  \draw (-.8,-.5)+(-.1,-.6) node {\small$-3$};
  \draw (-.5,-.5)+(0,-.6) node {\small$-1$};
  \draw (-.1,-.5)+(0,-.6) node {\small$2$};
  \draw (.6,-.5)+(0,-.6) node {\small$4$};
  \draw (.9,-.5)+(0,-.6) node {\small$5$};

   \draw (.2,.5)+(0,.2) node {\small$3$};
  \draw (-.5,.5)+(0,.2) node {\small$-1$};
  \draw (-.1,.5)+(0,.2) node {\small$-1$};
  \draw (.6,.5)+(0,.2) node {\small$4$};
  \draw (.9,.5)+(0,.2) node {\small$6$};

  \draw (.2,.5)+(0,.6) node {\small$x$};
  \draw (-.5,.5)+(0,.6) node {\small$x$};
  \draw (-.1,.5)+(0,.6) node {\small$x$};
  \draw (.6,.5)+(0,.6) node {\small$y$};
  \draw (.9,.5)+(0,.6) node {\small$y$};

  \draw (-.15,0) node [circle,fill=black,inner sep=2pt]{};
} \]
On the top and bottom we depict the string labels and longitudes in two rows.  So for instance, the bottom longitudes of the black strands in the above diagram are $-3,2,5$.  Notice that the longitude of a black strand can change from bottom to top, but the longitudes on the red strands are fixed by the choice of $\bR$.
\end{Example}

Note that a metric structure on the idempotent $e(\Bi,\kappa)$ is
precisely the same as a choice of $\bS$ such that
$e(\Bi,\kappa)=e(\bS)$. We let $d(\bS)\in\tmetric$ be the
 idempotent whose diagram is the same as $e(\bS)$, with the (metric) longitudes given by
the longitudes of $\bS$.
The key connection between the metric and parity KLR algebras is:
\begin{Lemma}\label{lem:compat-long}
  An idempotent in $\tilde{T}$ has a compatible metric longitude if and
  only if it is a parity idempotent.
\end{Lemma}
\begin{proof}
Suppose an idempotent in $\tilde{T}$ has a compatible longitude.  Let's call the difference between the metric longitude of two strands their ``metric distance''.  We claim that the metric distance is always greater than or equal to parity distance.  This will show that the idempotent is parity.

Both distances satisfy the strict triangle
  identity, so we need only check this for consecutive strands. If two
  consecutive black strands have the same parity, then the parity distance
  between them is 0.  If they have opposite parity, then their metric longitudes
  have opposite parity as well, so the metric distance between them is at
  least 1, and thus greater than or equal to the parity distance.  If
  we have a black strand and then a red of the same parity, then they
  cannot have the same metric longitude by Definition \ref{def:long}(3), so they must differ by at least 2, which is the parity
  distance.  Thus, indeed, the parity distance is a lower bound on the
  metric distance.

Conversely, given a parity idempotent, we have a corresponding monomial as described in Section \ref{sec:crystalisom}, which gives a
  choice of compatible metric longitudes.
\end{proof}

We have a natural bimodule between $\tmetric$ and
$\tilde{T}$ where we only label the top of strands with a compatible longitude, and do not choose one on the bottom.  Multiplying on the right by
$e_\bR$, we obtain a bimodule between $\tmetric$ and
$\tilde{P}$.
One can easily confirm that:
\begin{Lemma}\label{lem:Morita}
  The bimodule above induces a Morita equivalence between
  $\tmetric$ and $\tilde{P}$, and between $\pmmetric$
  and ${}_{\pm}P$.
\end{Lemma}
\begin{proof}
For each parity idempotent, we choose a single compatible longitude,
and let $e'$ be the sum of the corresponding idempotents in
$\tmetric$.  We obtain an isomorphism $\tilde{P}\to
e'\tmetric e'$ by labeling the top and bottom of each diagram
with the corresponding longitude.

On the other hand, we have that
$\tmetric e'\tmetric=\tmetric$, since any idempotent
$e\in\tmetric$ can be written as product of two
idempotents, with labels at the extremes being arbitrary
and in the center being the fixed longitudes for $e$.
\end{proof}

As with the parity algebra, a simple ${}_-\metric$-module $L$ has {\bf highest weight} $\bS$ if $x(\bS)$ is maximal such that $d(\bS)L\neq0$.
By Theorem \ref{thm:highest-bijection} we obtain:
\begin{Corollary}
\label{cor:metric-bijection}
There is a bijection, which (by slight abuse of notation) we denote $\varphi$, between the simple modules of
${}_-\metric$ and the product monomial crystal $\B(\bR)$, sending a
simple $L$ to $\yMon_\bR \zMon_\bS^{-1}$, where $\bS$ is the highest weight of $L$.
\end{Corollary}

\subsection{Coarse metric KLRW algebras}
\label{sec:coarsemetric}

We'll also want to consider a more general notion of coarse
longitudes.    These are useful for the combinatorics relating the
representation theory of \Stendhal algebras to the representations of
the KLR Yangian algebra $\BK$ we will introduce in Section
\ref{sec:cylindrical}.

We'll define a preorder on the integers by $a\succeq b$ if and only if
$\lfloor \frac{a}{2}\rfloor\geq \lfloor \frac{b}{2}\rfloor$.  This is
coarsening of the standard order, where we also have $2q\succeq
2q+1$ for any integer $q$.  We write $a\approx b$ if $a\preceq b $ and $b\preceq a$ (that
is, $a,b\in \{2p,2p+1\}$ for some integer $p$).

\begin{Definition}
\label{def:coarselongs}
Consider a sequence $\Bi\in I^n$ and a function $\kappa\colon
  [1,\ell]\to [0,n]$ fixing the
  position of the red strands.   A {\bf coarse
   longitude}  compatible with this data is a sequence $\Ba=(a_1,\dots, a_n)\in \Z^n$ which satisfies the conditions of a longitude for the
order $\succeq$, that is:
\begin{enumerate}
\item $i_k$ and $a_k$ have the same parity,
\item $a_{k}\preceq a_{k+1}$, and
\item if  $k\leq  \kappa(p)$ then $a_k< r_p$, and if $k>
  \kappa(p)$ then $a_k\succeq r_p$.
\end{enumerate}
We call $(\Bi,\kappa,\Ba)$ a {\bf coarse longitude triple}.
\end{Definition}
Note that every longitude is a coarse longitude, but not vice versa.  A {\bf coarse metric Stendhal diagram} is a Stendhal diagram together with a choice of coarse longitude triples on the top and bottom of the diagram.  As with the metric KLRW algebra, we define a product on the formal span of coarse metric Stendhal diagrams, requiring that coarse longitudes in addition to labels in $I$ must match, or else we take the product to be zero.
The {\bf coarse metric KRLW algebra} ${\TL}=\TL^\bR$ consists of finite spans of coarse metric Stendhal diagrams, modulo the local relations
  (\ref{first-QH}--\ref{cost}).

 For an example of an element of $\TL$ which is not in $\tmetric$,
 consider the diagram of Example \ref{ex:metricdiag} but with the
 bottom longitudes of the black strands changed to $-3,-2,5$:
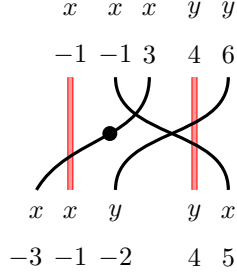
\[ \tikz[very thick,scale=1.5,baseline]{
    \draw[wei] (-.5,-.5) -- (-.5,.5);
    \draw[wei] (.6,-.5) -- (.6,.5);
    \draw (-.8,-.5) to [out=60,in=-90] (.2,.5);
    \draw (-.1,-.5) to [out=90,in=-90] (.9,.5);
    \draw (.9,-.5) to [out=90,in=-90] (-.1,.5);

  \draw (-.8,-.5)+(0,-.2) node {\small$x$};
  \draw (-.5,-.5)+(0,-.2) node {\small$x$};
  \draw (-.1,-.5)+(0,-.2) node {\small$y$};
  \draw (.6,-.5)+(0,-.2) node {\small$y$};
  \draw (.9,-.5)+(0,-.2) node {\small$x$};

  \draw (-.8,-.5)+(-.1,-.6) node {\small$-3$};
  \draw (-.5,-.5)+(0,-.6) node {\small$-1$};
  \draw (-.1,-.5)+(0,-.6) node {\small$-2$};
  \draw (.6,-.5)+(0,-.6) node {\small$4$};
  \draw (.9,-.5)+(0,-.6) node {\small$5$};

   \draw (.2,.5)+(0,.2) node {\small$3$};
  \draw (-.5,.5)+(0,.2) node {\small$-1$};
  \draw (-.1,.5)+(0,.2) node {\small$-1$};
  \draw (.6,.5)+(0,.2) node {\small$4$};
  \draw (.9,.5)+(0,.2) node {\small$6$};

  \draw (.2,.5)+(0,.6) node {\small$x$};
  \draw (-.5,.5)+(0,.6) node {\small$x$};
  \draw (-.1,.5)+(0,.6) node {\small$x$};
  \draw (.6,.5)+(0,.6) node {\small$y$};
  \draw (.9,.5)+(0,.6) node {\small$y$};

  \draw (-.15,0) node [circle,fill=black,inner sep=2pt]{};
} \]

\subsubsection{Generators}
One reason introducing the coarse metric KLRW algebra is that its generators are easier to describe, as compared to the metric KLRW algebra.

We let $e(\Bi,\kappa,\Ba) \in \TL$ be the straight line diagram with
$(\Bi,\kappa,\Ba)$ the coarse longitude triple assigned to both
the top and bottom of the diagram.  We'll draw these diagrams with the
label $i_k$ and longitude $a_k$ written at the start and end of each
strand.  The element
$y_k(\Bi,\kappa,\Ba)\in \TL$ is obtained from $e(\Bi,\kappa,\Ba)$ by
placing a dot on the $k$th black strand.

We also have generators which cross strands; due to the presence of coarse longitudes we have to be more careful than usual with these.

The crossing $\psi_k$ of two neighbouring black strands is only
defined when $a_k\approx a_{k+1}$. In that case we assign
$(\Bi,\kappa,\Ba)$ as the bottom longitude, and
$(s_k\Bi,\kappa,s_k\Ba)$ as the top longitude.

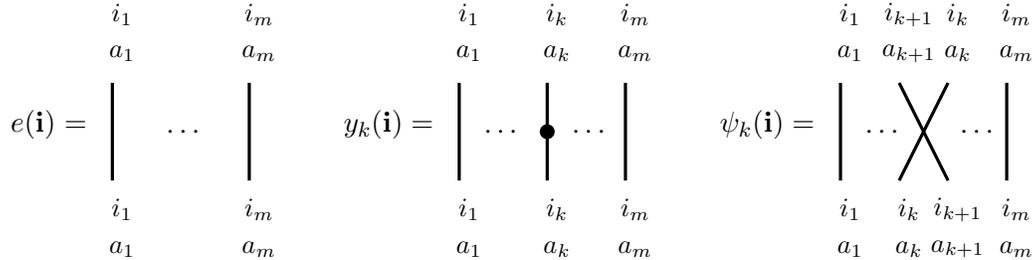
\begin{equation*}
e(\Bi)=\,\begin{tikzpicture}[baseline,scale=1.3]
 \draw[very thick] (-3.4,0) +(0,-.5) -- +(0,.5);
  \draw[very thick] (-2,0) +(0,-.5) -- +(0,.5);
  \draw (-3.4,0) +(0.1,-.8) node {\small$i_1$};
\draw (-3.4,0) +(0.1,-1.2) node {\small$a_1$};
\draw (-3.4,0) +(0.1,1.2) node {\small$i_1$};
\draw (-3.4,0) +(0.1,.8) node {\small$a_1$};
  \draw (-2,0) +(0.1,-.8) node {\small$i_m$};
  \draw (-2,0) +(0.1,-1.2) node {\small$a_m$};
  \draw (-2,0) +(0.1,1.2) node {\small$i_m$};
  \draw (-2,0) +(0.1,.8) node {\small$a_m$};
  \draw (-2.65,0) node {$\cdots$};  \end{tikzpicture}
\qquad  y_k(\Bi)=\,\begin{tikzpicture}[baseline,scale=1.3]
 \draw[very thick] (.3,0) +(0,-.5) -- +(0,.5);
  \draw[very thick] (2,0) +(0,-.5) -- +(0,.5);
  \draw (.3,0) +(0.1,-.8) node {\small$i_1$};
\draw (.3,0) +(0.1,-1.2) node {\small$a_1$};
\draw (.3,0) +(0.1,1.2) node {\small$i_1$};
\draw (.3,0) +(0.1,.8) node {\small$a_1$};
\draw (2,0) +(0.1,-.8) node {\small$i_m$};
 \draw (2,0) +(0.1,-1.2) node {\small$a_m$};
  \draw (2,0) +(0.1,1.2) node {\small$i_m$};
  \draw (2,0) +(0.1,.8) node {\small$a_m$};
  \draw (.75,0) node {$\cdots$};
    \draw[very thick] (1.2,0) +(0,-.5) -- node [midway,circle,fill=black,inner sep=2pt]{} +(0,.5);
     \draw (1.65,0) node {$\cdots$};
      \draw (1.2,0) +(0.1,-.8) node {\small$i_k$};
 \draw (1.2,0) +(0.1,-1.2) node {\small$a_k$};
  \draw (1.2,0) +(0.1,1.2) node {\small$i_k$};
  \draw (1.2,0) +(0.1,.8) node {\small$a_k$}; \end{tikzpicture}
\qquad  \psi_k(\Bi)=\,\begin{tikzpicture}[baseline,scale=1.3]
 \draw[very thick] (4.3,0) +(0,-.5) -- +(0,.5);
  \draw[very thick] (6,0) +(0,-.5) -- +(0,.5);
  \draw (4.3,0) +(0.1,-.8) node {\small$i_1$};
\draw (4.3,0) +(0.1,-1.2) node {\small$a_1$};
\draw (4.3,0) +(0.1,1.2) node {\small$i_1$};
\draw (4.3,0) +(0.1,.8) node {\small$a_1$};
  \draw (6,0) +(0.1,-.8) node {\small$i_m$};
 \draw (6,0) +(0.1,-1.2) node {\small$a_m$};
  \draw (6,0) +(0.1,1.2) node {\small$i_m$};
  \draw (6,0) +(0.1,.8) node {\small$a_m$};
  \draw[very thick] (4.9,0) +(0,-.5) -- +(.5,.5);
  \draw[very thick] (5.4,0) +(0,-.5) -- +(-.5,.5);
\draw (4.9,0) +(0.1,-.8) node {\small$i_k$};
 \draw (4.9,0) +(0.1,-1.2) node {\small$a_k$};
  \draw (5.4,0) +(0.1,1.2) node {\small$i_k$};
  \draw (5.4,0) +(0.1,.8) node {\small$a_k$};
\draw (5.4,0) +(0.1,-.8) node {\small$i_{k+1}$};
 \draw (5.4,0) +(0.1,-1.2) node {\small$a_{k+1}$};
  \draw (4.9,0) +(0.1,1.2) node {\small$i_{k+1}$};
  \draw (4.9,0) +(0.1,.8) node {\small$a_{k+1}$};
\draw (4.75,0) node {$\cdots$};
\draw (5.72,0) node {$\cdots$};
\end{tikzpicture}
\end{equation*}

The red-black crossings are defined in two cases.  In the first case,
the rightmost black strand with longitude  $a_k\approx 2q$ for fixed
$q\in \Z$ increases
  longitude by $2$,  and moves rightward by crossing all red strands with longitude $\preceq 2q+2$. (Of course, there may be no such red strands, in which case this generator doesn't contain any crossings and just changes the longitude.)  All other longitudes remain the same.  We denote this generator by $\phi^+_k=\phi^+_k(\Bi,\kappa,\Ba)$, where $(\Bi,\kappa,\Ba)$ is the coarse longitude triple on the bottom of the diagram.
  \[ \phi^+_k=\,\tikz[very thick,scale=1.5,baseline]{\draw[wei] (-.3,-.5) -- (-.3,.5);
    \draw[wei] (.3,-.5) -- (.3,.5); \node at (0,.3){$\cdots$}; \node
    at (0,-.3){$\cdots$}; \draw (-2,-.5) -- (-2,.5); \draw (-1.3,-.5) --
    (-1.3,.5);\node at (-1.65,0){$\cdots$}; \draw (-.6,-.5) to
    [out=45,in=-135] (.6,.5); \draw (2,-.5) -- (2,.5); \draw (1.3,-.5)
    -- (1.3,.5);\node at (1.65,0){$\cdots$};
  \draw (-2,0) +(0.1,-.8) node {\small$i_1$};
\draw (-2,0) +(0.1,-1.2) node {\small$a_1$};
\draw (-2,0) +(0.1,1.2) node {\small$i_1$};
\draw (-2,0) +(0.1,.8) node {\small$a_1$};
  \draw (2,0) +(0.1,-.8) node {\small$i_m$};
 \draw (2,0) +(0.1,-1.2) node {\small$a_m$};
  \draw (2,0) +(0.1,1.2) node {\small$i_m$};
  \draw (2,0) +(0.1,.8) node {\small$a_m$};
    \draw (-1.3,0) +(0.1,-.8) node {\small$i_{k-1}$};
 \draw (-1.3,0) +(0.1,-1.2) node {\small$a_{k-1}$};
  \draw (-1.3,0) +(0.1,1.2) node {\small$i_{k-1}$};
  \draw (-1.3,0) +(0.1,.8) node {\small$a_{k-1}$};
    \draw (1.3,0) +(0.1,-.8) node {\small$i_{k+1}$};
 \draw (1.3,0) +(0.1,-1.2) node {\small$a_{k+1}$};
  \draw (1.3,0) +(0.1,1.2) node {\small$i_{k+1}$};
  \draw (1.3,0) +(0.1,.8) node {\small$a_{k+1}$};
\draw (-.6,0) +(0.1,-.8) node {\small$i_k$};
 \draw (-.6,0) +(0.1,-1.2) node {\small$a_k$};
  \draw (.6,0) +(0.1,1.2) node {\small$i_k$};
  \draw (.6,0) +(0.1,.8) node {\small$a_k+2$};
} \]

  The second case is the mirror image - it consists of the diagram
  where the leftmost strand with longitude $a_k\approx 2q$ decreases
  longitude by $2$, crossing all
red strands with longitude $\approx 2q$, and all other longitudes
remain the same.  We denote this generator by
$\phi^-_k=\phi^-_k(\Bi,\kappa,\Ba)$, where $(\Bi,\kappa,\Ba)$ is the
coarse longitude triple on the bottom of the diagram.
\[  \phi^-_k=\,\tikz[very thick,scale=1.5,baseline]{\draw[wei] (-.3,-.5)
    -- (-.3,.5); \draw[wei] (.3,-.5) -- (.3,.5); \node at
    (0,.3){$\cdots$}; \node at (0,-.3){$\cdots$}; \draw (-2,-.5) --
    (-2,.5); \draw (-1.3,-.5) -- (-1.3,.5);\node at
    (-1.65,0){$\cdots$}; \draw (.6,-.5) to [out=135,in=-45] (-.6,.5);
    \draw (2,-.5) -- (2,.5); \draw (1.3,-.5) -- (1.3,.5);\node at
    (1.65,0){$\cdots$};
  \draw (-2,0) +(0.1,-.8) node {\small$i_1$};
\draw (-2,0) +(0.1,-1.2) node {\small$a_1$};
\draw (-2,0) +(0.1,1.2) node {\small$i_1$};
\draw (-2,0) +(0.1,.8) node {\small$a_1$};
  \draw (2,0) +(0.1,-.8) node {\small$i_m$};
 \draw (2,0) +(0.1,-1.2) node {\small$a_m$};
  \draw (2,0) +(0.1,1.2) node {\small$i_m$};
  \draw (2,0) +(0.1,.8) node {\small$a_m$};
    \draw (-1.3,0) +(0.1,-.8) node {\small$i_{k-1}$};
 \draw (-1.3,0) +(0.1,-1.2) node {\small$a_{k-1}$};
  \draw (-1.3,0) +(0.1,1.2) node {\small$i_{k-1}$};
  \draw (-1.3,0) +(0.1,.8) node {\small$a_{k-1}$};
    \draw (1.3,0) +(0.1,-.8) node {\small$i_{k+1}$};
 \draw (1.3,0) +(0.1,-1.2) node {\small$a_{k+1}$};
  \draw (1.3,0) +(0.1,1.2) node {\small$i_{k+1}$};
  \draw (1.3,0) +(0.1,.8) node {\small$a_{k+1}$};
\draw (.6,0) +(0.1,-.8) node {\small$i_k$};
 \draw (.6,0) +(0.1,-1.2) node {\small$a_k$};
  \draw (-.6,0) +(0.1,1.2) node {\small$i_k$};
  \draw (-.6,0) +(0.1,.8) node {\small${a_k-2}$};
} \]

\subsubsection{Polynomial representation}
Later in the paper, we will need to use the polynomial representation of this algebra, which is a little complicated since we need many sets of variables.

The polynomial representation of $\TL$ is given by
\begin{equation*}
\label{eq:PolTL}
\PolKLR_{\mathcal{L}}=
\bigoplus_{(\Bi,\Ba,\kappa)} \PolKLR(\Bi,\Ba,\kappa), \quad \text{where}\quad \PolKLR(\Bi,\Ba,\kappa)=\C[Y_1(\Bi,\Ba,\kappa),\dots Y_n(\Bi,\Ba,\kappa)]
\end{equation*}
where the sum is over all coarse longitude triples.  The action is given by the formulae (\ref{eq:T-action1}--\ref{eq:T-action2}).  More precisely, $x \in e(\Bi',\Ba',\kappa')\TL e(\Bi,\Ba,\kappa)$ defines an operator
$$
x:\PolKLR(\Bi,\Ba,\kappa) \to \PolKLR(\Bi',\Ba',\kappa')
$$
which acts by ignoring longitudes, acting by (\ref{eq:T-action1}--\ref{eq:T-action2}), and then replacing the correct longitudes in the variables.
Since $\Pol$ is a faithful $\tilde{T}^\blam$-module for any $\blam$,
it follows  that $\PolKLR_{\mathcal{L}}$ is a faithful $\TL$-module.
In fact, the modules $\PolKLR_{\mathcal{L}}$ and $\Pol$ are identified
under the Morita equivalence below.


\subsubsection{An equivalence}
It turns out that (unlike the metric KLRW algebra or parity KLRW algebra) the coarse metric KLRW algebra is always Morita equivalent to a usual KLRW algebra, for a sequence $ \boldsymbol{\lambda}$  which depends on $ \bR $.

Recall that to $\bR$ we associate multiplicity functions $\rho_i:\Z\to\N$ (cf. Section \ref{section:TR}).  For $q\in\Z$ define $\lambda^{(q)}=\sum_i\rho_i(q)\varpi_i$, and set $\boldsymbol{\lambda}=(...,\lambda^{(q-1)},\lambda^{(q)},...)$.  Note that generically, $ \boldsymbol{\lambda} = \varpi_\bR $, but in general it is different.

There is a natural inclusion
$\tilde{T}^{\blam}\to\tilde{T}$, which splits the red strand labeled $\lambda^{(q)}$ into $\sum_i\rho_i(q)$ red strands with labels from left to right given by the sequence (\ref{eq:varpiq}).  Observe that the idempotents in the image of this map will never have a black strand between two red strands
\begin{itemize}
\item where both red strands have longitude $2q$, or
\item where both red strands have longitude $2q+1$, or
\item where the left red strand has longitude $2q$ and right one has longitude $2q+1$.
\end{itemize}

\begin{Proposition}
\label{prop:coarselongs}
  The algebras $\TL$ and
  $\tilde{T}^{\boldsymbol{\lambda}}$ are Morita equivalent.
\end{Proposition}
\begin{proof}
As in the proof of Lemma \ref{lem:Morita}, our desired result reduces to showing
that the idempotents that carry coarse longitudes are precisely
those isomorphic to ones in the image of $\tilde{T}^{\boldsymbol{\lambda}}$.

Given an idempotent with a coarse longitude, the only obstruction to
 it lying in $\tilde{T}^{\boldsymbol{\lambda}}$ is if one of the three cases above occurs.  First notice that the first one cannot occur since that would not support a coarse longitude.  In the latter two cases, there must be a black strand of longitude $2q$ in between the red strands. Crossing this black
 strand over the right red strand is irrelevant since the strands have
 opposite parities.  Thus, this idempotent is isomorphic to one in
 $\tilde{T}^{\boldsymbol{\lambda}}$.

On the other hand, given an idempotent in
$\tilde{T}^{\boldsymbol{\lambda}}$, we'll show that it carries a coarse longitude.  Consider a black strand between two red strands with longitudes $r_p$ and $r_{p+1}$.  Note that $|r_p-r_{p+1}|\neq 0$.  If $|r_p-r_{p+1}|=1$ then we must have $r_{p}=2q-1$ and $r_{p+1}=2q$ for some $q$.  In this case we let the longitude of the black strand be whichever
element of  $\{2q-2,2q-1\}$ has the correct parity.  If $|r_p-r_{p+1}|>1$ then we can just choose a longitude $a$ of the correct parity such that $r_p \leq a \leq r_{p+1}$.  Clearly it's possible to do this for all black strands so that the coarse longitude conditions are met.
\end{proof}

\excise{It's easy to check that up to irrelevant crossings, any idempotent where there is no black strand
between red strands with $r_i'=r_j'$ has a coarse longitude: we
simply label each black strand so that its coarse longitude is
equal to that of the closest congruent strand to its left (which is
thus
the same as the coarse longitude of the closest congruent red
strand); if there is a black strand left of all reds; our
observation above shows that this is the same as
$\tilde{T}^{\boldsymbol{\lambda}}$ where
$\boldsymbol{\lambda}=(\la_1,\dots, \la_{\# R'})$ is given by $\la_1$
being the sum of fundamental weights attached to the smallest element
of $R'$, and $\la_2$ the sum for the second smallest, etc. where
  }
\excise{
Consider the deformation of ${\mathbb T}$ where we adjoin a formal variable $h$ a formal variable and deform the relation \eqref{cost} by adding $+\nu(r)h$ for the red line at $x=r$ for some chosen bijection of $\nu\colon \bR\to [1,\ell]$.  For generic $h$, this is isomorphic to  $T^{\bR_h}$ where we replace $r\in R_i$ with $r+\nu(r)h$.
If $y_{\bR}z_{\bS}^{-1}$ lies in the monomial crystal, then we can choose a map $\delta\colon \bS\to \bR$ such that $y_{i,r}z_{\delta^{-1}(r)}^{-1}$ lies in the fundamental monomial crystal for each $r\in S_i$, and we let $\bS_h$ be the deformation where we replace $s$ by $s+\nu\circ \delta(s(h))$.
\begin{Lemma}\label{lem:deform}
  The $\C[h]$-module ${\mathbb T}$ is a flat deformation of $Ty$ and $M(\bS)$ deforms flatly over this algebra, with generic fiber isomorphic to $\oplus_{\delta} M(\bS_h)$.
\end{Lemma}}


\section{Truncated shifted Yangians and KLR Yangian algebras}

Throughout this section we let $\lambda$ be a dominant weight, $\mu$ a weight with $\lambda \geq \mu$, and let $\bR$ be a set of parameters $\bR$ of weight $\lambda$.

\subsection{The shifted Yangian}
\label{subsection: The shifted Yangian}

Following \cite[Appendix B]{BFNslices}, we now define a variation on the shifted Yangian $Y_\mu$, but without the assumption that $\fg$ is finite type.


\begin{Definition}
\label{def: shifted Yangian}
The {\bf shifted Yangian} $Y_\mu = Y_\mu(\fg, \bR)$ is defined to be the $ \C $--algebra with generators $ E_i^{(q)}, F_i^{(q)}, A_i^{(q)} $ for $ i\in I$, $ q > 0 $, with relations
\begin{align*}
[A_i^{(p)}, A_j^{(q)}] &= 0,  \\
[E_i^{(p)}, F_j^{(q)}] &=  2\delta_{ij} H_i^{(p+q-1)}, \\
[A_i^{(p+1)},E_j^{(q)}] - [A_i^{(p)}, E_j^{(q+1)}] &= - \delta_{i,j} A_i^{(p)} E_j^{(q)} , \\
[A_i^{(p+1)},F_j^{(q)}] - [A_i^{(p)}, F_j^{(q+1)}] &= \delta_{i,j} F_j^{(q)} A_i^{(p)}, \\
[E_i^{(p+1)}, E_j^{(q)}] - [E_i^{(p)}, E_j^{(q+1)}] &= \alpha_i \cdot \alpha_j (E_i^{(p)} E_j^{(q)} + E_j^{(q)} E_i^{(p)}), \\
[F_i^{(p+1)}, F_j^{(q)}] - [F_i^{(p)}, F_j^{(q+1)}] &= -\alpha_i \cdot \alpha_j (F_i^{(p)} F_j^{(q)} + F_j^{(p)} F_i^{(q)}),\\
i \neq j, N = 1 - \alpha_i \cdot \alpha_j \Rightarrow
\operatorname{sym} &[E_i^{(p_1)}, [E_i^{(p_2)}, \cdots [E_i^{(p_N)}, E_j^{(q)}]\cdots]] = 0, \\
i \neq j, N = 1 - \alpha_i \cdot \alpha_j \Rightarrow
\operatorname{sym} &[F_i^{(p_1)}, [F_i^{(p_2)}, \cdots [F_i^{(p_N)}, F_j^{(q)}]\cdots]] = 0.
\end{align*}
Here we define elements $H_i^{(q)} \in Y_\mu(\bR)$ by the rule
\begin{equation} \label{eq: H and A}
H_i(u) = p_i(u) \frac{\prod_{j \con i} (u-1)^{m_j}}{u^{m_i} (u-2)^{m_i}} \frac{\prod_{j \con i} A_j(u-1 )}{A_i(u) A_i(u-2)},
\end{equation}
where $p_i(u) = \prod_{c \in R_i} (u - c)$ and
$$ 
H_i(u) = u^{\langle\mu, \alpha_i^\vee\rangle} + \sum_{q>-\langle \mu, \alpha_i^\vee\rangle} H_i^{(q)} u^{-q}, \quad A_i(u) = 1 + \sum_{p>0} A_i^{(p)} u^{-p}
$$
\end{Definition}

For simplicity we will write $Y_\mu$, but we note again that this definition depends in general on the choice of $\mu$ and $\bR$ (and so implicitly on $\lambda$), as opposed to the more usual dependence solely on $\mu$ in \cite[Definition B.2]{BFNslices}.    For this reason, as well as our choice of generators, our definition may seem unfamiliar to the reader.  First, let us note:
\begin{Lemma}
\label{lemma: H and A}
In $Y_\mu$, there are relations
\begin{align*}
[H_i^{(p+1)}, E_j^{(q)}] - [H_i^{(p)}, E_j^{(q+1)}] &= \alpha_i \cdot \alpha_j (H_i^{(p)} E_j^{(q)} + E_j^{(q)} H_i^{(p)}), \\
[H_i^{(p+1)}, F_j^{(q)}] - [H_i^{(p)}, F_j^{(q+1)}] &= -\alpha_i \cdot \alpha_j (H_i^{(p)} F_j^{(q)} + F_j^{(p)} H_i^{(q)})
\end{align*}
\end{Lemma}
\begin{proof}
Similar to part of the proof of \cite[Theorem 6.6]{FT1}.
\end{proof}

The definition of shifted Yangians given in \cite[Definition B.2]{BFNslices} also makes sense for arbitrary simply-laced $\fg$.  That is, we could take generators $H_i^{(q)}$ instead of $A_i^{(p)}$, with the relations from the lemma.  Denoting the resulting algebra by $Y_\mu^{\text{BFN}}$, it follows from the lemma that there is a homomorphism
$$
Y_\mu^{\text{BFN}} \longrightarrow Y_\mu
$$
defined by sending the generators $E_i^{(p)}, F_i^{(p)}, H_i^{(q)}$ to the same-named elements.   In general this map is neither injective nor surjective, because of a discrepancy between respective commutative subalgebras:  the $H_i^{(r)}$ are like simple coroots for $\fg$, while the $A_i^{(q)}$ are like (negatives of) fundamental coweights. Since the Cartan matrix need not be invertible, we cannot necessarily express the fundamental coweights in terms of coroots, and similarly we may not be able to express the $A_i^{(p)}$ in terms of $H_i^{(q)}$.  The coroots written na\"ively in terms of fundamental coweights may not be linearly independent, and likewise our $H_i^{(q)} \in Y_\mu$ may not be algebraically independent. That said, we believe our definition of $Y_\mu$ is a reasonable one because it makes contact with quantized Coulomb branches in general type, as we will see below.  

In finite type our definition is equivalent to the usual one:
\begin{Lemma}
If $\fg$ is of finite ADE type, then $Y_\mu^{\text{BFN}} \stackrel{\sim}{\longrightarrow }Y_\mu$.
\end{Lemma}
\begin{proof}
This follows from results of \cite{GKLO}: the equation (\ref{eq: H and A}) can be inverted to express the generators $A_i^{(p)}$ in terms of the $H_j^{(q)}$, and the relations in Lemma \ref{lemma: H and A} imply those from  Definition \ref{def: shifted Yangian}.
\end{proof}

\begin{Remark}
Recall that we take $\fg = [\fg,\fg]$ to be the derived subalgebra of a Kac-Moody Lie algebra, and that the full Kac-Moody algebra has a larger Cartan.  It seems likely that by appropriately incorporating the full Cartan into the definition of $Y_\mu$ (c.f.~\cite[Definition 2.1]{GNW}), we could avoid having to choose between generators $A_i^{(p)}$ and $H_i^{(q)}$.
\end{Remark}

\begin{Remark}
\label{rmk: h=2conventions}
Our conventions regarding Yangians differ slightly from the literature in another way: simply put, we have taken $\hbar = 2$ instead of the more standard choice of $\hbar = 1$.  This is not an essential difference, as is well-known, since these choices are isomorphic via a simple rescaling of the generators: roughly speaking, we map $X_i^{(r)} \mapsto (\tfrac{1}{2})^r X_i^{(r)}$ for $X = E, H, F$.
\end{Remark}
%
%

\excise{
  Choose a total order on the set of PBW variables
\begin{equation} \label{total-order-generators}
\left\{ E_\beta^{(q)} : \beta\in \Delta^+, q>0\right\} \cup \left\{ F_\beta^{(q)} : \beta\in \Delta^+, q>0 \right\} \cup \left\{ H_i^{(p)} : i\in I, p> -\langle \mu, \alpha_i\rangle \right\}
\end{equation}

For simplicity we will assume that we have chosen a block order with respect to the three subsets above, i.e. ordered monomials have the form $EFH$.  The following result is from \cite{FKPRW}.
\begin{Proposition} \label{PBW monomials span}
Ordered monomials in the PBW variables give a basis for $Y_\mu$.
\end{Proposition}

}

\subsubsection{Action on symmetric polynomials}
\label{subsubsection: Action on symmetric polynomials}

Since  $\lambda \geq \mu$, we may write $\lambda - \mu = \sum_i m_i \alpha_i$ with all $m_i\geq 0$.  Consider the polynomial ring
$$ P = \C[z_{i,j} : i \in I, 1 \le j \le m_i ] $$
Let $ \Sigma = \prod_{i \in I} \Sigma_{m_i} $ be the product of symmetric groups, acting on $P$ by permuting each set of variables.  We will be interested in the $\Sigma$-invariant polynomials $P^\Sigma$.
Denote
$$W_i (u) = \prod_{r=1}^{m_i}(u - z_{i,r}) \ \ \text{ and } \ \ W_{i,r}(u) = \prod_{\substack{s=1 \\ s\neq r}}^{m_i} (u - z_{i,s}).$$
Let $ \beta_{i,r} $ denote the algebra morphism $ P \rightarrow P $ that shifts the variable $ z_{i,r} $ by 2, i.e. $ \beta_{i,r}(z_{j,s}) = z_{j,s} + 2\delta_{ij} \delta_{rs}. $

\begin{Theorem} \label{th:filGKLO} There is an action of $Y_\mu $ on $P^\Sigma$, defined by \newseq
\begin{align*}
A_i(u) & \mapsto z^{-m_i} W_i(u), \subeqn \label{eq:A-action}\\
E_i(u) & \mapsto -\sum_{r=1}^{m_i}  \frac{\prod_{j\rightarrow i} W_j(z_{i,r} - 1)}{(u-z_{i,r}) W_{i,r}(z_{i,r})} \beta_{i,r}^{-1}, \subeqn\label{eq:E-action} \\
F_i(u) & \mapsto \sum_{r=1}^{m_i} \frac{p_i(z_{i,r}+2) \prod_{j\leftarrow i} W_j(z_{i,r} + 1 )}{(u-z_{i,r} - 2) W_{i,r}(z_{i,r})} \beta_{i,r}.\subeqn \label{eq:F-action}
\end{align*}
\end{Theorem}

This theorem is inspired by the results of \cite{GKLO}. As in \cite[Theorem B.15]{BFNslices}, we can see that the above formulas define a homomorphism
\begin{equation}
\label{eq: bfnappendix map}
\Phi_\mu^\lambda: Y_\mu\rightarrow \widetilde{\mathcal{A}}
\end{equation}
to a suitable localization $\widetilde{\mathcal{A}}$ of the algebra of difference operators on $P$. Thus the above map certainly defines an action of $Y_\mu$ on the field of fractions $\operatorname{Frac}(P)$.  To prove the theorem, it remains to show that the action actually preserves $P^\Sigma$.  This can be proven geometrically using results comparing $ Y_\mu $ to the quantized Coulomb branch algebra, and identifying $ P^\Sigma $ with the equivariant cohomology of a point.  It can also be proven algebraically, as we do now.

Let us introduce the following notation.  Let $R$ be a commutative domain over $\C$, and consider the polynomial ring $ R[z_1,\ldots,z_m]$ as well as its field of fractions $R(z_1,\ldots,z_m)$.  For $1\leq r \leq m$, we define an algebra endomorphism $\beta_r  \in \End_R\left( R(z_1,\ldots,z_m)\right)$ by $\beta_r(z_s) = z_s + 2\delta_{r,s}$.  For $1\leq r < m$, we define the divided difference operator $\partial_r \in \End_R \left( R(z_1,\ldots,z_m) \right)$ by $\partial_r = \frac{1}{z_{r+1}-z_r}(s_r - 1)$ where $s_r$ denotes the transposition permuting the variables $z_r, z_{r+1}$ and fixing all others.

\begin{Lemma}
\label{lemma:filGKLO}
With notation as above, fix a polynomial $c(t) \in R[t]$.  Consider the following two endomorphisms of $R(z_1,\ldots,z_m)$:
\begin{align}
& \sum_{r=1}^m \frac{c(z_r)}{\prod_{1\leq s \leq m, s \neq r}(z_r - z_s)} \beta_r, \label{eq:operatorfilGKLO1} \\
& \partial_{m-1} \cdots \partial_1 c(z_1) \beta_1 \label{eq:operatorfilGKLO2}
\end{align}
Applied to any element of $ R(z_1,\ldots,z_m)^{\Sigma_m}$, these two operators agree.  In particular, both preserve $R[z_1,\ldots,z_m]^{\Sigma_m}$.
\end{Lemma}

\begin{proof}
We prove the first claim by induction on $m$, for any domain of coefficients $R$. The case $m=1$ is trivial, so let us assume the claim is true for $m$.  In particular, letting $R' = R(z_{m+1})$ we can think of any $f \in R(z_1,\ldots, z_m, z_{m+1})^{\Sigma_{m+1}} \subset R'(z_1,\ldots,z_m)^{\Sigma_m}$.  Then by our inductive assumption,
\begin{align*}
\partial_m \partial_{m-1} \cdots \partial_1 c(z_1) \beta_1 \cdot f &=  \partial_m \cdot ( \partial_{m-1}\cdots \partial_1 c(z_1) \beta_1 \cdot f ) \\
&= \partial_m \cdot \sum_{r=1}^m \frac{c(z_r)}{\prod_{1\leq s \leq m, s \neq r}(z_r - z_s)} f(z_1,\ldots, z_r+2, \ldots, z_m, z_{m+1})
\end{align*}
Using the fact that $f$ is symmetric under $\Sigma_{m+1}$, applying $\partial_m$ to a summand where $1\leq r<m$ we get
$$ \frac{1}{z_{m+1} - z_m} \left( \frac{1}{z_r - z_{m+1}} - \frac{1}{z_r- z_m} \right) \frac{c(z_r)}{\prod_{1\leq s < m, s\neq r} (z_r - z_s)} f(z_1,\ldots, z_r +2, \ldots, z_{m+1}) $$
$$ = \frac{c(z_r)}{\prod_{1\leq s \leq m+1, s\neq r} (z_r - z_s)} f(z_1,\ldots, z_r+2,\ldots, z_{m+1}) $$
Next, applying $\partial_m$ to the summand where $r = m$, we get
\begin{align*}
&\frac{c(z_{m+1})}{(z_{m+1}-z_m) \prod_{1\leq s < m} (z_{m+1} - z_s)} f(z_1,\ldots, z_{m+1}+2, z_m) \\
+& \frac{c(z_m)}{(z_m - z_{m+1}) \prod_{1\leq s < m}(z_m - z_s)} f(z_1,\ldots, z_m + 2, z_{m+1})
\end{align*}
Putting all summands together, this proves the claim.

For the second claim, note on the one hand that the operator (\ref{eq:operatorfilGKLO1}) preserves $R(z_1,\ldots,z_m)^{\Sigma_m}$, as it is equivariant for the action of $\Sigma_m$.  On the other hand, the operator (\ref{eq:operatorfilGKLO2}) preserves $R[z_1,\ldots,z_m]$ since this is true of divided difference operators. Together with the first claim, it follows that these operators map $R[z_1,\ldots, z_m]^{\Sigma_m}$ into
$$ R(z_1,\ldots,z_m)^{\Sigma_m} \cap R[z_1,\ldots,z_m] = R[z_1,\ldots, z_m]^{\Sigma_m}$$
\end{proof}

\begin{proof}[Proof of Theorem \ref{th:filGKLO}]
Clearly the image of (\ref{eq:A-action}) preserves $P^\Sigma$, so let us prove that (\ref{eq:F-action}) does as well (the case (\ref{eq:E-action}) being similar).  This follows from Lemma \ref{lemma:filGKLO}, applied to
$$ R = \C[z_{j,s} : j\neq i, 1\leq j \leq m_j]^{\prod_{j\neq i} \Sigma_{m_j}} $$
$$ c(t) = (t+2)^{r-1} p_i(t+2) \prod_{j\leftarrow i}W_j(t+1) \in R[t]$$
In this case the operator (\ref{eq:operatorfilGKLO1}) from the lemma agrees with the action of $F_i^{(r)}$ from (\ref{eq:F-action}) (i.e. the coefficient of $u^{-r}$, under expansion at $u= \infty$).

\end{proof}

\subsection{The truncated shifted Yangian}
\label{subsection: The truncated shifted Yangian}
We now define the truncated shifted Yangian $ Y^\lambda_\mu=Y^\lambda_\mu(\bR)$.  By Theorem \ref{th:filGKLO}, we have a homomorphism
$$ Y_\mu \longrightarrow \End_\C( P^\Sigma) $$

\begin{Definition}
\label{def:Ymula}
The {\bf truncated shifted Yangian} $ Y^\lambda_\mu = Y^\lambda_\mu(\bR) $ is the image of $ Y_\mu $ in $ \End_\C(P^\Sigma) $.
\end{Definition}

From the definition, $Y_\mu^\lambda$ has a faithful action on $P^\Sigma$.  Note that by (\ref{eq:A-action}), we see that $ A_i^{(s)} = 0 $ in $ Y^\la_\mu $ for $ s  > m_i$.   We can identify the image $ \C[A_i^{(s)}]  \subset Y^\lambda_\mu$ with $P^\Sigma$.

When $\fg$ is finite type, the definition of $Y_\mu^\lambda$ first appeared in \cite[Section 4C]{KWWY} in the case when $\mu$ is dominant, and in \cite[Appendix B]{BFNslices} for all $\mu$.  Our present definition of $Y_\mu^\lambda$ is a straightforward generalization to arbitrary simply-laced Kac-Moody type.

\begin{Remark}
More precisely, in both \cite[Section 4C]{KWWY}, \cite[B(viii)]{BFNslices}, $Y_\mu^\lambda$ is defined as the image $\operatorname{Im}(\Phi_\mu^\lambda) \subset \widetilde{\mathcal{A}}$ of the map $\Phi_\mu^\lambda: Y_\mu \rightarrow \widetilde{\mathcal{A}}$ into an algebra of difference operators, as in (\ref{eq: bfnappendix map}).  But this definition is isomorphic: the action of $Y_\mu$ on $P^\Sigma$ factors through $\operatorname{Im}(\Phi_\mu^\lambda)$, and in fact there is a diagram
$$
\begin{tikzcd}
Y_\mu \ar[d, two heads] \ar[r, two heads] & Y^\lambda_\mu \ \subset \  \End_\C( P^\Sigma) \\
\operatorname{Im}(\Phi_\mu^\lambda) \ar[ur,sloped,pos=0.6,"\sim", end anchor = south west] &
\end{tikzcd}
$$
Indeed, the action of $\operatorname{Im}(\Phi_\mu^\lambda)$ on $P^\Sigma$ is faithful, since difference operators act faithfully on polynomials.
\end{Remark}
\begin{Remark}
Yet another variation on the definition of truncated shifted Yangian appears in \cite{KTWWY}, where it is defined instead as the quotient
$$Y_\mu / \langle A_i^{(r)} :  i \in I, r> m_i \rangle $$
This algebra surjects onto $Y_\mu^\lambda$ (as defined in the present paper), and we conjecture that this surjection is an isomorphism, c.f. \cite[Theorem 4.10]{KWWY}.  This conjecture is proven in finite type A in \cite{KMWY}.
\end{Remark}

Our interest in $Y_\mu^\lambda$ comes from deformation quantization. When $\fg$ is finite type there is a filtration on $Y_\mu$, which is defined explicitly in terms of a PBW basis, see \cite[Appendix B(i)]{BFNslices} or \cite[Section 5.4]{FKPRW}.  Consider the corresponding quotient filtration on $Y_\mu^\lambda$. When $\mu$ is dominant, in \cite[Theorem 4.8]{KWWY} it was shown that $Y_\mu^\lambda$ quantizes a scheme supported on the Poisson variety $ \Gr^\lambda_\mu$, a transverse slice to a spherical Schubert variety in the affine Grassmannian.  It was conjectured that $Y_\mu^\lambda$ quantizes precisely $\Gr^\lambda_\mu$, i.e. that
$$
\operatorname{gr} Y_\mu^\lambda \cong \C[ \Gr^\lambda_\mu]
$$
For $\fg$ finite type and $\mu$ dominant, this conjecture was established by \cite[Corollary B.28]{BFNslices} using the theory of Coulomb branches. 

\begin{Remark}
When $\fg$ is not finite type, we do not know whether an analogous filtration on $Y_\mu$ exists, although this seems very reasonable.  The existence of this filtration in finite type comes from Drinfeld-Gavarini duality (alias {\em quantum duality principle}), which ultimately uses the Hopf algebra structure on the Yangian, see \cite[Section 3C]{KWWY}, \cite[Appendix A]{FT2}.
\end{Remark}

Coulomb branches are developed mathematically in \cite{BFN}, where an affine variety $\mathcal{M}_ C = \mathcal{M}_C(H, V)$ called the Coulomb branch, is associated to each pair $(H, V)$ consisting of a reductive group $H$ and a representation $V$ of $H$. Its coordinate ring admits a natural deformation quantization $\mathcal{A}_\hbar = \mathcal{A}_\hbar(H, V)$ over $\C[\hbar]$; in particular $\mathcal{M}_C(H, V)$ has a Poisson structure.  The case of greatest present interest to us is when the pair $(H, V)$ is
$$H = \prod_{i\in I} GL(m_i), \qquad V = \bigoplus_{i\rightarrow j} \operatorname{Hom}(\C^{m_i}, \C^{m_j}) \oplus \bigoplus_i \operatorname{Hom}(\C^{\lambda_i}, \C^{m_i}) $$
By \cite[Theorem 3.10]{BFNslices}, if $\fg$ is finite type, then for $(H, V)$ as above $\mathcal{M}_C$ is a {\it generalized} slice $\overline{\mathcal{W}}_\mu^\lambda$.  Note that when $\mu$ is dominant there is an isomorphism $\overline{\mathcal{W}}_\mu^\lambda \cong \Gr_\mu^\lambda$.

It is not a priori clear what Coulomb branches have to do with Yangians.  However, a connection is established by \cite[Appendix B]{BFNslices}: there is a homomorphism of algebras $Y_\mu \rightarrow \mathcal{A}_{\hbar=2}$, for any simply-laced Kac-Moody $\fg$. This map comes from the homomorphism (\ref{eq: bfnappendix map}), and in particular there is an embedding of algebras
$$Y_\mu^\lambda \hookrightarrow \mathcal{A}_{\hbar=2}$$
When $\fg$ is finite type, \cite[Appendix B]{BFNslices} shows that this map respects filtrations, and that it is an isomorphism of filtered algebras when $\mu$ is dominant.  These results have been extended by the fourth author:

\begin{Theorem}[\mbox{\cite[Theorem A]{Weekes}}]
\label{thm: truncated shifted Yangian equals Coulomb}\mbox{}
Let $\fg$ be any simply-laced Kac-Moody type.  For any $\lambda, \mu$ and $\bR$:
\begin{enumerate}
\item[(a)]  There is an isomorphism of algebras
$$ Y_\mu^\lambda \stackrel{\sim}{\rightarrow} \mathcal{A}_{\hbar = 2} $$

\item[(b)] If $\fg$ is finite ADE type, then the above is an isomorphism of filtered algebras.   In particular, $\operatorname{gr} Y_\mu^\lambda \cong \C[\overline{\mathcal{W}}_\mu^\lambda]$.
\end{enumerate}
\end{Theorem}

\begin{Remark}
There is a subtlety which we have neglected, namely the role of the parameters $\mathbf{R}$.  In terms of Coulomb branches, by working equivariantly for the flavour symmetry group $G_F = \prod_{i\in I} GL_{\lambda_i}$, we can deform $\mathcal{A}_\hbar$ to an algebra $\mathcal{A}_{\hbar, F}$ over $H_{G_F}^\ast(pt)$.  A choice of parameters $\bR$ defines a homomorphism
$$ H_{G_F\times \C^\times}^\ast(pt)\cong \C[\hbar] \otimes \bigotimes_{i\in I} H_{GL_{\lambda_i}}^\ast(pt) \  \longrightarrow \C, $$
which sends $\hbar \rightarrow 2$ and specializes $H_{GL_{\lambda_i}}^\ast(pt)\rightarrow \C$ according to the multiset $R_i$. Then more precisely, Theorem \ref{thm: truncated shifted Yangian equals Coulomb} states that there is an isomorphism
$$ Y_\mu^\lambda(\bR) \stackrel{\sim}{\longrightarrow} \mathcal{A}_{\hbar, F} \otimes_{H_{G_F\times\C^\times}^\ast(pt)} \C, $$
and that this is an isomorphism of filtered algebras in finite ADE type.

\end{Remark}

\subsection{The flag Yangian}
\label{sec:flagYang}
We will now study a matrix algebra over $ Y^\lambda_\mu $.  We begin with the following general setup.

Assume we have a commutative complex algebra $ A $ and an algebra $ B \subset \End_\C(A) $ containing $ A $.  We consider $\End_\C(A^n) $.  We can embed $ B \subset \End_\C(A^n) $ by acting on just the first copy of $ A $.  On the other hand, the matrix algebras $ M_n(A) = \End_A(A^n) $ and $ M_n(B) $ also embed in $ \End_\C(A^n)$.  We have the following observation.

\begin{Lemma} \label{le:observ}
With the above setup, the subalgebra of $ \End_\C(A^n) $ generated by $ B $ and by $ M_n(A) $ equals $ M_n(B)$.
\end{Lemma}

Now, we recall the following standard results. For each $ i \in I, 1 \le r \le m_i$, we define the divided difference operator $ \partial_{i,r} : P \rightarrow P $ by $\partial_{i,r}=\frac{1}{z_{i,r+1}-z_{i,r}}( s_{i,r}-1)$ where $ s_{i,k} $ denotes the operator permuting the variables $ z_{i,r}, z_{i,r+1}$.

\begin{Proposition}\hfill
\begin{enumerate}
\item $ P $ is a rank $ \prod m_i^! $ free module over $P^\Sigma $.
\item The algebra $ \End_{P^\Sigma} P $ is generated by the operators of multiplication by elements of $ P $ and by the divided difference operators $ \{ \partial_{i,r} \} $.
\end{enumerate}
\end{Proposition}

In particular, $\End_{P^\Sigma} P $ contains the projection operator onto the subspace $ P^\Sigma $.  This projection is given by
\begin{equation}
\label{eq: symmetrizing idempotent}
e'_{\Bm} := \prod_i \frac{1}{m_i^!}\partial_{i, w_0} \prod_{1 \le k < l \le m_i} (z_{i,k} - z_{i,l} )
\end{equation}
where $ \partial_{i,w_0} $ denotes the product of the divided difference operators $ \partial_{i,r} $ corresponding to the longest element in $ \Sigma_{m_i} $.  Note that $\Bm=(m_i)_{i\in I}$.


\begin{Definition}
\label{def:flagyang}
The {\bf flag truncated shifted Yangian} $ FY^\lambda_\mu = FY^\lambda_\mu(\bR) $ is defined to be the subalgebra of $\End(P) $ generated by the divided difference operators, multiplication by elements of $ P $ and by $ Y^\lambda_\mu(\bR)$.
\end{Definition}


Thus, applying Lemma \ref{le:observ} to $ A = P^\Sigma, B = Y^\la_\mu$ and $n=\prod m_i!$, we immediately see that $ FY^\lambda_\mu $ is a rank $\prod m_i^! $ matrix algebra over $ Y^\lambda_\mu $.  Moreover, the module $  e'_{\Bm}FY^\lambda_\mu \cong Y^\lambda_\mu\otimes_{P^{\Sigma}}P $ induces a Morita equivalence between $FY^\lambda_\mu $ and $Y^\lambda_\mu $.

We can embed $ \Sigma $ into $ FY^\lambda_\mu $ using the elements $s_{i,k} =(z_{i,k+1}-z_{i,k})\partial_{i,k}+1$.  When we do this, we see that the idempotent $ e'_\Bm $ equal the usual idempotent of $ \Sigma $.  Thus we see that the Morita equivalence between $ FY^\la_\mu $-modules and $ Y^\lambda_\mu$-modules is just given by $ M \mapsto M^\Sigma$.

\subsection{The KLR Yangian algebra}
\label{sec:cylindrical}
Fix $\lambda$ and an integral set of parameters $\bR$ of weight $\lambda$.  We will now introduce a yet bigger algebra (depending on $ \bR$), which will contain $ FY^\lambda_\mu $ as a subalgebra, for all $ \mu $.

\begin{Definition}
  A {\bf cylindrical KLR diagram} is a collection of finitely many
  black curves in a cylinder $\R/\Z\times [0,1]$. Each curve is labeled with $i\in I$ and decorated with
  finitely many dots.  The diagram must be locally of the
  form
%
%
%
%
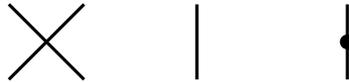
\begin{equation*}
\begin{tikzpicture}
  \draw[very thick] (-2,0) +(-.5,-.5) -- +(.5,.5);
  \draw[very thick](-2,0) +(.5,-.5) -- +(-.5,.5);


 \draw[very thick](0,0) +(0,-.5) --  +(0,.5);

  \draw[very thick](2,0) +(0,-.5) --  node
  [midway,circle,fill=black,inner sep=2pt]{}
  +(0,.5);
\end{tikzpicture}
\end{equation*}
with each curve oriented in the negative direction (we don't depict the orientation).  The curves must
meet the circles at
$y=0$ and $y=1$ at distinct points with $x\neq 0$. We consider these
up to isotopy preserving the conditions above.
\end{Definition}

We draw cylindrical KLR diagrams on the plane, using a dashed line at $x=0$ to depict the gluing of the cylinder.  We refer to the dashed line as the ``seam'' of the cylinder.  Here is an example of a cylindrical KLR diagram:
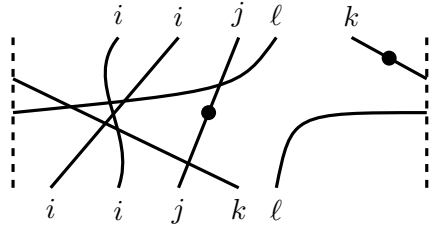
\begin{equation*}
\begin{tikzpicture}
\draw[very thick, dashed] (-4,0) +(0,-1) -- +(0,1);
\draw[very thick, dashed] (-4,0) +(5.5,-1) -- + (5.5,1);
\draw[very thick] (-4,0) +(0.5,-1) -- +(2.2,1) node[at start,below]{
$i$} node[at end,above]{
$i$};
\draw[very thick] (-4,0) +(1.4,-1) to[out=70,in=-130] +(1.4,1);
\draw[very thick] (-4,0) +(2.2,-1) -- node[midway,circle,fill=black,inner sep=2pt]{} +(3,1) node[at start,below]{
$j$} node[at end,above]{
$j$};
\draw[very thick] (-4,0) +(3,-1) -- +(0,.45) node[at start,below]{
$k$};
\draw[very thick] (-4,0) +(5.5,.45) -- node[midway,circle,fill=black,inner sep=2pt]{} +(4.5,1) node[at end,above]{
$k$};
\draw[very thick] (-4,0) +(3.5,-1) .. controls +(0.2,1) .. +(5.5,0);
\draw[very thick] (-4,0) +(0,0) .. controls +(3,.3) .. +(3.5,1);

\draw (-4,0) +(1.4,-1.3) node {$i$};
\draw (-4,0) +(1.4,-1)+(1.4,1.3) node {$i$};
\draw (-4,0) +(3.5,-1.3) node {$\ell$};
\draw (-4,0) +(0,0)+(3.5,1.3) node {$\ell$};


\end{tikzpicture}
\end{equation*}
As usual, the product of two such diagrams is defined by stacking them on top of each other if their labels match, and is zero otherwise.

Given a list $\Bi\in I^m$, we let $e(\Bi) $ be the cylindrical KLR diagram with vertical strands with these labels in order.  We write $\yz_k(\Bi) $ for the same diagram where the $k$th strand carries a dot, and $\psi_k(\Bi)$ for the diagram where the $k$th and $k+1$st strands cross.  (When the labels are clear from the context we will simply write $ \yz_k$ and $\psi_k$.). Diagrammatically if $\Bi=(i_1,...,i_m)$ then $e(\Bi), \yz_k(\Bi)$, and $\psi_k(\Bi)$ are given by:

\begin{equation*}
e(\Bi)=\,\begin{tikzpicture}[baseline,scale=1.2]
\draw[very thick, dashed] (-3.7,0) +(0,-.5) -- +(0,.5);
 \draw[very thick] (-3.4,0) +(0,-.5) -- +(0,.5);
  \draw[very thick] (-2,0) +(0,-.5) -- +(0,.5);
  \draw[very thick, dashed] (-1.7,0) +(0,-.5) -- +(0,.5);
  \draw (-3.4,0) +(0.1,-.8) node {\small$i_1$};
  \draw (-2,0) +(0.1,-.8) node {\small$i_m$};
  \draw (-2.65,0) node {$\cdots$};  \end{tikzpicture}
\qquad \qquad   \yz_k(\Bi)=\,\begin{tikzpicture}[baseline,scale=1.2]
\draw[very thick, dashed] (0,0) +(0,-.5) -- +(0,.5);
 \draw[very thick] (.3,0) +(0,-.5) -- +(0,.5);
  \draw[very thick] (2,0) +(0,-.5) -- +(0,.5);
  \draw[very thick, dashed] (2.3,0) +(0,-.5) -- +(0,.5);
  \draw (.3,0) +(0.1,-.8) node {\small$i_1$};
  \draw (2,0) +(0.1,-.8) node {\small$i_m$};
  \draw (.75,0) node {$\cdots$};
    \draw[very thick] (1.2,0) +(0,-.5) -- node [midway,circle,fill=black,inner sep=2pt]{} +(0,.5);
     \draw (1.65,0) node {$\cdots$};
      \draw (1.2,0) +(0.1,-.8) node {\small$i_k$};  \end{tikzpicture}
\qquad \qquad  \psi_k(\Bi)=\,\begin{tikzpicture}[baseline,scale=1.2]
      \draw[very thick, dashed] (4,0) +(0,-.5) -- +(0,.5);
 \draw[very thick] (4.3,0) +(0,-.5) -- +(0,.5);
  \draw[very thick] (6,0) +(0,-.5) -- +(0,.5);
  \draw[very thick, dashed] (6.3,0) +(0,-.5) -- +(0,.5);
  \draw (4.3,0) +(0.1,-.8) node {\small$i_1$};
  \draw (6,0) +(0.1,-.8) node {\small$i_m$};
  \draw[very thick] (4.9,0) +(0,-.5) -- +(.5,.5);
  \draw[very thick] (5.4,0) +(0,-.5) -- +(-.5,.5);
\draw (4.9,0) +(0.1,-.8) node {\small$i_k$};
\draw (5.4,0) +(0.1,-.8) node {\small$i_{k+1}$};
\draw (4.75,0) node {$\cdots$};
\draw (5.72,0) node {$\cdots$};
\end{tikzpicture}
\end{equation*}
We also have seam crossing diagrams $\sigma_\pm(\Bi)$:
    \begin{equation*}
 \sigma_+(\Bi)=\,\begin{tikzpicture}[baseline,scale=1.2]
      \draw[very thick, dashed] (4,0) +(0,-.5) -- +(0,.5);
 \draw[very thick] (4.3,0) +(0,.5) -- +(-.3,0);
  \draw[very thick] (6,0) +(.3,0) -- node[at end, below] {\small$i_m$}+(0,-.5);
  \draw[very thick] (4.3,-.5) --node[at start, below] {\small$i_1$} (4.9,.5);
 \draw[very thick] (5.4,-.5) -- node[at start, below] {\small$i_{m-1}$} (6,.5);
  \draw[very thick, dashed] (6.3,0) +(0,-.5) -- +(0,.5);
\draw (5.15,0) node {$\cdots$};
    \end{tikzpicture}
\qquad \qquad \sigma_-(\Bi)=\,\begin{tikzpicture}[baseline,scale=1.2]
      \draw[very thick, dashed] (4,0) +(0,-.5) -- +(0,.5);
 \draw[very thick] (4.3,0) +(0,-.5) -- node[at start, below] {\small$i_1$} +(-.3,0);
  \draw[very thick] (6,0) +(.3,0) -- +(0,.5);
 \draw[very thick] (4.3,.5) --node[at end, below] {\small$i_2$} (4.9,-.5);
 \draw[very thick] (5.4,.5) --node[at end, below] {\small$i_m$} (6,-.5);
  \draw[very thick, dashed] (6.3,0) +(0,-.5) -- +(0,.5);
\draw (5.15,0) node {$\cdots$};
    \end{tikzpicture}
  \end{equation*}

Let $ p_{i,+}(u)=p_i(u+2)$ and let
\[
\barX_{ij}(u,v)=
\begin{cases}
 u-v-1
  & i\leftarrow j\\
 1 & \text{otherwise}\\
 \end{cases}\]\[
\barQ_{ij}(u,v)=\barX_{ij}(u,v) \barX_{ji}(v,u)=
\begin{cases}
  u-v-1
  & i\leftarrow j\\
  v-u-1 & i\to j\\
  1 & i \not \leftrightarrow j
\end{cases}
\]
\begin{Remark}
  The -1 in the formula for $\barX_{ij}$ might look strange to readers used to KLR algebras; in fact, we can get rid of it by a carefully chosen automorphism of the algebra $\BK$ defined below.  However, in order to relate $\BK$ to the usual presentation of the Yangian (cf. Theorem \ref{thm:flagYa}), it is more convenient to use this shifted version.    
\end{Remark}
\begin{Definition}
\label{def:Ya}
  The {\bf KLR Yangian algebra} $\BK= \BK(\bR)$\footnote{Following the Cyrillic spelling ``\foreignlanguage{russian}{Янгиан}.''  This is, of course, pronounced ``Ya.''}
  is the quotient of the span of cylindrical KLR
  diagrams by
  \begin{itemize}
  \item the usual KLR algebra relations  (\ref{first-QH}--\ref{triple-smart}) for $I$ using the polynomials
    $\barQ_{ij}$; since these relations are local, we can
    apply them in a disk that avoids $x=0$.
    \item the following relations around $x=0$:
\newseq
    \begin{equation*}\subeqn\label{dot-slide}
    \begin{tikzpicture}[very thick,baseline,scale=.7]
  \draw(-3,0) +(-1,-1) -- +(1,1);
  \draw[dashed](-3,0) +(0,-1) -- +(0,1);
\fill (-3.5,-.5) circle (3pt); \end{tikzpicture}
=
 \begin{tikzpicture}[very thick,baseline,scale=.7] \draw(1,0) +(-1,-1) -- +(1,1);
  \draw[dashed](1,0) +(0,-1) -- +(0,1);
\fill (1.5,.5) circle (3pt);
    \end{tikzpicture} -  2\begin{tikzpicture}[very thick,baseline,scale=.7] \draw(1,0) +(-1,-1) -- +(1,1);
  \draw[dashed](1,0) +(0,-1) -- +(0,1);
    \end{tikzpicture}
\qquad \qquad     \begin{tikzpicture}[very thick,baseline,scale=.7]
  \draw(-3,0) +(1,-1) -- +(-1,1);
  \draw[dashed](-3,0) +(0,-1) -- +(0,1);
\fill (-2.5,-.5) circle (3pt); \end{tikzpicture}
=
 \begin{tikzpicture}[very thick,baseline,scale=.7] \draw(1,0) +(1,-1) -- +(-1,1);
  \draw[dashed](1,0) +(0,-1) -- +(0,1);
\fill (.5,.5) circle (3pt);
    \end{tikzpicture} +  2\begin{tikzpicture}[very thick,baseline,scale=.7] \draw(1,0) +(1,-1) -- +(-1,1);
  \draw[dashed](1,0) +(0,-1) -- +(0,1);
    \end{tikzpicture}
  \end{equation*}
 \begin{equation*}\subeqn\label{x-cost}
  \begin{tikzpicture}[very thick,baseline,scale=.7]
    \draw (-2.8,0)  +(0,-1) .. controls (-1.2,0) ..  +(0,1) node[below,at start]{$i$};
       \draw[dashed] (-2,0)  +(0,-1)--  +(0,1);
  \end{tikzpicture}
=   p_{i,+}\Bigg(
  \begin{tikzpicture}[very thick,baseline,scale=.7]
 \draw[dashed] (2.3,0)  +(0,-1) -- +(0,1);
       \draw (1.5,0)  +(0,-1) -- +(0,1) node[below,at start]{$i$};
       \fill (1.5,0) circle (3pt);
\end{tikzpicture}\hspace{5mm}\Bigg)\qquad \qquad
  \begin{tikzpicture}[very thick,baseline,scale=.7]
          \draw[dashed] (-2,0)  +(0,-1)-- +(0,1);
  \draw (-1.2,0)  +(0,-1) .. controls (-2.8,0) ..  +(0,1) node[below,at start]{$i$};\end{tikzpicture}
           =p_i\Bigg(\hspace{5mm}
  \begin{tikzpicture}[very thick,baseline,scale=.7]
    \draw (2.5,0)  +(0,-1) -- +(0,1) node[below,at start]{$i$};
       \draw[dashed] (1.7,0)  +(0,-1) -- +(0,1) ;
       \fill (2.5,0) circle (3pt);\end{tikzpicture}\Bigg)
\end{equation*}
\excise{ \begin{equation*}\subeqn\label{x-cost}
  \begin{tikzpicture}[very thick,baseline,scale=.7]
    \draw (-2.8,0)  +(0,-1) .. controls (-1.2,0) ..  +(0,1) node[below,at start]{$i$};
       \draw[dashed] (-2,0)  +(0,-1)--  +(0,1);
  \end{tikzpicture}
=
  \begin{tikzpicture}[very thick,baseline,scale=.7]
 \draw[dashed] (2.3,0)  +(0,-1) -- +(0,1);
       \draw (1.5,0)  +(0,-1) -- +(0,1) node[below,at start]{$i$};
\end{tikzpicture}\qquad \qquad
  \begin{tikzpicture}[very thick,baseline,scale=.7]
          \draw[dashed] (-2,0)  +(0,-1)-- +(0,1);
  \draw (-1.2,0)  +(0,-1) .. controls (-2.8,0) ..  +(0,1) node[below,at start]{$i$};\end{tikzpicture}
           =
  \begin{tikzpicture}[very thick,baseline,scale=.7]
    \draw (2.5,0)  +(0,-1) -- +(0,1) node[below,at start]{$i$};
       \draw[dashed] (1.7,0)  +(0,-1) -- +(0,1) ;
\end{tikzpicture}
\qquad
\end{equation*}
}
\begin{equation*}\subeqn\label{x-triple-smart}
    \begin{tikzpicture}[very thick,scale=.9,baseline]
      \draw (-3,0) +(1,-1) .. controls (-4,0) .. +(-1,1) node[below,at start]{$i$}; \draw
      (-3,0) +(-1,-1) .. controls (-4,0) .. +(1,1) node[below,at start]{$i$}; \draw[dashed]
      (-3,0) +(0,-1)--  +(0,1); \node at (-1,0) {=}; \draw (1,0) +(1,-1) .. controls
      (2,0) .. +(-1,1)
      node[below,at start]{$i$}; \draw (1,0) +(-1,-1) .. controls
      (2,0) .. +(1,1)
      node[below,at start]{$i$}; \draw[dashed] (1,0) +(0,-1) -- +(0,1); \node at (2.8,0)
      {$+$};        \draw (6.2,0)
      +(1,-1) -- +(1,1) node[below,at start]{$i$}; \draw (6.2,0)
      +(-1,-1) -- +(-1,1) node[below,at start]{$i$}; \draw[dashed] (6.2,0)
      +(0,-1) -- +(0,1);
\node[inner ysep=8pt,inner xsep=5pt,fill=white,draw,scale=.8] at (6.2,0){$\displaystyle \frac{p_{i,+}(\yz_m)-p_i(\yz_1)}{\yz_m-\yz_1+2}$};
    \end{tikzpicture}
  \end{equation*}
\begin{equation*}\subeqn\label{x-triple-dumb}
    \begin{tikzpicture}[very thick,scale=.9,baseline]
      \draw (-3,0) +(1,-1) .. controls (-4,0) .. +(-1,1) node[below,at start]{$j$}; \draw
      (-3,0) +(-1,-1) .. controls (-4,0) .. +(1,1) node[below,at start]{$i$}; \draw[dashed]
      (-3,0) +(0,-1)--  +(0,1); \node at (-1,0) {=}; \draw (1,0) +(1,-1) .. controls
      (2,0) .. +(-1,1)
      node[below,at start]{$j$}; \draw (1,0) +(-1,-1) .. controls
      (2,0) .. +(1,1)
      node[below,at start]{$i$}; \draw[dashed] (1,0) +(0,-1) -- +(0,1);
    \end{tikzpicture}
\qquad \text{if $i\neq j$.}
  \end{equation*}
  \end{itemize}
\end{Definition}

Because diagrams can only be multiplied if their labels match, we have an obvious decomposition of the algebra according to multiplicities of the simple roots in the label.  The diagrams $ e(\Bi) $ give the idempotents corresponding to this decomposition.


Let$$\PolKLR_{\tiny \BK} = \bigoplus_{\substack{m \geq 1, \\ \Bi\in I^m }} \Pol_\Bi, \quad \Pol_\Bi := \C[\YZ_1(\bi),\dots, \YZ_m(\bi)]$$
be a direct sum of polynomial rings, one for each list $ \Bi$.

Let
$\sigma : I^m \rightarrow I^m$ be the rotation, so that $ \sigma(i_1, \dots, i_m) = (i_m, i_1 \dots, i_{m-1}) $  and extend this to an algebra automorphism of $ \Pol_{\tiny \BK} $ by
\[\sigma(\YZ_k(\Bi))=
\begin{cases}
  \YZ_{k+1}(\sigma(\Bi)) & k<m\\
  \YZ_{1}(\sigma(\Bi)) - 2 & k=m.
  \end{cases}\]

Also for any $1 \le k < m$, we write $ s_k : I^m \rightarrow I^m $ for the usual transposition, extended to an algebra automorphism of $ \Pol_{\tiny \BK} $ by $ s_k(\YZ_r(\Bi)) = \YZ_{s_k(r)}(s_k( \Bi))$.

If we have a list $ \Bi $ with $ i_k = i_{k+1} $, then we define the divided difference operator $ \partial_k : \Pol_\Bi \rightarrow \Pol_\Bi $ by $\partial_k =\frac{1}{\YZ_{k+1}(\bi) - \YZ_k(\bi)}(s_k - 1)$.  (Note that in this formula $s_k$ acts on $\Pol_\Bi$ since $i_k=i_{k+1}$.).

\begin{Theorem}
  We have a faithful action of $\BK$ on $\Pol_{\tiny \BK}$ sending:
  \begin{itemize}
  \item the idempotent $ e(\Bi) $ to the projection onto $ \Pol_\Bi$
  \item the dot $\yz_k(\Bi)$ to the multiplication operator $\YZ_k(\Bi)$
\item  the crossing $\psi_k(\Bi)$ to the Demazure operator or multiplication operator, depending on whether $i_k=i_{k+1}$ or not:
\[\psi_k(\Bi) \mapsto
    \begin{cases}
      \partial_k & i_k=i_{k+1}\\
      s_k \circ \barX_{ij}(\YZ_k(\bi),\YZ_{k+1}(\bi)) &i_k\neq i_{k+1}
    \end{cases}\subeqn\label{eq:black-cross-action}
\]
  \item the seam crossings $\sigma_{\pm}(\Bi)$ to the automorphism $\sigma$ or its inverse times a polynomial, depending on the sign:
    \begin{align*}\subeqn\label{eq:x-cross-right}
\sigma_+(\Bi)&\mapsto \sigma|_{\Pol_\Bi}
      \\ \subeqn\label{eq:x-cross-left}
\sigma_-(\Bi)&\mapsto
                   \left(p_{i,+}(\YZ_m(\Bi))\circ
                       \sigma^{-1}\right)|_{\Pol_\Bi}=\left(\sigma^{-1}\circ p_i(\YZ_1(\sigma(\Bi)))\right)|_{\Pol_\Bi}
    \end{align*}
  \end{itemize}
\end{Theorem}
\begin{proof}
The KLR relations hold by comparison with the usual polynomial relation of the KLR algebra (for example, as in \cite{Rou2KM}).  The other relations are straightforward to check.

Thus, it remains to check that the representation is injective.  Assume that we have an element of the kernel.  This is given by a sum of a finite number of cylindrical diagrams, and the strands of each diagram trace out a word in the length 1 and length 0 generators in an extended affine symmetric group.  First note that the relations (\ref{triple-dumb}--\ref{triple-smart},\ref{x-triple-smart}--\ref{x-triple-dumb}) allow to replace the element for one word with the diagram for the word with a braid relation apply to it, plus diagrams with shorter words.

If one of these words is not reduced, then we can use braid relations until we have two consecutive appearances of a generator.  At this point, we can apply the relations (\ref{black-bigon},\ref{x-cost}) to reduce the length.  Thus, using the relations of the algebra, we can assume this element of the kernel is a sum of diagrams corresponding to reduced words of different elements of the affine symmetric group, times polynomials in dots.  Note that the action of each of such diagram in the representation is the action of the corresponding symmetric group element, times a rational function, plus similar terms of shorter length.  Now consider an element $w$ of maximal length with non-zero coefficient on its diagram.  In this case, the image of this element in the representation must be a sum of the usual action of $w$ with non-zero coefficient and shorter elements of the affine symmetric group.  Thus, this image is non-zero, contradicting the assumption.
\end{proof}

\begin{Remark}
\label{rem:Polm}
Let $\Pol_m=\bigoplus_{\Bi\in I^m}\Pol_\Bi$ so that $\Pol_{\tiny \BK}=\bigoplus_{m\geq0}\Pol_m$.  It's clear from the formulas that $\Pol_m$ is $\BK$-invariant for any $m$.
\end{Remark}

\subsection{Relating the KLR Yangian and the flag Yangian algebras}

As in the previous section, we fix $\lambda$ and an integral set of parameters $\bR$.  In addition we fix a weight $\mu\leq \lambda$ and let $\Bm=(m_i)_{i\in I}$.  Recall, that we previously fixed a total order on the set $ I $ such that all vertices in
$I_{\bar 0}$ come before those in $I_{\bar 1}$.
Consider the sequence $\Bi_{\Bm}$ where the nodes of $ I $ appear in the above order, with $i$ appearing with
multiplicity $m_i$.  Let $e(\Bi_{\Bm})$ be the corresponding idempotent in $ \BK $.  By the previous theorem, we can view $ e(\Bi_{\Bm}) \BK e(\Bi_{\Bm}) $ as a subalgebra of endomorphisms of $\Pol_{\Bi_\Bm} $.  Using the above total order on $ I $, we will identify the variables $ z_{i,k} $ with the variables $ \yz_1(\Bi_\Bm), \dots, \yz_m(\Bi_\Bm) $.  In this way, we identify $ \Pol_{\Bi_\Bm} = P $.

The algebra $e(\Bi_\Bm) \BK e(\Bi_\Bm)$ also naturally contains a tensor product $\bigotimes_i \operatorname{NH}_{m_i}$ of nilHecke algebras.  This embeds as the subalgebra where we only permit crossings between pairs of strands which have equal labels (and no crossing of the seam $x=0$). We may thus identify the element $e'_{\Bm} \in e(\Bi_\Bm) \BK e(\Bi_\Bm)$ from (\ref{eq: symmetrizing idempotent}).

\begin{Theorem}
\label{thm:flagYa}
As subalgebras of $ \End_\C P $, we have $ FY^\la_\mu = e(\Bi_\Bm) \BK e(\Bi_\Bm) $.  Moreover, the elements $ E_i^{(r)}, F_i^{(r)} \in Y^\lambda_\mu \subset FY^\la_\mu $ correspond to elements of $ \BK $ as follows:
\newseq
\begin{align*}
E_i^{(r)} &=   - w_{i,+} (\yz_{m_1 + \dots + m_{i}}+2)^{r-1} e'_{\Bm} \subeqn \label{eq:E-twist}\\
F_i^{(r)} &=  (-1)^{\sum_{i\rightarrow j} m_j} w_{i,-} \yz_{m_1 + \dots + m_{i-1} + 1}^{r-1} e'_{\Bm} \subeqn\label{eq:F-twist}
\end{align*}
Here $w_{i,\pm}$ is the diagram wrapping the rightmost
(resp. leftmost) strand with label $i$ a full positive
(resp. negative) revolution around the cylinder.
\end{Theorem}
For example, diagrammatically:
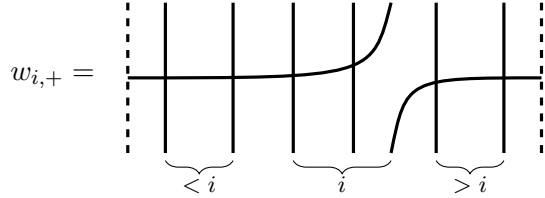
\begin{equation*}
\begin{tikzpicture}
\node at (-5,0) {$ w_{i,+} = $};
\draw[very thick, dashed] (-4,0) +(0,-1) -- +(0,1);
\draw[very thick, dashed] (-4,0) +(5.5,-1) -- + (5.5,1);

\draw[very thick] (-4,0) +(0.5,-1) -- +(0.5,1);
\draw[very thick] (-4,0) +(1.4,-1) -- +(1.4,1);
\draw[very thick] (-4,0) +(2.2,-1) -- +(2.2,1);
\draw[very thick] (-4,0) +(3,-1) -- +(3,1);
\draw[very thick] (-4,0) +(4.1,-1) -- +(4.1,1);
\draw[very thick] (-4,0) +(5,-1) -- +(5,1);

\draw[very thick] (-4,0) +(3.5,-1) .. controls +(0.2,1) .. +(5.5,0);
\draw[very thick] (-4,0) +(0,0) .. controls +(3.25,0) .. +(3.5,1);

\draw [decorate,decoration={brace,amplitude=5pt}] (-4,0) +(1.4, -1.1) -- +(0.5,-1.1)  node [black,midway,below,yshift=-2pt] {\footnotesize $<i$};
\draw [decorate,decoration={brace,amplitude=5pt}] (-4,0) +(3.5, -1.1) -- +(2.2,-1.1)  node [black,midway,below,yshift=-2pt] {\footnotesize $i$};
\draw [decorate,decoration={brace,amplitude=5pt}] (-4,0) +(5, -1.1) -- +(4.1,-1.1)  node [black,midway,below,yshift=-2pt] {\footnotesize $>i$};
\end{tikzpicture}
\end{equation*}
where the braces at the bottom denote labels from $I$, appearing in the order $\Bi_\Bm$.  Meanwhile, the element $z_{m_1+\ldots + m_i}$ corresponds to a dot on the right-most strand labeled $i$.

\begin{proof}
First, let us confirm that the left and right sides of the equations (\ref{eq:E-twist}--\ref{eq:F-twist}) act in the same way on $P^{\Sigma}$, i.e. that the right-hand sides match (\ref{eq:E-action}--\ref{eq:F-action}).

We will prove (\ref{eq:E-twist}) carefully, and leave it to the reader to derive (\ref{eq:F-twist}) similarly.   First, we note that $(\yz_{m_1 + \dots + m_{i}}+2)^{r-1}$ acts on $P$ by $(z_{i,m_i}+2)^{r-1}$.  Then, as we move the strand to the right, we pass strands with label $j>i$.  By (\ref{eq:black-cross-action}), this has no effect besides permuting the labels of variables, since either there is no arrow between the nodes, or $i\to j$ (by our choice of orientation, cf. Section \ref{subsection:Notation}).  Then, we cross rightward over $x=0$, which by (\ref{eq:x-cross-right}) contributes an action of $\sigma$.  Next, we cross rightward over the strands with labels $j<i$.  Those with $j \to i$ contribute factors of $\overline{P}_{ij}$ in appropriate variables.  The crossing of all strands labelled $j<i$ contributes a factor of
\[\prod_{j<i} \prod_{k=1}^{m_j} Q_{ij}(z_{i,1},z_{j,k})=\prod_{j<i}\prod_{k=1}^{m_j} (z_{i,1}-z_{j,k}-1)=\prod_{j\rightarrow i} W_j(z_{i,1}-1).\]
Note that, although $\Bi_\Bm$ is permuted during the process of crossing strands, after crossing all those discussed above we are back at $\Bi_\Bm$.

Thus far, we have wrapped the strand almost all the way around the cylinder, but without crossing any of the other strands with label $i$.  We can summarize the action of the element defined thus far by
\begin{equation*}
\prod_{j\rightarrow i} W_j(z_{i,1}-1) \rho_i (z_{i,m_i}+2)^{r-1} = \prod_{j\rightarrow i} W_j(z_{i,1}-1) z_{i,1}^{r-1}\rho_i
\end{equation*}
where $\rho_i$ is the algebra automorphism of $P$ defined by
\[\rho_i(z_{j,k})=\begin{cases}
  z_{j,k} & i\neq j\\
  z_{i,k+1} & i=j, k\neq m_i\\
  z_{i,1}-2 & i=j, k=m_i
\end{cases}\]
Note that $\rho_i$ has the same effect on $P^{\Sigma}$ as $\beta_{i,1}^{-1}$.

Finally, in crossing over the remaining strands with label $i$, by applying (\ref{eq:black-cross-action}) again, we get a contribution of $ \partial_{i, m_i -1} \cdots \partial_{i, 2} \partial_{i, 1} $.  Altogether, we see that $w_{i,+} (z_{m_1+\ldots+m_i}+2)^{r-1}$ acts on $P^\Sigma$ by
$$
\partial_{i,m_i-1}\cdots \partial_{i,1} z_{m_1}^{r-1} \prod_{j\rightarrow i} W_j(z_{i,1} - 1) \beta_{i,1}^{-1}
$$
By Lemma \ref{lemma:filGKLO} (see also the proof of Theorem \ref{th:filGKLO}), this agrees with the action of $ - E_i^{(r)}$ on $P^\Sigma$ under (\ref{eq:E-action}).  This shows that the both sides of (\ref{eq:E-twist}) agree as operators $P^\Sigma \to P$.  Since we can embed $\Hom(P^\Sigma,P)=\Hom(P,P)e'_{\Bm}\subset \Hom(P,P)$, and both the flag Yangian and the KLR Yangian act faithfully on $P$, this proves  (\ref{eq:E-twist}).

Note that equations (\ref{eq:E-twist},\ref{eq:F-twist}) imply that, as subalgebras of $\End(P)$, $$FY_\mu^\lambda \subset e(\Bi_\Bm) \BK e(\Bi_\Bm).$$  Indeed these equations prove that the generators of $Y_\mu^\lambda$ appear in $e(\Bi_\Bm) \BK e(\Bi_\Bm)$.  The inclusion follows since $e(\Bi_\Bm) \BK e(\Bi_\Bm)$ also contains the tensor product of nil-Hecke algebras corresponding to packets of like-colored strands.
It remains to check that this inclusion is an equality.  Our proof of this requires a detour into the theory of Coulomb branches, and appears as \cite[Corollary 4.12]{Weekes}.
\end{proof}
\begin{Remark}
One can reasonably wonder if there is an interpretation of the entirety of $\BK$ in terms of the other aspects of this theory; one such interpretation is given in \cite[Thm. 4.11]{Weekes} in terms of the extended BFN category of \cite[Def. 3.5]{Webdual}.  This generalizes the ``affine KLR algebra'' which is sketched out in \cite{FinkelbergICM}.
\end{Remark}
\begin{Remark}
  It is possible to define a variant of $\BK$ in order to give a diagrammatic description of $Y^\la_\mu$.  To accomplish this, we use thick calculus in the sense of \cite{KLMS}. We consider diagrams in $ \BK$ which start and end with the idempotent $ e(\Bi_\Bm) $ and then we use a splitter to collect all strands with the same label $i$, in the sense of \cite[\S 2]{KLMS}.  We can consider these diagrams modulo the relations of Definition \ref{def:Ya} away from the splitters, as well as the relation that crossing two strands entering the splitter gives 0:
\[\tikz[very thick,baseline]{\draw (0,-.5) to [out=105,in=-60](-1,.5); \draw (0,-.5) to [out=90,in=-90] (.2,.1) to [out=90,in=-60](-.2,.5); \draw (0,-.5) to [out=90,in=-90] (-.2,.1) to [out=90,in=-120](.2,.5); \draw (0,-.5) to [out=75,in=-120](1,.5); }=0\]
We compose these diagrams by stacking them, and resolving the product of splitters into a half twist of the strands passing through it, as in \cite[Cor. 2.6]{KLMS}.  This algebra has the unusual feature that it is not easy to write the identity in it, though it is possible using \cite[(2.70)]{KLMS}; we can extend our diagrammatic calculus further by allowing thick strands that represent several strands with a single label ``zipped together'' (in \cite{KLMS}, these are drawn as green).  We will not go into details about the relations between these diagrams, as they can get quite complicated.
In this framework, we can write $E_i^{(s)}$ and $F_i^{(s)}$ as simpler diagrams:
  \begin{equation*}E_i^{(s)}=-\,\,
 \begin{tikzpicture}[baseline,scale=2.3]
      \draw[very thick, dashed] (4,0) +(0,-.5) -- +(0,.5);
 \draw[line width =1mm] (4.3,0) +(0,-.5) -- +(0,.5);
  \draw[line width =1mm] (6,0) +(0,-.5) -- +(0,.5);
  \draw[very thick, dashed] (6.3,0) +(0,-.5) -- +(0,.5);
  \draw (4.3,0) +(0,-.7) node {\small$i_1$};
  \draw (6,0) +(0,-.7) node {\small$i_r$};
  \draw[line width =1mm] (4.8,0) +(0,-.5) -- +(0,.5);
  \draw[line width =1mm] (5.5,0) +(0,-.5) -- +(0,.5);
  \draw[line width =1mm] (5.15,0) +(0,-.5) -- +(0,.5);
\draw (5.15,0) +(0,-.7) node {\small$i$};
\draw (4.55,0) node {$\cdots$};
\draw (5.75,0) node {$\cdots$};
\draw[very thick] (5.15,-.5) to[out=90,in=-150] (6.3,0);
\draw[very thick] (5.15,.5) to[out=-90,in=30]
node[pos=.615,scale=.8,label=above:$s-1$,circle, inner sep=3pt,fill=black]{} (4,0);
\end{tikzpicture}\qquad F_i^{(s)}=\pm\,\,
 \begin{tikzpicture}[baseline,scale=2.3]
      \draw[very thick, dashed] (4,0) +(0,-.5) -- +(0,.5);
 \draw[line width =1mm] (4.3,0) +(0,-.5) -- +(0,.5);
  \draw[line width =1mm] (6,0) +(0,-.5) -- +(0,.5);
  \draw[very thick, dashed] (6.3,0) +(0,-.5) -- +(0,.5);
  \draw (4.3,0) +(0,-.7) node {\small$i_1$};
  \draw (6,0) +(0,-.7) node {\small$i_r$};
  \draw[line width =1mm] (4.8,0) +(0,-.5) -- +(0,.5);
  \draw[line width =1mm] (5.5,0) +(0,-.5) -- +(0,.5);
  \draw[line width =1mm] (5.15,0) +(0,-.5) -- +(0,.5);
\draw (5.15,0) +(0,-.7) node {\small$i$};
\draw (4.55,0) node {$\cdots$};
\draw (5.75,0) node {$\cdots$};
\draw[very thick] (5.15,.5) to[out=-90,in=150] (6.3,0);
\draw[very thick] (5.15,-.5) to[out=90,in=-30]
node[pos=.615,scale=.8,label=below:$s-1$,circle, inner sep=3pt,fill=black]{} (4,0);
\end{tikzpicture}
\end{equation*}

\end{Remark}

\section{Weight modules and parity KLR algebras}

\subsection{Weight modules for (flag) truncated shifted Yangians}
The algebras $Y^\la_\mu, FY^\la_\mu$ have natural polynomial subalgebras, so we can consider weight modules.

The algebra $ Y^\la_\mu $ contains the polynomial subalgebra $ P^\Sigma = \C[A_i^{(s)}] $.  We identify $ \Spec P^\Sigma $ with the set of collections $ \bS = (S_i)_{i \in I} $ of multisets, where $|S_i|=m_i$.  Here $ \bS $ gives rise to the map $ ev_{\bS} : P^\Sigma \rightarrow \C $ taking $ A_i^{(s)} $ to the $s$th elementary symmetric function of $ S_i $.

 The algebra $ FY^\la_\mu $ contains the polynomial subalgebra $ P = \C[z_{i,k}] $.  We identify $ \Spec P= \prod_i \C^{m_i}$, where $ \Bnu = (\nu_i)_{i \in I}\in \Spec P$ gives rise to the map $ev_\Bnu : P \rightarrow \C $ taking $ z_{i,k} $ to $\nu_{i,k}$.

\begin{Definition}\label{Def5.1}
  A {\bf weight module} over $Y^\la_\mu$ (resp. $FY^\la_\mu$) is a module $M$ for which the subalgebra $ P^\Sigma $ (resp. $P $) acts locally finitely with finite-dimensional generalized eigenspaces.
  \end{Definition}

Let $M $ be a weight module for $Y^\la_\mu $.  So we can write
  \begin{gather*}
  M = \bigoplus_{\bS \in \Spec P^\Sigma} W_\bS(M) \\
   W_\bS(M) = \big\{ v \in M : \exists N \text{ s.t. }(f - ev_\bS(f))^N v = 0 \text{ for all } f \in P^\Sigma \big\}
 \end{gather*}
We refer to $ W_\bS(M) $ as the $\bS $-weight space of $ M $ and we say that $ M \mapsto W_\bS(M) $ is a weight functor.
Similarly, for a  module $ M$ over $ FY^\la_\mu $, we define $W_\Bnu(M)$ and we have weight modules where $ M = \bigoplus_{\Bnu} W_\Bnu(M)$.

Note that under the Morita equivalence, weight modules over
$\FY^\la_\mu$ correspond to weight modules over $Y^\la_\mu$.
Now, there is an obvious map from $ \Spec P \rightarrow \Spec P^\Sigma $ which takes the tuples $ \nu_i $ and turns them into multisets $ \bS(\Bnu)=(S_i(\Bnu))$.
Let $ M $ be a weight module over $ FY^\la_\mu $ and recall that its image under the Morita equivalence is $ M^\Sigma $, where $ \Sigma $ is the product of symmetric groups.  The group $ \Sigma $ also acts on the space $ \Spec P = \prod \C^{m_i} $ of weights for $M$ and we can consider the stabilizer of a weight $ \Bnu $.  It is easy to see that
\begin{equation}
\label{eq:weights-Y}
W_{\bS(\Bnu)}(M^\Sigma) = W_{\Bnu}(M)^{Stab_{\Sigma}(\Bnu)}.
\end{equation}
\excise{
As usual, we call a weight $\Bnu$ {\bf dominant} if for each $i$,
we have the sequence $\nu_{i,1},\dots, \nu_{i,m_i}$ are
weakly increasing.  For each $J$, there is a dominant
$\Ba$ it arises from as above, which is unique up to permutation
of entries with the same real part.
}
We say that $ \Bnu \in \prod_i \C^{m_i} $ is an \textbf{integral weight}, if for each $i \in I $ and $k \le m_i$, $\nu_{i,k}$ is an integer of the same parity
as $i$.  We say that a $FY^\la_\mu $ weight module $ M $ \textbf{has integral weights},
if $ W_\Bnu(M) \ne 0 $ implies that $\Bnu $ is an integral weight.  Let $\wtmodFY$ denote the category of weight modules over $FY^\la_\mu $ with integral weights.

Considering the adjoint action of the dots on the flag Yangian, we see
that if a module $M$ is generated by finitely many weight vectors, then
every other weight that appears will lie in the affine Weyl group
orbit of the original weights.  Since integral weights are closed under the action of the
affine Weyl group, we could equivalently define $\wtmodFY$ to
be the category of weight modules generated by vectors of integral weight.

\subsection{An equivalence relating the flag Yangian and the metric KLRW algebras}


Recall that idempotents in the  metric KLRW algebra $\tmetric= \tmetric^\bR $ are indexed by longitude triples $ (\Bi, \Ba, \kappa) $ satisfying the conditions of Definition \ref{def:long}.  For each integral weight $ \Bnu $, let $ d(\Bnu)=d(\bS(\Bnu))$ be the idempotent in $\tmetric$.
%
Let $\Sigma_{\Bnu}$ be the subgroup of $\Sigma$ that stabilizes $\Bnu$; this is a product of symmetric groups which permute the groups of strands that have the same 
label and longitude.   Using the formula in (\ref{eq: symmetrizing idempotent}), we can find an idempotent in $d(\Bnu)\tmetric d(\Bnu)$ which acts in the polynomial representation $ \PolKLR(\Bi, \Ba, \kappa)$ by projecting to invariants for this group. We denote this idempotent by $d'(\Bnu) $.  Since the nilHecke algebra of $\Sigma_\Bnu$ is a matrix algebra, any other primitive idempotent in it will be isomorphic to $d'(\Bnu)$.  

\excise{
Conversely given $\longi\in \Long$, there is a unique dominant weight
$\Ba (\longi)$ such that $
a_{i,1},\dots,a_{i,v_i}$ give the
longitudes of strands with label $i$.
}

Let $\tmetric^{\bR}_\mu\operatorname{-mod}_{\operatorname{nil}}$ be the
subcategory of finitely generated modules over
$\tmetric^{\bR}_\mu$ where the
dots act nilpotently; since $\tmetric^{\bR}_\mu$ is finitely generated as
a module over the dots, such a module is necessarily finite
dimensional.
\begin{Theorem} \label{co:tilde-main}
There is an equivalence $\Theta \colon \tmetric^{\bR}_\mu\operatorname{-mod}_{\operatorname{nil}} \to \wtmodFY$ such that
\begin{equation}  W_\bS(\Theta(M)^\Sigma) =  d'(\bS)M \qquad  W_\Bnu(\Theta(M))= d(\Bnu)M.\label{eq:w-e}
\end{equation}
\end{Theorem}

This is actually a consequence of a more general equivalence, which
we'll discuss below.  While this equivalence factors through objects
which are more auxilliary from our perspective, it is easier to prove
since it involves categories with simpler generators and relations.

Let us note one interesting corollary.  Let $\tilde{p}(\bS)=[\tilde{T}d'(\bS)]$; this is equal to the corresponding vector $\tilde{p}_{\Bi}^\kappa$ up to division by the size of the stabilizer of $\Bnu$ in $\Sigma$.  Transfering Lemma \ref{lem:idempotent-canonical} through this equivalence, we find that:
\begin{Corollary} We have \[\displaystyle \tilde{p}(\bS)=\sum_{b}\big(\dim W_\bS(\Theta(L_b)^\Sigma)\big)\cdot b\] where the sum is over parity canonical basis vectors $b$, and $L_b$ is the corresponding simple $\tilde{P}$-module.  
\end{Corollary}

\subsection{Weight modules for \foreignlanguage{russian}{Я}}
Let $ M $ be a module for $ \BK$.  Such a module comes with a decomposition $ M = \bigoplus_{\bi} e(\Bi) M $, according to the idempotents $ e(\Bi) $ of $ \BK $.  We say that $ M $ is a \textbf{weight module}, if for all sequences $ \Bi \in I^m  $, we have
$$
e(\Bi) M = \bigoplus_{\Ba \in \C^m} W_{\Bi, \Ba}(M)
$$
where this is a decomposition into generalized eigenspaces
$$
W_{\Bi,\Ba}(M) = \{ v \in e(\Bi)M : \exists N \text{ s.t. }(\yz_k(\Bi) - a_k)^N v = 0, \text{ for } k = 1, \dots, m \}
$$

We call the pair $(\Bi,\Ba) \in I^m \times \C^m $ {\bf integral}, if for each $k$, $a_k$ is
an integer of the same parity as $i_k$.  We say that a $ \BK$-module $ M $ has \textbf{integral weights} if $ W_{\Bi,\Ba}(M) \ne 0 $ implies that $ (\Bi, \Ba) $ is integral.  Let $\wtmodBK$ denote the category of $\BK$ weight modules with integral weights.

We have a quotient functor from $ \wtmodBK $ to $ \wtmodFY $ given by multiplication by the idempotent $e(\Bi_\Bm)$
defined earlier.  This quotient functor has a fully faithful left adjoint
sending $M$ to $ \BK\otimes_{FY}M$.

The category $ \wtmodBK $ has weight functors $ W_{\Bi, \Ba} $ indexed by integral pairs $ (\Bi, \Ba) $ as above.
We will be interested in morphisms between these weight functors, the most important (for our purposes) we call the neutral crossing which goes between pairs connected by admissible permutations.

\subsection{Admissible permutations and neutral crossings}
\label{sec:admissible}
Recall that the usual action of the symmetric group $S_m$ on $m$-tuples is given by the rule $\pi(\Ba)=(a_{\pi^{-1}(1)},\dots, a_{\pi^{-1}(m)})$.

\begin{Definition}
Let $ (\Bi, \Ba) \in I^m \times \C^m $.  We say that a pair of distinct indices $ k, l \in \{1, \dots, m \} $ are \textbf{not $ (\Bi, \Ba)$-switchable} if $ i_k \leftarrow i_l $ and $ a_k = a_{l}+1 $; otherwise, we say that they are switchable.  We say that a permutation $ \pi \in \Sigma_m $ is $(\Bi, \Ba)$-\textbf{admissible} if whenever $ (k,l) $ are not $ (\Bi, \Ba)$-switchable, then $ \pi(k), \pi(l) $ are in the same order as $ k, l$.
\end{Definition}
Note that $k,k+1$ are not $(\Bi,\Ba)$ switchable if and only if $\barX_{i_k,i_{k+_1}}(a_k,a_{k+_1})=0$.
Observe also that the simple transposition $ s_k $ is admissible  if and only if $k, k+1 $ are switchable, which means
$$\text{ if $ i_k \to i_{k+1} $ then $ a_k \ne a_{k+1} -1$ and if $ i_k \leftarrow i_{k+1} $, then $ a_k \ne a_{k+1} + 1 $.}  $$

\begin{Lemma}
Let $ \pi_1, \pi_2 $ be two permutations.  Assume that $ \pi_1 $ is $(\Bi, \Ba)$-admissible and $ \pi_2 $ is $ (\pi_1(\Bi), \pi_1(\Ba))$-admissible.  Then $\pi_2 \circ \pi_1 $ is $ (\Bi, \Ba)$-admissible.
\end{Lemma}
\begin{proof}
Suppose $k<l$ are such that $ i_k \leftarrow i_l $ and $ a_k = a_{l}+1 $.  Then $\pi_1(k)<\pi_1(l)$ since $ \pi_1 $ is $(\Bi, \Ba)$-admissible.  Letting $\Bj=\pi_1(\Bi)$ and $\Bb=\pi_1(\Ba)$, we then have that $j_{\pi_1(k)}=i_k \leftarrow i_l=j_{\pi_1(l)}$ and $b_{\pi_1(k)}=a_k=a_l+1=b_{\pi_1(l)}+1$.  Since $\pi_2$ is $(\Bj,\Bb)$-admissible it follows that $\pi_2\pi_1(k)<\pi_2\pi_1(l)$. The analogous argument holds if $k>l$.
\end{proof}

\begin{Lemma}
Let $ \pi $ be a permutation which is $ (\Bi, \Ba)$-admissible.  Assume that $ \ell(\pi s_k) < \ell(\pi) $.  Then $ s_k $ is $ (s_k(\Bi), s_k(\Ba)) $-admissible.
\end{Lemma}

\begin{proof}
Since $ \ell(\pi s_k) < \ell(\pi) $, we see that $ \pi(k+1) < \pi(k) $.  Hence $ k, k+1 $ are $(\Bi, \Ba)$-switchable.  This implies the desired result.
\end{proof}

The following corollary follows immediately from the previous two lemmas.

\begin{Corollary} \label{cor:admissibleword}
If $ \pi $ is an $ (\Bi, \Ba)$-admissible permutation, and if $ \pi = s_{k_r} \cdots s_{k_1} $ is a reduced word, then each $ s_{k_p} $ is $ (s_{k_p} \cdots s_{k_1}\Bi, s_{k_p} \cdots s_{k_1}\Ba)$-admissible.
\end{Corollary}

The following examples of admissible permutations will be important later.

\begin{Lemma} \label{le:changetoBiBm}
Let $ (\Bi, \Ba) $ be an integral pair and assume that $ \Ba $ is weakly increasing (with respect to the usual partial order), and that up to reordering $\Bi=\Bi_\Bm$.  Let $ \pi =\pi_\Bi$ be the shortest permutation such that $ \pi(\Bi) = \Bi_\Bm$.  Then $\pi $ is $ (\Bi,\Ba)$-admissible.
\end{Lemma}

\begin{proof}
Suppose that $i_k \leftarrow i_l$.  Then $i_l$ is even and $i_k$ is odd.  By our choice of the ordering of $I$, $i_l$ must come before $i_k$ in $\Bi_\Bm$, and hence $\pi(l)<\pi(k)$.  Now if $a_k=a_l+1$ then $a_l \leq a_k$ and so $l<k$.  Therefore $ \pi(k), \pi(l) $ are in the same order as $ k, l$ and so $\pi$ is admissible.
\end{proof}

\begin{Lemma}
\label{lemma:pi_a}
Let $ (\Bi, \Ba)$ be an integral pair, and  let $\pi = \pi_\Ba $ be the shortest permutation of $ \{1, \dots, m \} $ such that $\pi(\Ba)$ is weakly increasing in the order $\preceq$ (cf. Section \ref{sec:coarsemetric}).  Then  $\pi$ is $(\Bi,\Ba)$-admissible.
\end{Lemma}

\begin{proof}
Suppose $i_k \leftarrow i_l$ and $a_k=a_l+1$.  Then $i_l$ is even node and $i_k$ is odd.  Hence $a_l=2q$ and $a_k=2q+1$ for some $q$, and so $a_k \approx a_l$.  Therefore, by construction, $\pi$ does not change the relative order of $k$ and $l$.
\end{proof}


\begin{Definition}
Fix $ \Bi, \Ba $.  Suppose that $s_k $ is $(\Bi, \Ba)$-admissible.  We define the \textbf{neutral crossing} $ \chi_k : W_{\Bi, \Ba} \rightarrow W_{s_k \Bi, s_k \Ba} $ as follows.  For a $\BK$ weight module $M$ and $ v\in W_{\Bi, \Ba}(M) $, we define
\begin{equation} \label{eq:neutral}
\chi_k(v) = \begin{cases}
     \psi_k \frac{1}{\barX_{i_k,i_{k+1}}(\yz_k,\yz_{k+1}) } v &
    i_{k}\neq i_{k+1}\\
    (\yz_k-\yz_{k+1})\psi_k v + v & i_{k}=i_{k+1}.
  \end{cases}
\end{equation}
Note that by the above admissible assumption $ \barX_{i_k, i_{k+1}}(\yz_k, \yz_{k+1}) $ is invertible on $ W_{\Bi,\Ba}(M) $.
\end{Definition}

Though $ \Pol $ is not a weight module, we have the following result which is proved by an elementary computation.
\begin{Lemma}
The neutral crossing $ \chi_k $ is well-defined on $ \Pol $ and acts by algebra automorphism $ s_k $.
\end{Lemma}

\begin{Lemma} \label{le:CommuteNeutral1}
Suppose that $ s_k $ is admissible for $ \Bi, \Ba $ and let $ 1 \le l \le m $.  Then
$$
e(\Bi) \chi_k = \chi_k e(s_k \Bi) \quad z_l \chi_k = \chi_k z_{s_k(l)} \quad \psi_l \chi_k = \chi_k \psi_{s_k(l)}
$$
\end{Lemma}
In fact, this lemma implies that $ \chi_k $ maps $ W_{\Bi, \Ba} $ to $W_{s_k \Bi, s_k \Ba}$.

\begin{proof}
Because $ \Pol $ is a faithful representation, we can check these relations in $ \End(\Pol) $ where they are obvious.
\end{proof}

\begin{Proposition}\label{prop:neutral-relations}
The neutral crossings obey the relations in the symmetric group in the following sense.
\begin{enumerate}
\item For any admissible elementary transposition, we have $ \chi_k^2 = I $.
\item If $ s_k $ and $ s_l $ are both admissible and $ |k-l|> 1 $, then $ \chi_k \chi_l = \chi_l \chi_k $.
\item If $ s_k $ and $ s_{k+1} $ are both admissible, then we have $
  \chi_k \chi_{k+1} \chi_k = \chi_{k+1} \chi_k \chi_{k+1} $.
\end{enumerate}
\end{Proposition}

\begin{proof}
As in the proof of the previous lemma, we can check these relations in $ \Pol $.
\end{proof}

As a consequence of this proposition and of Corollary \ref{cor:admissibleword}, we obtain the following.
\begin{Corollary} \label{co:AdPermIso}
Suppose that $ \pi $ is $ (\Bi, \Ba)$-admissible and let $ \Bi' = \pi(\Bi), \Ba' = \pi(\Ba) $.  Using a composition of neutral crossings, we get an isomorphism $ \chi_\pi : W_{\Bi, \Ba} \xrightarrow{\sim} W_{\Bi', \Ba'} $ which depends only on $ \pi $.
\end{Corollary}
\begin{Lemma} \label{le:CommuteNeutral2}
Suppose that $ \pi $ is $(\Bi, \Ba)$-admissible.  Let $ 1 \le l \le m $.  Then
$$
e(\Bi) \chi_\pi = \chi_\pi e(\pi(\Bi)) \quad z_l \chi_\pi = \chi_\pi z_{\pi(l)} \quad \psi_l \chi_\pi = \chi_\pi \psi_{\pi(l)}
$$
\end{Lemma}

There are other morphisms between weight functors which will be useful to us, which we call the seam crossings.
\begin{Definition}  Given $ (\Bi, \Ba) \in I^m \times \C^m $, we have $ \sigma(\Bi) = (i_m, i_1, \dots, i_{m-1}) $ and $ \sigma(\Ba) = (a_m - 2, a_1 \dots, a_{m-1}) $.  We define the \textbf{rightward seam crossing} $ \hsigma_+ : W_{\Bi, \Ba} \to W_{\sigma(\Bi), \sigma(\Ba)} $ by applying the rightward crossing of the dashed line $ x = 0 $.  We similarly define the \textbf{leftward seam crossing} $ \hsigma_- : W_{\Bi, \Ba} \to W_{\sigma^{-1}(\Bi), \sigma^{-1}(\Ba)} $.
\end{Definition}

We remark that the appearance of the minus 2 in the definition of $\sigma(\Ba)$ above is due to relation (\ref{dot-slide}), and ensures that $ \hsigma_+ $ defines a natural transformation $ W_{\Bi, \Ba} \to W_{\sigma(\Bi), \sigma(\Ba)} $.

\subsection{An equivalence relating the coarse metric KLRW algebra and  \foreignlanguage{russian}{Я}}
\label{sec:equivalence-TL-Ya}
Given an integral pair $ (\Bi, \Ba) $, we will define an idempotent $ d(\Bi, \Ba)\in \TL $ as follows.
We let $ \kappa(p)$ be the number of indices $i$ such that $a_i\prec r_p$; note that $\kappa(2q)=\kappa(2q+1)$ for all $q\in \Z$.  Set $\pi=\pi_\Ba$ (cf. Lemma \ref{lemma:pi_a}).
Then $ (\pi(\Bi), \kappa, \pi(\Ba)) $ is a coarse longitude triple,
and we define $ d(\Bi, \Ba) := e(\pi(\Bi), \kappa, \pi(\Ba)) \in \TL$.
%
This defines a map
$$
\{ \text{ integral pairs $ (\Bi, \Ba) $ } \} \rightarrow \{ \text{ idempotents for $ \TL $ }\}
$$
This map admits a 1-sided inverse by simply mapping the idempotent $ e(\Bi, \kappa, \Ba) $ to $ (\Bi, \Ba)$.

Let $ \TL\operatorname{-mod}_{\operatorname{nil}} $ denote the category of finitely-generated $ \TL $-modules where the dots act nilpotently.  Given $ M \in \TL\operatorname{-mod}_{\operatorname{nil}}$, we consider the action of various generators of $\TL$ on $M$.  Firstly, for any integral pair $(\Bi,\Ba)$ and any $1\leq k\leq m$ we have an operator
\begin{equation}
\label{eq:y_k}
y_k:d(\Bi,\Ba)M \to d(\Bi,\Ba)M
\end{equation}
given by the action of $y_k(\pi(\Bi),\kappa,\pi(\Ba)) \in \TL$ (cf. the discussion following Definition \ref{def:coarselongs}).

Secondly, we consider the action of the crossing $\psi_{\pi(k)}$ on $d(\Bi,\Ba)M$, where as above $\pi=\pi_\Ba$.  Since we are working in $\TL$ this crossing must carry coarse longitudes on the top and bottom.  If we hope to define a nonzero operator, on the bottom $\psi_{\pi(k)}$ must carry $(\pi(\Bi),\kappa,\pi(\Ba))$.  On the top the natural choice is $(s_{\pi(k)}\pi(\Bi),\kappa,s_{\pi(k)}\pi(\Ba))$.  But note  that if $a_{k}\not\approx a_{k+1}$ then this is not necessarily a coarse longitude.  Therefore we consider the operator $\psi_{\pi(k)}$ on $d(\Bi,\Ba)M$ only when $a_k \approx a_{k+1}$, and the definition of this operator breaks up into two cases:
\begin{align}
\psi_{\pi(k)}: d(\Bi,\Ba)M \to d(\Bi,\Ba)M \quad &\text{ if } i_k=i_{k+1},  \label{eq:psi_k-equal} \\
\psi_{\pi(k)}: d(\Bi,\Ba)M \to d(s_k\Bi,s_k\Ba)M \quad &\text{ if } i_k\neq i_{k+1}  \label{eq:psi_k-nequal}
\end{align}
In the first case the crossing is given top longitude $(\pi(\Bi),\kappa,\pi(\Ba))$, and in the second case $(\pi(s_k\Bi),\kappa,\pi(s_k\Ba))$.  Notice that since $a_k \approx a_{k+1}$, by the definition of $\pi=\pi_\Ba$, we have that $\pi(k+1)=\pi(k)+1$.

\begin{Lemma}\label{lem:Theta}
There is a functor  $\Theta \colon \TL\operatorname{-mod}_{\operatorname{nil}}\to \wtmodBK$, such that for any integral pair $ (\Bi, \Ba) $,
  \begin{equation}
W_{\Bi,\Ba}(\Theta(M)) =
d(\Bi,\Ba)M.\label{eq:weight-match}
\end{equation}

\end{Lemma}

\begin{proof}
Let $ M \in \TL\operatorname{-mod}_{\operatorname{nil}}$. We begin by defining $ \Theta(M) $ as a vector space by
$$
\Theta(M) = \bigoplus_{(\Bi, \Ba) \text{ integral }}  d(\Bi, \Ba) M \delta(\Bi, \Ba)
$$
where $ \delta(\Bi,\Ba)$ is a formal symbol.  This formal symbol is introduced to distinguish between different pairs $(\Bi,\Ba)$ which give rise to the same idempotent $d(\Bi,\Ba)$.

%

Recall that the algebra $ \BK $ is generated by the dots $ \yz_k$, the crossings $\psi_k$, and the seam crossings of $\sigma_\pm$.  On $d(\Bi, \Ba) M \delta(\Bi, \Ba)$ we consider operators as in (\ref{eq:y_k}-\ref{eq:psi_k-nequal}), and define the action of $ \BK $ on $ \Theta(M) $ as follows:
\begin{itemize}
\item We let $\yz_k\in \BK$ act by
$$
\yz_k v \delta(\Bi, \Ba) = (y_{\pi_\Ba(k)} + a_k)v \delta(\Bi, \Ba)
$$
Note that this has the correct eigenvalue, since
  $y_{\pi_\Ba(k)}$ is nilpotent.
\item We let the crossing $\psi_k$ of two strands
  act by
  \begin{equation}\label{eq:psi-image}
\psi_k v \delta(\Bi,\Ba)=
    \begin{cases}
      \psi_{\pi(k)}v \delta(\Bi,\Ba) & i_{k}=i_{k+1}, a_{k}\approx a_{k+1}\\
     \frac{1} {y_{\pi(k+1)}-y_{\pi(k)}+a_k-a_{k+1}} v\delta(\Bi,\Ba)&
    i_{k}=i_{k+1}, a_{k}\not\approx a_{k+1}  \\
\qquad +\frac{1} {y_{\pi(k)}-y_{\pi(k+1)}+a_{k}-a_{k+1}} v
\delta(\Bi,s_k\Ba)\\
\psi_{\pi(k)} v \delta(s_k\Bi,s_k\Ba)&  i_{k}\neq i_{k+1}, a_{k}\approx a_{k+1}\\
\barX_{i_ki_{k+1}}(y_{\pi_{\Ba}(k)}+a_k,&  i_{k}\neq i_{k+1}, a_{k}\not\approx a_{k+1}\\
\qquad y_{\pi_{\Ba}(k+1)}+a_{k+1}) v \delta(s_k\Bi,s_k\Ba)\\
    \end{cases}
  \end{equation}
As a sanity check on these formulae, note that when we apply $\psi_k$ in $\BK$ to a weight vector, we
  get a sum of weight vectors for the weights $(s_k\Bi, \Ba)$ and
  $(s_k\Bi,s_k\Ba)$ by
  (\ref{first-QH}--\ref{nilHecke-2}). If $i_{k}\neq i_{k+1}$, then we only have a
  non-zero term in the second case; if $i_{k}=i_{k+1}$ and
  $a_k=a_{k+1}$, then these two weight spaces are the same.  This will only change the idempotent $d(\Bi,\Ba)$ if
  $a_k\approx a_{k+1}$.  Note also that in the last case (where $i_k\neq i_{k+1}$ and $a_k\not\approx a_{k+1}$) that $d(s_k\Bi,s_k\Ba)=d(\Bi,\Ba)$, so given $v\delta(\Bi,\Ba)$ then $ v\delta(s_k\Bi,s_k\Ba)$ also makes sense.
\item We send $\sigma_+$  to $\phi^+_{\pi_{\Ba}(m)}$ as defined in Section \ref{sec:coarsemetric}, which is the unique
  straight-line diagram attaching $d(\Bi,\Ba)$ to
  $d((i_m,i_1,\dots,i_{m-1}),(a_m+2,a_1,\dots a_{m-1}))$
 with the
  least number of crossings.
 The only crossings in this diagram involve the single black
  strand with bottom longitude $a_m$, and it only crosses the red strands for elements of $R_j$ in
  the interval $a_m<r\leq a_m+2$.
\item We send ${\sigma}_-$  to the
precomposition of the polynomial
  \[\prod_{\substack{r\in R_i \\ r\neq a_1}} (y_{\pi(1)}-r+a_1).\]
  with the diagram $\phi^-_{\pi_{\Ba}(1)}$, which is the straight-line diagram that joins  $d(\Bi,\Ba)$ to
  $d((i_2,\dots,i_{m},i_1),(a_2,\dots a_{m},a_1-2))$.
The only crossings in this diagram involve the single black strand with bottom longitude $a_1$, and it only crosses the red strands for elements of $R_j$ in
  the interval $a_1-2<r\leq a_1$.
\end{itemize}

We will now check that these formulae define an action of $ \BK $ on $
\Theta(M)$.
Consider the subcategory of $\TL$-modules for which this does define a $\BK$-action.  Obviously, this set is closed under taking submodules, quotients, and direct sums.

Recall that we have a faithful  $\TL$-module $\PolKLR_{\mathcal{L}}$.
Let $\PolKLR_{\mathcal{L}}^{(N)}$ denote the quotient of $\PolKLR_{\mathcal{L}}$ by the ideal $I^{(N)}$ generated by all symmetric polynomials in the $Y_i$'s of degree $\geq N$.
It's clear from (\ref{eq:T-action1}--\ref{eq:T-action2}) that $I^{(N)}$ is invariant under $\TL$, hence $\PolKLR_{\mathcal{L}}^{(N)}$ has a well-defined $\TL$ action. Note that any $Y_i$ must be nilpotent on this module, since the coinvariant algebra is finite dimensional.

Since $\TL$ is faithful on the inverse limit of this system of modules, every module on which the dots are nilpotent is a subquotient of the direct sum of these modules.  Thus, it suffices to check that we have a $\BK$ action when applying the
construction $\Theta$ to $\PolKLR_{\mathcal{L}}^{(N)}$ for all $N$.

We have that $\Theta(\PolKLR_{\mathcal{L}}^{(N)})=\bigoplus d(\Bi,\Ba)\PolKLR_{\mathcal{L}}^{(N)}\delta(\Bi,\Ba)$.
Applying the formulae (\ref{eq:T-action1}--\ref{eq:T-action2}), we find that on $f\delta(\Bi,\Ba)\in d(\Bi,\Ba)\PolKLR_{\mathcal{L}}^{(N)}\delta(\Bi,\Ba)$:
\begin{itemize}
\item
The generator $\yz_k$ acts by multiplication by $Y_{\pi_{\Ba}(k)}+a_k$.
\item The crossing $\psi_k$ maps $f \delta(\Bi,\Ba)$ to:
\[
\begin{cases}
\displaystyle
  \frac{f-f^{(\pi_{\Ba}(k),\pi_{\Ba}(k+1))}}{Y_{\pi_{\Ba}(k)}-Y_{\pi_{\Ba}(k+1)}}\delta(\Bi,\Ba) &
  i_k=i_{k+1}, a_k\approx a_{k+1}\\
\frac{f\delta(\Bi,\Ba)} {Y_{\pi(k+1)}-Y_{\pi(k)}+a_k-a_{k+1}} +\frac{ f\delta(\Bi,s_k\Ba) } {Y_{\pi(k)}-Y_{\pi(k+1)}+a_{k}-a_{k+1}}&
  i_k=i_{k+1}, a_k\not\approx a_{k+1}\\
  X_{i_ki_{k+1}}(Y_{\pi_{\Ba}(k)},Y_{\pi_{\Ba}(k+1)})
  f^{(\pi_{\Ba}(k),\pi_{\Ba}(k+1))} \delta(s_k\Bi,s_k\Ba)& i_k\neq i_{k+1}, a_k\approx a_{k+1}\\
  \barX_{i_ki_{k+1}}(Y_{\pi_{\Ba}(k)}+a_k,Y_{\pi_{\Ba}(k+1)}+a_{k+1})
  f^{(\pi_{\Ba}(k),\pi_{\Ba}(k+1))} \delta(s_k\Bi,s_k\Ba)& i_k\neq i_{k+1}, a_k\not\approx a_{k+1}
\end{cases}
\]

\item The seam crossings $\sigma_{\pm}$ act by
  \begin{align*}
\sigma_{+}\cdot f \delta(\Bi,\Ba)&= f \delta ((i_m,i_1,\dots,i_{m-1}),(a_m+2,a_1,\dots a_{m-1}))\\ \sigma_{-}\cdot f \delta(\Bi,\Ba)&= \prod_{r\in R_i}
  (Y_{\pi(1)}-r+a_1)f \delta((i_2,\dots,i_{m},i_1),(a_2,\dots a_{m},a_1-2))
  \end{align*}
since the leftward crossing crosses one red strand with the same label
for each time $a_1$ appears in $r$.
\end{itemize}

If we change basis by the formula
$(Y_{\pi_{\Ba}(k)}+a_k)\delta(\Bi,\Ba)=\YZ_{k}\delta(\Bi,\Ba)$, then these formulas become
\begin{align*}
  \yz_k\cdot f (\mathbf{\YZ})\delta(\Bi,\Ba)&=Z_k f (\mathbf{\YZ})\delta(\Bi,\Ba) \\
\psi_k\cdot f  (\mathbf{\YZ})\delta(\Bi,\Ba)&=\begin{cases} \displaystyle
  \frac{f-f^{s_k}}{Z_k-Z_{k+1}}\delta(\Bi,\Ba) &
  i_k=i_{k+1}, a_k\approx a_{k+1}\\
\frac{ f } {Z_k-Z_{k+1}}\delta(\Bi,\Ba)+\frac{f^{s_k}} {Z_k-Z_{k+1}} \delta(\Bi,s_k\Ba)&
  i_k=i_{k+1}, a_k\not \approx a_{k+1}\\
  \barX_{i_ki_{k+1}}(Z_k,Z_{k+1})
  f^{s_k} \delta(s_k\Bi,s_k\Ba)& i_k\neq i_{k+1}
\end{cases}\\
\sigma_{+}\cdot f  (\mathbf{\YZ})\delta(\Bi,\Ba)&= \sigma(f) \delta ((i_m,i_1,\dots,i_{m-1}),(a_m+2,a_1,\dots a_{m-1}))\\
\sigma_{-}\cdot f  (\mathbf{\YZ})\delta(\Bi,\Ba)&= \prod_{r\in R_i}
  (Z_1-r) \sigma^{-1}(f)\delta((i_2,\dots,i_{m},i_1),(a_2,\dots a_{m},a_1-2))
\end{align*}
Note that in the case of the $\psi_k$ action, the formula breaks up into three cases (instead of four), since if $i_k \leftarrow i_{k+1}$ and $a_k \approx a_{k+1}$ then $a_k=2q+1$ and $a_{k+1}=2q$ for some $q$.  Hence $X_{i_k,i_{k+1}}(Y_{\pi_{\Ba}(k)},Y_{\pi_{\Ba}(k+1)})=\barX_{i_k,i_{k+1}}(Z_k,Z_{k+1})$ by the equality
\begin{align*}
X_{i_ki_{k+1}}(Y_{\pi_{\Ba}(k)},Y_{\pi_{\Ba}(k+1)})&=X_{i_ki_{k+1}}(\YZ_{k}-a_k,\YZ_{k+1}-a_{k+1})\\&=\bar{X}_{ij}(\YZ_k, \YZ_{k+1}+a_k-a_{k+1}-1)
\end{align*}

Finally note that the quotient of polynomials in $\C[Y_1,\dots, Y_m]$ by symmetric polynomials of degree $\geq N$ is equal to the quotient of polynomials $\C[\YZ_1,\dots, \YZ_m]$ by the symmetric polynomials of degree $\geq N$ in the shifted variables $\{\YZ_k-a_k\}$.   That is, we
find $\Theta(\PolKLR_{\mathcal{L}}^{(N)})$ is the sum over integral weights
$\Ba=(a_1,...,a_m)$ of the quotient of $\Pol_m$ by $I^{(N)}$ shifted to this point (cf. Remark \ref{rem:Polm}).  Comparing with the formulas
for
the polynomial action defined in
(\ref{eq:black-cross-action}--\ref{eq:x-cross-left}) shows that the action defined via $\Theta$ agrees   with the induced
action of $\BK$ by those formulas.  Thus, we indeed get a $\BK$ action in this case, so we do in the general case as well.
\excise{\begin{itemize}
\item The relations (\ref{dot-slide}) say that the nilpotent part of the
  dot commutes past $\hat{\sigma}_\pm$, but that its semi-simple part
  changes eigenvalue.  Indeed the straight line diagrams that $
  \hat{\sigma}_\pm$ are sent to commute with dots, and by definition,
  they change the weight space as needed.
\item The relations (\ref{x-cost}) show that $\hat{\sigma}_+ \hat{\sigma}_-$
  give multiplication by a particular polynomial $p_i$
  This multiplication is  the image under our map of $  \prod_{r\in
    R_i} (y_{\pi(1)}-r+a_1)$.  The order of vanishing of this
  polynomial at $y_{\pi(1)}=0$ is (by definition) the number of red
  strands with label $i$ that our image of $\hat{\sigma}_\pm$
  crosses.  Thus, by (\eqref{cost}), this image is the same as
  $\prod_{\substack{r\in R_i \\ r\neq a_1}} (y_{\pi(1)}-r+a_1)$ times
  a bigon crossing these red strands.  This proves the relation.
\item The LHS of the relation (\ref{x-triple-smart}) is sent to
\end{itemize}}
\end{proof}

\begin{Lemma}
There is a functor $ \Gamma : \wtmodBK \rightarrow \TL\operatorname{-mod}_{\operatorname{nil}} $ such that for all coarse longitude triples, we have
$$
e(\Bi, \kappa, \Ba) \Gamma(N) = W_{\Bi, \Ba}(N)
$$
\end{Lemma}
\begin{proof}
Let $N $ be a $\BK$ weight module with integral weights.  We'll define an action
of $\TL$ on the vector space
$$
\Gamma(N) = \bigoplus_{\Bi, \Ba \text{ weakly $\preceq$-increasing }} W_{\Bi, \Ba}(N)
$$
Note that $\Gamma(N) $ will in general be smaller than $ N $, since typically $N $ will have weights which are not weakly $ \preceq$-increasing.

Recall the generators of the algebra $\TL$ defined in Section \ref{sec:coarsemetric}.
We let:
\begin{itemize}
\item the dots $y_k\in \TL$ act by the
  nilpotent part $\yz_k-a_k\in \BK$.
\item the crossing $\psi_k\in \TL$ of two strands with longitudes satisfying
  $a_k\approx a_{k+1}$ is simply given by
  $\psi_k\in \BK$
\item
 The generator $\phi_k^+$ is sent to the composition of morphisms between weight functors
  \begin{equation}
  \chi_{k-1} \cdots \chi_1 \sigma_+ \chi_{m-1} \cdots
  \chi_k\label{eq:wrap1}
\end{equation}
  In other words, in $ \BK $, we wrap the $k$th strand one full rightward twist around the
  cylinder crossing each other strand once, using a neutral crossing each time.
\item
  The generator $\phi_k^-$ is sent to the composition
  \begin{equation}
  \prod_{\substack{r\in R_i \\ r\neq a_k}} (y_{k}-r)^{-1}  \chi_k
  \cdots \chi_{m-1} \sigma_- \chi_1 \cdots \chi_{k-1}\label{eq:wrap2}
\end{equation}
  In other words, in $ \BK $, we wrap the $k$th strand one full leftward twist around the
  cylinder,  crossing each other strand once, using a neutral crossing each time, multiplied by the above rational function.
\end{itemize}
As in the proof of Lemma \ref{lem:Theta}, we can check this by considering the subcategory of weight modules on which this does define a $\TL$ action.  We would like to claim that the polynomial representation $\Pol_\BK$ lies in this subcategory.  This is slightly tricky since of course, the polynomial representation is not a weight module, but aiming to invert the proof of Lemma \ref{lem:Theta} shows how we can get around this difficulty.  Let $I_{\Ba}^{(N)}$ be the ideal in $\C[\mathbf{Z}]$ generated by symmetric functions of degree $\geq N$ in $Z_i-a_i$.  The action of $\BK$ on $\Pol_\BK$ induces one on the sum of quotients $\sum_{\Ba}\Pol_\BK/I_{\Ba}^{(N)}$ which is a weight module.
The result of applying the functor $\Gamma$ just gives the polynomial
representation of $\TL$ modulo the ideal $I^{(N)}$, since
the formulas above just invert those of  Lemma \ref{lem:Theta}. Thus, these modules lie in the category where $\Gamma$ defines a $\TL$-module structure; the inverse limit of this system of modules is faithful, so every weight module is a subquotient of a direct sum of them.

Thus, $\Gamma$ is well-defined as a functor.  By direct construction, this functor sends weight modules to
modules with finite dimensional images of $e(\Bi,\Ba)$ where the dots
are nilpotent, and has the desired action on weight spaces.
\end{proof}

\begin{Theorem}
\label{thm:bigequiv}
The functors $ \Theta :  \TL\operatorname{-mod}_{\operatorname{nil}} \rightarrow \wtmodBK $ and $ \Gamma : \wtmodBK \rightarrow\TL\operatorname{-mod}_{\operatorname{nil}}  $ give mutually inverse equivalences
$$
\wtmodBK \cong \TL\operatorname{-mod}_{\operatorname{nil}}
$$
\end{Theorem}

\begin{proof}
Let $ N \in \wtmodBK $.  Unpacking the definitions from the previous two lemmas,
$$\Theta(\Gamma(N)) = \bigoplus_{(\Bi, \Ba) \text{ integral}}
W_{\pi_\Ba(\Bi),\pi_\Ba(\Ba)}(N) \delta(\Bi, \Ba)  $$

That is, the $(\Bi,\Ba)$-weight space of $\Theta(\Gamma(N))$ is given
by the $(\pi_\Ba(\Bi),\pi_\Ba(\Ba))$-weight space of $ N$.
We must construct an isomorphism $ \Theta(\Gamma(N)) \cong N $, which
will be constructed one weight space at a time as an isomorphism
$W_{\Bi,\Ba}(N)\cong W_{\pi_\Ba(\Bi),\pi_\Ba(\Ba)}(N)$.

This will be supplied by the neutral crossings.  For each $ \Bi, \Ba $, $ \pi_\Ba $ is an admissible permutation and so we have a neutral crossing map $ \chi_{\pi_\Ba} : W_{\Bi, \Ba}(N) \rightarrow W_{\pi_\Ba(\Bi), \pi_\Ba(\Ba)}(N)$.  We write $ \chi : N \rightarrow \Theta(\Gamma(N)) $ for the direct sum of these maps.  We must prove that $ \chi $ is a map of $ \BK $-modules.  To do this we will show that for each of the generators $ r $ of $ \BK $, we have $ \chi(r v) = r \chi(v) $ for all $ v \in N $.

Because the modules $ N ,  \Gamma(N) $ and $ \Theta(\Gamma(N)) $ have closely related underlying vector spaces, we will introduce the following temporary notation.  Given an element $ u \in W_{\pi_\Ba(\Bi), \pi_\Ba(\Ba)}(N) $, we will write $ \overline{u}$ when it is regarded as an element of $\Gamma(N) $ and $ \overline{u}\delta(\Bi, \Ba) $ when it is regarded as an element of $ W_{\Bi, \Ba}(\Theta(\Gamma(N))) $.  So with this notation in mind, we see that we need to show that
\begin{equation} \label{eq:toprove}
 \overline{\chi_{\pi_{r(\Ba)}}(r v)} \delta(r(\Bi), r(\Ba)) = r (\overline{\chi_\pi(v)}\delta(\Bi, \Ba))
 \end{equation}
 for each generator $ r $ and each $ v \in W_{\Bi, \Ba}(N)$. (Here $ r(\Bi), r(\Ba) $ indicates the effect of $ r $ on the weights; if $r v$ is not a weight vector, then we will have a sum of terms instead.)

%
The generators of $ \BK $ are the dots $ z_k $, the crossings $ \psi_k $, and the seam crossings $ \sigma_{\pm} $.  We will check these generators in turn.

We begin with $ r = z_k $ and $ v \in W_{\Bi, \Ba}(N)$.   Let $ \pi = \pi_\Ba $.  In this case $ r(\Bi) =\Bi, r(\Ba) = \Ba$.  Applying Lemma \ref{le:CommuteNeutral2} we have
$$
\chi_\pi(z_k v) = z_{\pi(k)} \chi_\pi(v)
$$
and on the other hand, we have
$$
z_k (\overline{\chi(v)} \delta(\Bi, \Ba)) =  (y_{\pi(k)} + a_k)\overline{\chi(v)} \delta(\Bi, \Ba) = \overline{z_{\pi(k)} \chi(v)} \delta(\Bi, \Ba)
$$
and so (\ref{eq:toprove}) holds for the dots.

Now we consider $ r = \psi_k $.  There will be three cases to consider.

\noindent \textbf{Case 1}: $ a_k \approx a_{k+1} $.  In this case, we find that $ \pi_\Ba = \pi_{s_k \Ba} $ and we write $ \pi = \pi_\Ba $.  Thus applying Lemma \ref{le:CommuteNeutral2} we have
$$
\overline{\chi_\pi(\psi_k v)} \delta(s_k \Bi, s_k \Ba)  = \overline{\psi_{\pi(k)} \chi_\pi(v)} \delta(s_k \Bi, s_k \Ba)
$$
and on the other hand applying \eqref{eq:psi-image}, we have
$$
\psi_k (\overline{\chi_\pi(v)}\delta(\Bi, \Ba)) = \psi_{\pi(k)} \overline{\chi_\pi(v)} \delta(s_k \Bi, s_k \Ba) =  \overline{ \psi_{\pi(k)} \chi_\pi(v)} \delta(s_k \Bi, s_k \Ba)
$$
so (\ref{eq:toprove}) holds in this case.

\noindent \textbf{Case 2}: $ i_k = i_{k+1} $ and $ a_k \not\approx a_{k+1} $.  This case is slightly more complicated since $ \psi_k v $ will not be a weight vector and in order to compute $ \chi(\psi_k(v)) $ we need to write it as a sum of weight vectors.  Applying the ``dot/crossing'' relation and the definition of the neutral crossing, we see that $$ \psi_k v = \frac{1}{z_k - z_{k+1}} \chi_k(v) + \frac{1}{z_{k+1} - z_k} v $$ is its expansion as a sum of weight vectors; the first term lies in the $ (s_k \Bi, s_k \Ba) $ weight space, whereas the second term lies in the $(s_k \Bi, \Ba) $ weight space.

There are actually two subcases to consider here, depending on which of $ a_k, a_{k+1} $ is larger.  These two cases are very similar, so let us assume that $ a_{k+1} $ is larger, so that $ \pi_{s_k \Ba} = \pi_\Ba s_k $.  In particular, this means that
$$
\chi(\psi_k v) = \overline{\chi_\pi \chi_k \frac{1}{z_k - z_{k+1}} \chi_k(v)} \delta(s_k \Bi, s_k \Ba) + \overline{\chi_\pi(\frac{1}{z_{k+1} - z_k} v)} \delta(s_k \Bi, \Ba)
$$
On the other hand applying the definitions from the previous two lemmas, especially \eqref{eq:psi-image}, we see that
$$
\psi_k(\chi(v)) =  \overline{ \frac{1}{z_{\pi(k+1)} - z_{\pi(k)}} \chi_\pi(v)} \delta(s_k \Bi, s_k \Ba) + \overline{\frac{1}{z_{\pi(k+1)} - z_{\pi(k)}} \chi_\pi(v)} \delta(s_k \Bi, \Ba)
$$
Since $ \chi_k^2 = 1$ and applying Lemma \ref{le:CommuteNeutral2}, we see that \eqref{eq:toprove} holds in this case.

\noindent \textbf{Case 3}: $ i_k \ne i_{k+1} $ and $ a_k \not\approx a_{k+1} $.  In this case $ \chi_k v $ is a weight vector of weight $( s_k \Bi, s_k \Ba )$.  Again, we have two subcases and we assume that $ a_{k+1} $ is larger which leads us to $ \pi_{s_k \Ba} = \pi_\Ba s_k $.  Then we have
$$
\chi(\psi_k(v)) = \overline{ \chi_\pi \chi_k (\psi_k(v))} \delta(s_k \Bi, s_k \Ba)
$$
and applying the definitions, we have
$$
\psi_k(\chi(v)) = \overline{ \barX_{i_k i_{k+1}}(z_{\pi(k)}, z_{\pi(k+1)}) v} \delta(s_k\Bi, s_k\Ba))
$$
The result follows upon noting that $ \chi_k \psi_k = \barX_{i_k i_{k+1}}(z_{\pi(k)}, z_{\pi(k+1)}) $ by the definition of the neutral crossing and the $\chi_k^2 $ relation.

Finally, we consider the positive seam crossing $\sigma_+$ (the negative one is similar).  As above let $ v = W_{\Bi, \Ba}(N) $ and let $ \pi = \pi_\Ba $.  Then $ \sigma_+(v) \in W_{\sigma(\Bi), \sigma(\Ba)} $ where as before
$$
\sigma(\Bi) = (i_m, \dots, i_1), \quad \sigma(\Ba) = (a_m + 2, a_1, \dots, a_{m-1})
$$
Let us write $ (b_1 \dots, b_m) = \pi(\Ba) $ and let $k $ be the maximal index such that $ b_k = m $ (so that $ \pi(m) = k $). By the definition of partial order $ \preceq $, we have that
$$ \pi_{\sigma(\Ba)}(\sigma(\Ba)) = (b_1, \dots, b_{k-1}, b_k+2, b_{k+1}, \dots, b_m)$$
(the strange definition of $ \preceq $ ensures that this is weakly-$\preceq$ increasing).  We conclude $ \pi_{\sigma(\Ba)} = \pi s_{m-1} \cdots s_1 $.

Now, let us examine the LHS of (\ref{eq:toprove}) in this case.  We see that it equals
$$
\overline{\chi_\pi \chi_{m-1} \cdots \chi_1 \hat{\sigma}_+ v} \delta(\sigma(\Bi), \sigma(\Ba))
$$
On the other hand, following the definitions of the previous two lemmas, the RHS of (\ref{eq:toprove}) is
$$
\sigma_+ \overline{\chi_\pi v} \delta(\Bi, \Ba) = \overline{\chi_{k-1} \cdots \chi_1 \sigma_+ \chi_{m-1} \cdots \chi_k \chi_\pi v} \delta(\sigma(\Bi), \sigma(\Ba))
$$
and so the result follows.

\excise{
diagram $\hat{\sigma_{+}}$
acts in $\Theta(\Gamma(N))$ by a diagram that grabs the
$\pi_{\Ba}(n)$th strand (reading in the positive direction from the
seam), and wraps it in the positive direction past the seam, using
neutral crossings, until it reaches the $\pi_{\Ba}(1)$th strand.
Thus, this is just $\hat{\sigma_{\pm}}$ pre- and post-composed with
neutral crossings.  As
in previous cases, the neutral crossings we have added just account for the
changing isomorphisms $W_{\Bi,\Ba}(N)\cong W_{\pi_\Ba(\Bi),\pi_\Ba(\Ba)}(N)$.
}

On the other hand, let $ M \in
\TL\operatorname{-mod}_{\operatorname{nil}}$.  Then from the
definitions, we naturally have $ \Gamma(\Theta(M)) = M $.  Checking
that the module action is correct just runs the logic of the previous
paragraph backwards.
\end{proof}

\excise{
\begin{Lemma}
  There is an equivalence $\tilde\vartheta'\colon \TL\operatorname{-mod}_{\operatorname{nil}}\to \mathcal{W}^{\bR}(\BK)$ such that
  \begin{equation}
W_{\Bi,\Ba}(\tilde\vartheta' M)\cong
d(\Bi,\Ba)M.\label{eq:weight-match}
\end{equation}
\end{Lemma}
\begin{proof}
Of course, in order to prove this, we need to construct functors in
both directions.  The vector space underlying these functors is
implicit in \eqref{eq:weight-match}, but we can discuss it a little
more carefully.

Given a $\TL$-module $M$, we need to define an action of $\BK$
on \[\tilde\vartheta' M\cong \oplus_{\Bi,\Ba}d(\Bi,\Ba)M\cdot \delta(\Bi,\Ba),\] where $
\delta(\Bi,\Ba)$ is just a formal symbol to remind
us that $d(\Bi,\Ba)M\cdot
\delta(\Bi,\Ba)$ should be the $\Ba$-weight space of $e(\Bi)
\tilde\vartheta' (M)$, since often different pairs $(\Bi,\Ba)$ will
give the same idempotent $d(\Bi,\Ba)$.
\begin{itemize}
\item We let the dots $y_m\in \BK$ act by $y_{\pi(m)}+a_m\in
  \TL$; note that this has the correct eigenvalue, since
  $y_{\pi(m)}$ is nilpotent.
\item We let the crossing $\psi_k$ of two strands
  act by
  \begin{equation*}
\psi_k m \delta(\Bi,\Ba)=
    \begin{cases}
      \psi_{\pi(k)}m \delta(\Bi,\Ba) & i_{k}=i_{k+1}, a_{k}\approx a_{k+1}\\
     \frac{1} {y_{\pi(k+1)}-y_{\pi(k)}+a_k-a_{k+1}} m\delta(\Bi,\Ba)&
    i_{k}=i_{k+1}, a_{k}\not\approx a_{k+1}  \\
\qquad +\frac{1} {y_{\pi(k)}-y_{\pi(k+1)}+a_{k}-a_{k+1}} m
\delta(s_k\Bi,s_k\Ba)\\
\psi_{\pi(k)} m \delta(s_k\Bi,s_k\Ba)&  i_{k}\neq i_{k+1}, a_{k}\approx a_{k+1}\\
P_{i_ki_{k+1}}(y_k+a_k,y_{k+1}+a_{k+1}) m \delta(s_k\Bi,s_k\Ba)&  i_{k}\neq i_{k+1}, a_{k}\not\approx a_{k+1}\\
    \end{cases}
  \end{equation*}
As a sanity check on these formulae, note that when we apply $\psi_k$ in $\BK$ to a weight vector, we
  get a sum of weight vectors for the weights $(\Bi, \Ba)$ and
  $(s_k\Bi,s_k\Ba)$ by
  (\ref{first-QH}--\ref{nilHecke-2}). If $i_{k}\neq i_{k+1}$, then we only have a
  non-zero term in the second case; if $i_{k}=i_{k+1}$ and
  $a_k=a_{k+1}$, then these two weight spaces are the same.  This will only change the idempotent $d(\Bi,\Ba)$ if
  $a_k\approx a_{k+1}$.
\item We send the rightward crossing over the seam $x=0$ to the unique
  straight-line diagram attaching $d(\Bi,\Ba)$ to
  $d((i_m,i_1,\dots,i_{m-1}),(a_m+1,a_1,\dots a_{m-1}))$ with the
  least number of crossings.
\item We send the leftward crossing over the seam $x=0$ to the
  straight-line diagram joining $d(\Bi,\Ba)$ to
  $d((i_2,\dots,i_{m},i_1),(a_2,\dots a_{m},a_1-1))$ times the polynomial
  \[\prod_{\substack{r\in R_i \\ r\neq a_1}} (y_{\pi(1)}-r+a_1).\]
\end{itemize}

This defines the functor $\tilde{\vartheta}'$.  Note that it sends nilpotent
representations to weight representations with the correct weight by
direct construction.

To define it the other
direction, we begin with a $\BK$-module $N$.  We'll define an action
of $\TL$ on $(\vartheta)^{-1}N\cong
\oplus_{(\Bi,\Ba)}W_{\Bi,\Ba}(N)$ where the sum ranges over all sequences $\Bi$ and
compatible
coarse longitudes $\Ba$ of the $\Ba$-weight space of $e(\Bi)N$. Note
that this does not imply that $N$ and $(\vartheta)^{-1}N$ have the
same underlying vector spaces, since $\BK$ can have weights which are
not coarse longitudes.

The algebra $\TL$ is generated by
\begin{itemize}
\item dots,
\item  the crossing $\psi_k$ of two strands with
  the longitudes in $a_k\approx a_{k+1}$, leaving these longitudes unchanged,
\item diagrams
  where the rightmost strand with longitude in $a_k\approx 2q$ increases
  longitude by $2$, crossing all red strands with longitude $\approx 2q+2$ and
\item the mirror image diagram
  where the leftmost strand with longitude in $a_k\approx 2q$ increases
  longitude by $2$, crossing all
red strands with longitude $\approx 2q$.
\end{itemize}
In order to define this action, we must define an element of a
suitable completion of $\BK$ (which we leave to the reader to make
explicit) which we call the {\bf neutral
    crossing}.  The action of this on a weight vector for a weight
  with $a_k\not \approx a_{k+1}$ is given by
  \begin{equation}
 \chi_k m= \begin{cases}
    \displaystyle \psi_k\frac{1}{p_{i_ki_{k+1}}(y_k,y_{k+1}) }m &
    i_{k}\neq i_{k+1}\\
    (y_k-y_{k+1})\psi_km+m & i_{k}=i_{k+1}.\\
  \end{cases}\label{eq:neutral}
\end{equation}

Applying the formulas of (\ref{eq:neutral})
in the polynomial
  representation show that this acts by the usual permutation of variables.
We let:
\begin{itemize}
\item the dots $y_k\in \TL$ act by the
  nilpotent part ${y}_k-a_k\in \BK$.
\item the crossing $\psi_k\in \TL$ of two strands with longitudes in
  $a_k\approx a_{k+1}$ keeping the longitudes unchanged is simply given by
  $\psi_k\in \BK$
\item The diagram
  where the rightmost strand with  $a_k\approx 2q$ for $q$ fixed increases longitude by $2$, crossing the minimal number
  of strands, is sent to
  wrapping the $k$th strand one full rightward twist around the
  cylinder crossing each other black strand once, where each crossing with a black strand is replaced with
  the neutral crossing $\chi_k$.
\item The diagram
  where the leftmost strand  with $a_k\approx 2q$ for $q$ fixed (whose label we
  denote $i$) decreases longitude by $2$, crossing the minimal number
  of strands, is sent to the diagram
  wrapping the $k$th strand one full leftward twist around the
  cylinder, where each crossing with a black strand is replaced with
  the neutral crossing $\chi_k$, multiplied on the left by  \[\prod_{\substack{r\in R_i \\ r\neq a_k}} (y_{k}-r).\]
\end{itemize}
That these satisfy the relations of $\TL$ is
a straightforward calculation with the polynomial representation.
Again by direct construction, this functor sends weight modules to
modules with finite dimensional images of $e(\Bi,\Ba)$ where the dots
are nilpotent.  Thus,
this functor defines an inverse to $\tilde{\vartheta}'$, and thus shows
that it is an equivalence.
\end{proof}
}

\begin{proof}[Proof of Theorem \ref{co:tilde-main}]
The algebra $ \tmetric $ is a subalgebra of $ \TL $ and we have an idempotent
$e_{\longi}\in \TL$ defined as
$$ e_{\longi} = \sum_{(\Bi, \kappa, \Ba) \text{ longitude triple}} e(\Bi, \kappa, \Ba) $$
This satisfies
$\tmetric =
e_{\longi}\TL e_{\longi}$.

We have a pair of quotient functors
\[e_{\longi}\TL\otimes- \colon
\TL\mmod_{\operatorname{nil}}\to
\tmetric\mmod_{\operatorname{nil}} \qquad e(\Bi_\Bm)\BK\otimes-
\colon \wtmodBK \to \wtmodFY .\]

We have mutually inverse equivalences $  \Theta :  \TL\operatorname{-mod}_{\operatorname{nil}} \rightarrow \wtmodBK $ and $ \Gamma : \wtmodBK \rightarrow\TL\operatorname{-mod}_{\operatorname{nil}}  $.  Thus it suffices to show that
$$
\text{ for $ M \in\TL\operatorname{-mod}_{\operatorname{nil}}  $, if  $ e_\longi M = 0 $, then $ e(\Bi_\Bm)\Theta(M) = 0 $ }
$$
and
$$
\text{ for $ N \in \wtmodBK $, if $ e(\Bi_\Bm) N = 0 $, then $ e_\longi \Gamma(N) = 0 $}
$$

Suppose that $ M \in\TL\operatorname{-mod}_{\operatorname{nil}}  $ and  $ e_\longi M = 0 $.  Then
$$ e(\Bi_\Bm) \Theta(M) = \bigoplus_{\Ba \text{ s.t. } (\Bi_\Bm, \Ba) \text{ integral}} d(\Bi_\Bm, \Ba)M \delta(\Bi_\Bm, \Ba)
$$
so it suffices to show that for all integral $ (\Bi_\Bm, \Ba) $ we have that $ d(\Bi_\Bm, \Ba) M = 0 $.  Now, $ \Bi_\Bm $ has the property that if $j$ is even and $i$ is odd, then all appearances of $j$ come
before any appearance of $i$ in $\Bi_\Bm$.  Assume that $p<q$ and
$ (\Bi_\Bm)_p = j $ and $ (\Bi_\Bm)_q = i$. Since we cannot have
$a_q$ even and $a_p$ odd,
if $a_q\succeq a_p$ then we must have $a_q\geq a_p$.  Thus, in the
idempotent $d(\Bi_\Bm,\Ba)$, the longitudes are weakly increasing. Thus $ d(\Bi_\Bm, \Ba) $ is a summand of $ e_\longi$ and so $ d(\Bi_\Bm, \Ba) M = 0 $ as desired.

On the other hand, suppose that $ N \in \wtmodBK $ and $ e(\Bi_\Bm) N = 0 $.  We need to show that for any $ \Ba $ weakly $\preceq$-increasing, we have $ e(\Bi, \kappa, \Ba) \Gamma(N) = 0 $.  Fix $ \Ba$ weakly increasing. Since $ e(\Bi, \kappa, \Ba) \Gamma(N) = W_{\Bi, \Ba}(N)$, we need to show that these weight spaces vanish.  By Lemma \ref{le:changetoBiBm} and Corollary \ref{co:AdPermIso}, there exists $ \Ba' $ such that $ W_{\Bi, \Ba}(N) \cong W_{\Bi_\Bm, \Ba'}(N) $ and thus we conclude that $ W_{\Bi, \Ba}(N) = 0 $ as desired.

Finally, note that the quotient functor $ \TL\mmod_{\operatorname{nil}}\to
\tmetric\mmod_{\operatorname{nil}} $ is split by the inclusion functor $ \tmetric\mmod_{\operatorname{nil}} \to \TL\mmod_{\operatorname{nil}} $ which is given by taking a $\tmetric $-module $ M $ to $ \TL e_\ell \otimes_\tmetric M $.  Thus, the equivalence $ \Theta $ takes a module $ M \in \tmetric\mmod_{\operatorname{nil}} $ to
$$
e(\Bi_\Bm)\Theta(\TL e_\ell \otimes_\tmetric M) = \bigoplus_{\Ba \text{ s.t. } (\Bi_\Bm, \Ba) \text{ integral}} d(\Bi_\Bm, \Ba)M \delta(\Bi_\Bm, \Ba)
$$
and in particular, we see that $ W_\Ba(\Theta(M)) = d(\Ba) M $.
\end{proof}
\excise{
Similarly, if we start with an honest longitude, then the associated sequence
$\Bi$ doesn't have to be equal to $\Bm$, but it can be brought to this
form while only crossing strands with $a_{k}\not\approx a_{k+1}$
for which we can use the neutral crossing
and no crossing of $x=0$, since amongst strands with $a_k\approx 2q$, all the strands with odd labels are right of those
with even labels.  This shows that the image of the original
idempotent matches a weight space for $\Bm$.

Since every coarse longitude is a longitude, we have an idempotent
$e_{\longi}\in \TL$ such that
$\tmetric^{\bR}\cong
e_{\longi}\TL e_{\longi}$.

Thus, we have a pair of quotient functors
\[e_{\longi}\TL\otimes- \colon
\TL\mmod_{\operatorname{nil}}\to
\tilde{P}^{\bR}\mmod_{\operatorname{nil}} \qquad e(\Bm)\BK\otimes-
\colon \mathcal{W}^{\bR}(\BK)\to \mathcal{W}^{\bR} .\]
In order to show that $\tilde{\vartheta}'$ induces an equivalence
$\tilde{\vartheta}$ as desired, we need to show that it intertwines
the kernels of these functors.  In this case we can conclude that it
induces an equivalence between the images of the quotients under the
left adjoints to projection, given by
$\TL e_{\longi}\otimes-$ and
$\BK e(\Bm)\otimes -$ respectively.

This follows immediately from the
discussion above.  A $\TL$-module is killed
by $e_{\longi}$ if and only if it is killed by all idempotents
isomorphic to one with an honest longitude.  This shows that it is
killed by $d(\Bm,\Ba)$ for all $\Ba$, and thus its image under
$(\tilde{\vartheta}')^{-1}$ is killed by $e(\Bm)$.  Similarly, if a
$\BK$-module is killed by $e(\Bm)$, then it is killed by all
idempotents of the form $d(\Bm,\Ba)$, which up to irrelevant
crossings includes all that carry an honest longitude, so its image
under $ \tilde{\vartheta}'$ is killed by $e_{\longi}$.

 This shows that we have an equivalence $\tilde{\vartheta}$, and the
 match \eqref{eq:w-e} follows from the fact that the isomorphism of
 $d(\Bm,\Ba)$ and $e(\Ba)$.
\end{proof}
}
\excise{
Now, we claim that there is an action of $\tmetric^{\bR}$ on this
functor.  First consider the diagrams where the longitudes on the top
and bottom differ by the permutation corresponding to the strands:
  \begin{itemize}
  \item As required by the statement of the theorem, $e(\longi)$
    acts by projection to the summand $W_{\Ba(\longi)}(M)$.
  \item
    We let the dot on the $k$th black strand with label $i$ act by the
    nilpotent part $\hat{z}_{i,k}$ of the
    action of $z_{i,k}$; thus it acts on $W_\Ba(M)$ by
    $z_{i,k}-a_{i,k}$, which is nilpotent by the definition of
    a weight module.
  \item Crossing two strands with the same label and longitude which
    are the $k$th and $(k+1)$-st from the left with label $i$ acts by
    \[u_{i,k}=\frac{(k,k+1)_i-1}{z_{i,k}-z_{i,k+1}}.\subeqn\label{eq:dem-same}\]
  \item Crossing two strands with the same label and longitudes which
    differ by a non-zero imaginary number which
    are the $k$th and $(k+1)$-st from the left with label $i$ acts by
    \[(z_{i,k}-z_{i,k+1} )u_{i,k}+1=(k,k+1)_i. \subeqn\label{eq:dem-different}\]
Note that in this case, $z_{i,k}-z_{i,k+1}=\hat{z}_{i,k}-\hat{z}_{i,k+1}-a_{i,k}-a_{i,k+1}$ is invertible, and thus
$u_{i,k}$ itself can be recovered from the action of this element.
 \item Crossing two strands with the same label and longitudes that
   differ by an imaginary number acts by
   the identity on the underlying weight space.
\end{itemize}
Now, we consider the diagrams that change longitude.  Such a diagram
can be factored as a product of diagrams like those discussed above
where the longitude doesn't change, and ones where a single strand changes
longitude from $a$ to $a\pm 2$, and crosses all the strands with real
parts strictly between those of $a$ and $a\pm 2$.
\begin{itemize}
 \item We send a leftward movement of the
    $k$th strand corresponding to $i\in I$ from longitude $2
    a_{i,k}$ to  $2a_{i,k}-2$  to
    \begin{multline*}\subeqn
       (-1)^{m_j}\beta_{i,k}\frac{ R_i(z_{i,k})\prod_{i\leftarrow
    j}Z_{j}(z_{i,k}-\nicefrac 12) }{\displaystyle\prod_{r\neq a_{i,k}\in R_I}(z_{i,k}-r)
    \prod_{a_{j,p}\neq a_{i,k}-\nicefrac{1}{2}}(z_{j,p}
    -z_{i,k}+\nicefrac 12 )}\\=\beta_{i,k}(z_{i,k}-r)^{\# r= a_{i,k}\in R_I}\prod_{a_{j,p}=a_{i,k}-\nicefrac{1}{2}}(z_{j,p}
      -z_{i,k}+\nicefrac 12 )\\
=\beta_{i,k}(\hat z_{i,k})^{\#\{r\in R_i\mid  r= a_{i,k}\}}\prod_j(\hat z_{j,p}
      -\hat z_{i,k})^{\#\{p\mid a_{j,p}=a_{i,k}-\nicefrac{1}{2}\}} .\label{eq:left-move}
    \end{multline*}
  This is well-defined, since
$\beta_{i,k}R_i(z_{i,k})\prod_{j\leftarrow i}Z_{j}(z_{i,k}-\nicefrac 12)$ is in
the GKLO image of $\FY^\la_\mu$ and
$\prod_{r\neq a_{i,k}\in R_I}(z_{i,k}-r)$ and $\prod_{a_{j,p}\neq a_{i,k}-\nicefrac{1}{2}}(z_{j,p}
-z_{i,k}+\nicefrac 12 )$
both act invertibly on $W_{a}(M)$.
\item Similarly, we send a rightward
movement to $2a_{i,k}+2$
\begin{multline*}\subeqn
  \beta_{i,k}^{-1}\prod_{i\to j}Z_{j}(z_{i,k}+\nicefrac 12)
  \frac{1}{\displaystyle \prod_{a_{j,p}\neq
      a_{i,k}+\nicefrac{1}{2}}(z_{i,k}-z_{j,p}+\nicefrac 12
    )}\\=\beta_{i,k}^{-1}\prod_{a_{j,p}=a_{i,k}+\nicefrac{1}{2}}(z_{i,k}-z_{j,p}+\nicefrac
  12
  )\\=\beta_{i,k}^{-1}\prod_j(\hat
  z_{i,k}-\hat z_{j,p} )^{\#\{p\mid
    a_{j,p}=a_{i,k}+\nicefrac{1}{2}\}}. \label{eq:right-move}
\end{multline*}
\end{itemize}
We can check the relations of the KLR algebra by noting that this acts
in a completion of the GKLO representation compatibly with a
completion of the polynomial representation of the KLR algebra, where
we identify these representations by matching the variable $Y_p$ in
the representation of the KLR algebra with $\hat{z}_{i,k}$ in the GKLO
representation if the $p$th strand from the left in $\ell$ has label
$i$ and is the $k$th such strand.

This functor is an equivalence
since the equations above, especially
(\ref{eq:dem-same}--\ref{eq:right-move}), can be inverted to give an
action on $\oplus_{\Ba\in \mathscr{I}_\bR}e(\Ba)N$ for $N$ a
module over $\tmetric^{\bR}$.  Obviously, the dots will act
nilpotently on this module; this gives the functor to
$\tilde{\vartheta}$.
}

These functors have a natural analogue for category $\cO$ and for the finite-dimensional modules.  Let
\[x(\Bnu)=\sum_{i\in I}\sum_{k=1}^{m_i}\nu_{i,k}.\]

\begin{Definition}
\label{def:catO}
Let $\fdFY$ be the category of finite dimensional $ \FY^\la_\mu$ modules with integral weights.  Let $\OFY^-$ be the full subcategory of $ \wtmodFY $ consisting of those objects $ M $, for which there exists $N\in \R$, such that $W_{\Bnu}(M) = 0$ whenever $x(\Bnu)<N$.  Similarly we define $ \OFY^+ $ using the condition $x(\Bnu)>N$.

  Similarly, let $\OY^-$ be the full subcategory of $ \wtmodY $ consisting of those objects $ M $, for which there exists $N\in \R$, such that $W_{\bS}(M)
  = 0$ whenever $x(\bS)<N$.    Similarly we define $ \OY^+ $ using the condition $x(\bS)>N$.
  \end{Definition}

\begin{Theorem} \label{co:main}
The equivalence $  \Theta \colon \tmetric^{\bR}_\mu\operatorname{-mod}_{\operatorname{nil}} \to \wtmodFY $ restricts to equivalences
$$
{}_\pm \metric_\mu\mmod \cong \OFY^\pm, \ {}_0 \metric_\mu\mmod \cong \fdFY
$$
\end{Theorem}

\begin{proof}
We will prove the statement about $ \OFY^- $.  The statement about $\OFY^+ $ is proved similarly and then the one for $\fdFY $ is obtained by intersecting.

Suppose that $ M \in {}_-\metric_\mu\mmod$.  Then for any $ \Bnu $, we have $ W_\Bnu(\Theta(M)) = e(\Bnu)(M) $.  Choose some $ \Bnu $ such that $ d(\Bnu)(M) \ne 0 $.  Now, $ d(\Bnu) = (\Bi, \kappa, \Ba) $ where $ \Ba $ is obtained by putting the elements of $ \Bnu $ in order.  Since $ M \in {}_-\metric_\mu\mmod $, for all $ 1 \le k \le m $, we have $ k < \kappa(\ell)$ (as the rightmost strand must be red).  Thus for all $k$, we have $ a_k < r_\ell$ and thus we see that $ x(\Bnu) = \sum a_k < m r_\ell $.  So we can choose $ N = \max(mr : r \in \bR) $ and we are done.

Suppose that $ \Theta(M) \in   \OFY^- $.  Suppose that $ M $ does not lie in $ {}_-\metric_\mu\mmod$.  Then for some triple $ (\Bi, \kappa, \Ba) $ not having a rightmost red strand, we have $ e(\Bi, \kappa, \Ba) M \ne 0 $.  Since the rightmost strand in $ e(\Bi, \kappa, \Ba) $ is black, we can increase the label on this black strand and still preserve the longitude condition.  The straight line diagram which increases a longitude is invertible in $ \metric$.  Thus we conclude that for all $ p \in \N $, we have $ e(\Bi, \kappa, \Ba + 2p\varepsilon_m) M \ne 0 $ (where $ \varepsilon_m = (0, \dots, 0, 1)$).  Now as in the proof of Theorem \ref{co:tilde-main}, there exists $ \Bnu$ such that $ W_\Bnu(\Theta(M))= e(\Bi, \kappa, \Ba + 2p\varepsilon_m) M \ne 0 $.  Then $ x(\Bnu) = \sum a_k + 2p $.  Since $ \Theta(M) \in   \OFY^- $, this is a contradiction.  Thus, we see that $M \in {}_-\metric_\mu\mmod$.

\end{proof}

\excise{
\begin{Theorem} \label{co:main}
There are equivalences ${}_{\pm}\vartheta\colon
{}_{\pm}{P}^{\bR}\operatorname{-mod}\to \O^{\pm}$  and
${}_{0}\vartheta\colon {}_{0}{P}^{\bR}\to \mathcal{F}$ such
that \[w_J(M)\cong e'(J)M \qquad   W_\upsilon(M)\cong e(\upsilon)M.\]
\end{Theorem}
\begin{proof}
Since the argument is symmetric, we only consider $\O^-$
If $e(\upsilon)$ has a violating strand on the right, then it is
isomorphic to an idempotent where the violating strands move further
and further to the right.  This corresponds to increasing
$x(\upsilon)$, so if $ W_{\upsilon}(M)\neq 0$, then there is no upper
bound $x(\upsilon')$ such that $ W_{\upsilon'} (M)\neq 0$, and $M$ is not in
$\O^{-}$.  This shows that the action of $\tilde{P}^{\bR}$ on the
module $\tilde\vartheta(M)$ factors through ${}_-P^{\bR}$.

On the other hand, if we begin with a $\tilde{P}^\bR$-module that
factors through ${}_{-}P^\bR$, then there is an upper bound on
the longitudes that appear, given by the longitude of the rightmost
red strands, and thus an upper bound on the sum of these longitudes.
This shows that  have an upper bound on $x(\upsilon')$ such that $
W_{\upsilon'} (\tilde{\vartheta}^{-1}M)\neq 0$, and so
$\tilde{\vartheta}^{-1}M$ is in $\O^-$.  This gives the desired equivalence.
\end{proof}
}

As noted before, this equivalence goes between a category with a well-known categorical action of the Lie algebra $\fg$ and one where no such action is apparent.  We'll correct this issue by constructing the counterpart action on category $\cO$ in a future paper \cite{KTWWY2}.

Finally, we obtain that this equivalence matches highest weights of simples in $\O$ to the product monomial crystal, as conjectured in \cite{KTWWY}:
\begin{Corollary}
\label{cor:moncrystalconj}
  The map $\varphi$ induces a bijection between the highest weights of simple modules in $\OY^-$ and $\B(\bR)_\mu$, the $\mu$ weight set of the product monomial crystal $\B(\bR)$.  This bijection maps the simple $\Yml(\bR)$-module of highest weight $\bS$ to $\yMon_\bR \zMon_\bS^{-1} \in \B(\bR)$.
\end{Corollary}

\begin{proof}
By Corollary \ref{cor:metric-bijection} we have a bijection $\varphi: \Irr({}_-\metric_\mu\mmod) \to \B(\bR)_\mu$.  By Theorem~\ref{co:main}, $\Irr({}_-\metric_\mu\mmod)=\Irr(\OFY^-)$.  Now the  equivalence $M\mapsto M^\Sigma$ between $\FYml \mmod$ and $\Yml\mmod$, induces an equivalence between the respective category $\cO$'s.  Hence we have that $\Irr(\OFY^-)=\Irr(\OY^-)$, and $\varphi$ induces the desired bijection.

Now let $L(\bS)$ be the simple $\Yml$-module of highest weight $\bS$.  Recall this means that $W_\bS(L(\bS))\neq 0$, and if $W_\bT(L(\bS))\neq 0$ then $x(\bT)< x(\bS)$.  Let $L=\FYml e_\Bm'\otimes _{\Yml} L(\bS)$ be the simple $\FYml$-module corresponding to $L(\bS)$ under the Morita equivalence.  By Theorem \ref{co:main} there is a simple ${}_-\metric_\mu$-module $M$ such that $\Theta(M)=L$, and hence  $\Theta(M)^\Sigma=L(\bS)$.  By Theorem \ref{co:tilde-main}  we have that $d'(\bS)M\neq 0$.  This implies that $d(\bS)M\neq 0$, and $\bS$ is the highest monomial that doesn't kill $M$.  Hence, by Corollary \ref{cor:metric-bijection}, $\varphi([M])=\yMon_\bR \zMon_\bS^{-1}$, which implies under the identification $\Irr({}_-\metric_\mu\mmod)=\Irr(\OY^-)$ that $\varphi([L(\bS)])= \yMon_\bR \zMon_\bS^{-1}$.

\end{proof}

%
%
%
%
\subsection{The non-integral case}
\label{sec:nonintegral}

In the preceding sections, we have only considered integral choices of $\bR$ and integral weights of
shifted Yangians.  This was mostly for simplicity of presentation
rather than due to any mathematical difficulties.  In this section we briefly describe how to generalize our results to the non-integral case.

First, we relax our assumptions to let $\bR$ contain arbitrary elements of $\C$; we could generalize even further by considering any field of characteristic 0 without changing any of the discussion below.  Elements of $r\in R_i$ and $r'\in R_j$ ``interact'' in an interesting way if $r-r'$ is an integer whose parity is even (odd) if $i$ and $j$ have the same (different) parity.  Let $\bar{i}$ be the parity of $i\in I$, i.e. the element of $\Z/2\Z$ such that $i\in I_{\bar{i}}$.  This motivates the following definition:
\begin{Definition}
\label{def:eqrelation}
We write $ (r,i) \sim (r', j) $ if $r-r'\equiv \bar i-\bar j\pmod {2\Z}$.  This defines an equivalence relation  on $\C \times I $.
\end{Definition}
The equivalence classes for this relation map injectively to $ \C / 2 \Z $ by sending $(r,i)\mapsto \bar r+ \bar i$. Note that $\bR$ is integral if and only if its image in $\C/2\Z$ is $\{0\}$.

We can naturally decompose  $Y^\la_\mu$ weight modules into ``integrality classes.''  For each $i\in I$, fix a size $ m_i $, multi-subset $\mathscr{S}_i$ of $\C/2\Z$.  This collection $\mathscr S $ defines a point in $   \prod_{i} (\C/2\Z)^{m_i}/\Sigma$.  Equivalently, it defines an orbit of the semi-direct product $\Sigma \ltimes \prod (2\Z)^{m_i}$ acting on $ \prod_i \C^{m_i}$, which we can think of as the extended affine Weyl group of $\prod_i GL_{m_i}$.

\begin{Definition}\label{def:integrality}
A weight module $M $ for $ Y^\la_\mu $ has \textbf{integrality class} $ \mathscr S $ if whenever $ W_\bS(M)\ne 0 $, then $ \overline{\bS} = \mathscr S $ (reduction modulo $2\Z$).  Similarly a weight module $ M$ for $ FY^\la_\mu $ has \textbf{integrality class} $ \mathscr S $ if whenever $ W_\Bnu(M) \ne 0 $, then $ \Bnu \in \mathscr S $ (regarded as a orbit).

 We write $\wtmodY_{\mathscr S}$ (respectively $\OY^\pm_{\mathscr S}, {{Y^\la_\mu\operatorname{-mod}_{\operatorname{fd},\mathscr S}}}$) for the categories of weight (respectively $\cO^\pm$, finite-dimensional) modules over $Y^\la_\mu$ having {\bf integrality class} $\mathscr S$.
 \end{Definition}
The Morita equivalence between $ Y^\la_\mu $ and $ FY^\la_\mu$ restricts to an equivalence between $ \wtmodY_{\mathscr S}$ and $ \wtmodFY_{\mathscr S}$.

By studying the adjoint action on the algebra $Y^\la_\mu$ itself, it is easy to deduce the following result.
\begin{Lemma}
  Every indecomposable weight module for $Y^\la_\mu$ has an integrality class.
\end{Lemma}
We can extend the notion of integrality class $ \mathscr S $ to $\BK$ by considering weight modules where the weight space for a pair $(\Bi,\Ba)$ is only non-zero if the multiset of residues mod $2\Z$ of the complex numbers $a_k$, such that $i_k=i$, equals $\mathscr{S}_i$.
Let $X$ be the union in $\C/2\Z$ of images of $R_i+\bar i$ and of $ \mathscr S_i+\bar i$ for $i\in I$.
If $\bR$ is integral, then the integral weight modules we have studied previously are weight modules with integrality class given by  $\mathscr S_i=\{\bar{i},\dots,\bar{i}\}$. In this case $ X = \{0 \} $.

Consider the Dynkin diagram $I\times X$ and the associated Lie algebra $\fg_{X}$, which is just $|X|$ copies of $\fg$.  We define weights $\la_{\bR}=\sum_i\sum_{r\in R_i} \omega_{i,\bar{r}+\bar i} $ and $\mu_{\mathscr S}=\la_{\bR}-\sum_i\sum_{s\in \mathsf{S}_i}\al_{i,s+\bar i}$.

The integrality class $\mathscr S$ is called {\bf $\bR$-integral}  if $ X$ coincides with the set of equivalence classes for the equivalence relation on $ \bR $.

\begin{Example}
  Let $I=\{x,y\}$, connected by a single edge $x\to y$ with $\bar{x}=0$ and $\bar{y}=1$; as before $\fg=\mathfrak{sl}_3$.  Consider $R_x=\{0,\nicefrac 12, 2+3i\}$ and   $R_y=\{0,-\nicefrac 12\}$.  Note that the only pair of the elements of $\bR$ which are equivalent is $(x, \nicefrac 12)\sim (y, -\nicefrac 12)$, so there are 4 equivalence classes.

Thus, an integrality class is $\bR$-integral if the elements of $\mathscr{S}_x$ only contain elements of $\{\bar{0},\bar{\nicefrac 12}, \bar{1},\bar{i}\}\subset \C/2\Z$, and $\mathscr{S}_y$ only contains elements of $ \{\bar{0},-\overline{\nicefrac 12}, \bar{1},\overline{1+i}\}$.

For example, if $\mathscr{S}_x=\{\bar{1}, \overline{\nicefrac 12}\}$ and $\mathscr{S}_y=\{\bar{0},- \bar{\nicefrac 12}\}$, then this is an $\bR$-integral class and $X= \{\bar{0},\bar{\nicefrac 12}, \bar{1},\bar{i}\}$.  We have that
\begin{align*}
\la_{\bR}=\omega_{x, \bar{0}}+\omega_{x,\overline{\nicefrac 12}}+\omega_{x, \bar{i}}+\omega_{y, \bar{1}}+\omega_{y, \overline{\nicefrac 12}}\\
\mu_{\mathscr S}=\la_{\bR}-\al_{x, \bar{1}}-\al_{y, \bar{1}}-\al_{x, \overline{\nicefrac 12}}-\al_{y, \overline{\nicefrac 12}}
\end{align*}
\end{Example}
Attached to $\bR$ and $\mathscr S $, we can define metric KLRW algebras $\tilde{\mathscr{T}}^{\bR}_{\mu_{\mathscr{S}}}, {}_{\pm}{\mathscr{T}}^{\bR}_{\mu_{\mathscr{S}}},{}_0 \mathscr{T}^{\bR}_{\mu_{\mathscr{S}}}$ associated to the Dynkin diagram $I\times X$, where the longitudes on strands with label $(i,s)$ for $s\in \C/2\Z$ are required to live in the coset corresponding to $s$ in $\C$; in Definition \ref{def:long}, we change the inequalities in (2) and (3) to equalities between real parts of the corresponding complex numbers.  Note that these algebras $ \tilde{\mathscr{T}}^{\bR}_{\mu_{\mathscr{S}}}, \dots $ depend on the integrality class $ \mathscr S $, since $ X $ depends on $ \mathscr S $.

We define parity distance (as in Section \ref{sec:parityKLRW}) between strands labeled with nodes $(i,x)$ and $(j,x)$ corresponding to the same element of $X$; we only count changes in parity between strands in the same equivalence class as those we are comparing.    Given this, we can easily confirm the generalization of Lemma \ref{lem:compat-long} and define the parity algebra  $\tilde{P}_{\mu_{\mathscr{S}}}^{\bR}$ (respectively ${}_{\pm}{P}_{\mu_{\mathscr{S}}}^{\bR},{}_0P_{\mu_{\mathscr{S}}}^{\bR}$) Morita equivalent to $\tilde{\mathscr{T}}_{\mu_{\mathscr{S}}}^{\bR}$ (respectively, ${}_{\pm}{\mathscr{T}}_{\mu_{\mathscr{S}}}^{\bR},{}_0 \mathscr{T}_{\mu_{\mathscr{S}}}^{\bR}$).

If $ \mathscr S  $ is not $ \bR $-integral, then that means that some black strand has a label that is not in the same component of $I\times X$ with any red strands.  By (\ref{black-bigon}, \ref{cost}), every idempotent in $\tilde{P}_{\mu_{\mathscr{S}}}^{\bR}$ is isomorphic to one where we pull all strands labeled with this element of $X$ to the far left or far right.  Thus, we have that
 $ {}_{\pm} {P}^{\bR}_{\mu_{\mathscr{S}}} = 0 $.

We can now state the generalization of Theorem \ref{th:intro2} to the non-integral case.  The proof of this result follows in the same way as Theorem \ref{th:intro2}, using the above metric (and coarse metric) KLRW algebras and using $\BK$-modules having integrality class $\mathscr S $.
\begin{Theorem}\label{th:non-integral}
Let $\bR,\mathscr S$ be as above.
The categories $\wtmodY_{\mathscr S} , \OY^\pm_{\mathscr S}, {{Y^\la_{\mu_{\mathscr{S}}}\operatorname{-mod}_{\operatorname{fd},\mathscr S}}}$ are equivalent to the categories $\tilde{P}^{\bR}_{\mu_{\mathscr{S}}}\mmod_{\operatorname{nil}}, {}_{\pm}{P}^{\bR}_{\mu_{\mathscr{S}}}\mmod,{}_0P^{\bR}_{\mu_{\mathscr{S}}}\mmod$ of modules over parity KLRW algebras associated to the Dynkin diagram $
I\times X$. In particular,  if $\mathscr S$ is not $\bR$-integral,
 then $\OY^\pm_{\mathscr S}$ is trivial.
 \end{Theorem}

We can prove this last assertion in a more lower tech way by noting which pairs of weights are related by admissible permutations.  Note that if $a_k$ and $a_{k+1}$ are not equivalent under the relation discussed above, then the permutation $s_k$ is always admissible.  Similarly, if $a_m$ is not equivalent to any element of $\bR$, then the rightward seam crossing  $\sigma_+ : W_{\Bi, \Ba} \to W_{\sigma(\Bi), \sigma(\Ba)} $ is an isomorphism (since $\sigma_-\sigma_+$ is multiplication by an invertible polynomial in the dots). Combining these, let $Q$ be a subset of $I\times \C/\Z$ which is closed under $\sim$, and contains no elements of $\bR$. If we let $\Ba^{(N)}$ be defined by \[a^{(N)}_i=
  \begin{cases}
    a_k & (i_k,a_k)\notin Q\\
    a_k+2N & (i_k,a_k) \in Q
  \end{cases}
\] then the isomorphisms above show $W_{\Bi, \Ba} \cong W_{\Bi, \Ba^{(N)}}$ for all $N\in Z$, but lying in $\mathcal{O}^\pm$ requires the latter weight space must be 0 for $N\gg 0$ (resp. $N\ll 0$).





\section{\GT modules}
\label{sec:gelf-zetl-modul}


Theorem \ref{co:tilde-main} allows us to address a number of more classical questions in the representation theory of $U(\mathfrak{gl}_N)$ and its $W$-algebras.  

W-algebras arise in our setting as truncated shifted Yangians for
$\mathfrak{g}=\mathfrak{sl}_n$, and the weight $\la=N\omega_1$.
Here, we follow the conventions of \cite{QMV}, with the exception of
the fact that we have specialized $\hbar=2$ in the Yangian rather than
$\hbar=1$.  We will note when this convention causes issues.

Thus, we identify the Dynkin diagram of $\fg$ with $I = \{1,\ldots,
n-1\}$.  We fix $\la=N\varpi_1$, and consider a weight $\mu$ with
\[ \mu = \sum_{i=1}^{n-1} \mu_i \varpi_{n-i}, \qquad \lambda - \mu =
  \sum_{i=1}^{n-1} m_i \alpha_{n-i}. \]
(Notice that in order to match the conventions of \cite{QMV} we have to index these coefficients in a slightly strange way.) 
Consider the partition $\pi \vdash N$, 
\begin{equation}
\label{eq: coweight data 4}
\pi=(p_1\leq \cdots \leq p_n),
\end{equation}
defined by 
\begin{equation}
\label{eq: coweight data 5}
p_1=m_1, p_2=m_2 -m_1,...,p_{n-1}=m_{n-1}-m_{n-2}, p_n=N-m_{n-1}.
\footnote{We have fixed the partition $\tau$ defined in \cite[\S 1.2]{QMV} to be
$(1^N)$.}
\end{equation}

Let $W(\pi)$ be the finite $W$-algebra quantizing the Slodowy slice to
the nilpotent orbit in $\mathfrak{gl}_N$ with Jordan type given by
$\pi$.   In particular, $U(\mathfrak{gl}_N)$ itself corresponds to $\mu=0$ and $\pi=(1,\dots, 1)$.
In general, if $\mu$ is arbitrary, then the result is the OGZ algebras
introduced by Mazorchuk in \cite{mazorchukOGZ}.

In \cite{QMV} we give a modification of Brundan and Kleshchev
\cite{BK} which matches the truncated shifted Yangians in type A in our
conventions to parabolic $W$-algebras.  
Applied to this case we get an isomorphism of algebras:
\begin{equation}
\label{QMViso}
Y^{N\omega_1}_\mu \cong W(\pi).
\end{equation}
This isomorphism is defined explicitly in \cite[Section 4.2]{QMV}.
It should be emphasized that on the LHS of this equation, the
coefficients of the polynomial $p_{n-1}(u)$ which appears in the definition of the $A_i^{(r)}$ (Equation \ref{eq: H and A})  are here formal parameters $R^{(1)}_{n-1}, \dots, R^{(N)}_{n-1}$.  
Under (\ref{QMViso}) these are sent to generators of the center of $W(\pi)$
(which is isomorphic to that of $U(\mathfrak{gl}_N)$).  To avoid confusion, in this section we will not suppress our choice of a set of parameters, so if we specialise $p_{n-1}(u)$ to have complex coefficients with roots given by a multiset $\bR$ of weight $N\omega_{n-1}$, then we will write $Y_\mu^\la(\bR)$ for the corresponding truncated shifted Yangian.

We also note
that our isomorphism $\C[R^{(1)}_{n-1}, \cdots, R^{(N)}_{n-1}]\cong
Z(\mathfrak{gl}_{N})$ is unchanged from \cite{QMV}; our changes of
conventions have had cancelling effects.   This isomorphism is fixed
by the fact that if we specialize $ p_{n-1}(u)$ to have roots $r_1<\cdots< r_N$,
then the corresponding ideal in $Z(\mathfrak{gl}_{N})$ kills the
modules with  highest weight 
\begin{equation}
\label{eq:hw}
\frac{1}{2}(r_1-N+1,r_2-N+3,\dots,r_N+N-1)
\end{equation}


%

We have a chain of algebras
$W_1\subset W_2\subset \cdots \subset W_n= W(\pi)$ where
$W_k=W((p_1,\dots, p_k))$
\begin{Definition}
  
The \GT
  subalgebra $\Gamma\subset W(\pi)$ is the subalgebra generated by the centers of all $W_k$
inside $W(\pi)$ \cite{FMO-GZ}.
\end{Definition}

Via the isomorphism (\ref{QMViso}), the inclusion of $W_k\subset W_n$
induces a map $Y^{m_k\omega_1}_{\mu_k}\to Y^{N\omega_1}_{\mu}$ sending
$p_{k-1}(u+1)\mapsto A_k(u)$.  Here $Y^{m_k\omega_1}_{\mu_k}$ is a truncated shifted Yangian for $\mathfrak{sl}_k$ , and $\mu_k=m_k\omega_1-\sum_{i=1}^{k-1}m_i\alpha_{k-i}$.

%
Thus, the center of $W_k$ is sent to
the subalgebra of $Y_\mu^{N\omega_1}$ generated by the $A_k^{(r)}$ for
all $r$, and $\Gamma$ is sent to the subalgebra generated by $A_k^{(r)},R_{n-1}^{(r)}$ for
all $k,r$, which we call the \GT subalgebra of $Y_\mu^{N\omega_1}$.
This matches the algebra of the same name defined in \cite{BK}.

\begin{Definition}
  We call a finitely generated $W(\pi)$ module a {\bf \GT module} (or $\GTc$ module for short) if $\Gamma$ acts locally finitely. 
\end{Definition}

As we'll now explain, under the isomorphism (\ref{QMViso}), these correspond to weight modules of truncated shifted Yangians, as in Definition \ref{Def5.1}.  
Theorem \ref{co:tilde-main} will therefore allow us to give a description
of the category of $\GTc$ modules for the $W$-algebra $W(\pi)$.

First note that fixing a maximal ideal of the center of $W(\pi)$
corresponds to a set of parameters $\bR$ of weight $N\omega_{n-1}$.  Any simple
$\GTc$ module factors through such a specialization.


Having fixed $\bR$ we can then further decompose the category of $\GTc$ modules according
to a choice of integrality class $\mathscr{S}$ corresponding to $\pi$.  Here
$\mathscr{S}=(\mathscr{S}_i)_{i\in I}$, and $\mathscr{S}_i$ is a
multi-subset of $\mathbb{C}/2\mathbb{Z}$ of cardinality $m_i$.  Every
maximal ideal of $\Gamma$ has a corresponding integrality class as in 
Definition \ref{def:integrality}. Following this definition,  we
say that a $\GTc$ module $M$   has {\bf integrality class}
$\mathscr{S}$ if all the maximal ideals of $\Gamma$ with non-zero weight spaces
are congruent modulo $2\Z$ to $\mathscr{S}$ (that is, $\mathscr{S}$ is
their integrality class).  

In the
terminology of Futorny, Molev and Osienko \cite{FMO-GZ}, the irreducible
\GT modules in
this subcategory are precisely those with central character $\bR$ that
are extended from maximal ideals of $\mathfrak{m} \subset \Gamma$
with this integrality class.  Also, in the study of $\GTc$ modules there has been  some recent focus on non-singular and singular modules (see e.g. \cite{FGRnew}).  In our setting $\mathfrak{m}$ is non-singular if the elements of
$\mathscr{S}_i$ are all distinct, and singular of index $\ell$ if
$\ell$ is the maximal multiplicity of an element of $\mathscr{S}_i$
for some $i$.

\begin{Definition}
For $\bR$  a set of parameters of weight $N\omega_{n-1}$ and $\mathscr{S}$ an integrality class corresponding to $\pi$,
  we let $\GTc(\pi, \bR, \mathscr{S})$ be the category of \GT modules of integrality class $\mathscr{S}$ and central character corresponding to $\bR$.  
\end{Definition}

Note that (\ref{QMViso}) induces an equivalence $\GTc(\pi, \bR, \mathscr{S}) \cong Y^{N\omega_1}_{\mu}(\bR)\operatorname{-wtmod}_{\mathscr S}$, where $\mu$ is determined from $\pi$ as above.
Now to relate the category $\GTc(\pi, \bR, \mathscr{S})$ to KLRW algebras, recall that attached to the choices of $\bR$ and $\mathscr{S}$ we have the set $X \subset \mathbb{C}/2\mathbb{Z}$ as in Section \ref{sec:nonintegral}.  Consider the Lie algebra $\mathfrak{sl}_n^{X}$ given by maps from $X$ to $\mathfrak{sl}_n$.  Let $r_1,r_2, \dots , r_h$ be the distinct complex numbers that appear in  $R_{n-1}$, and let $g_1,\dots , g_h$ be the multiplicity with which they appear.  Let $\la_k=g_k\omega_{1,\bar{r}_k+1}$ be the highest weight of $\Sym^{g_k}(\C^k)$, considered as a module over $\mathfrak{sl}_n^X$ with the copy of $\mathfrak{sl}_n$ corresponding to $\bar{r}_k+1\in \C/2\Z$ acting by the natural representation, and all others acting trivially.

Let $\bla=(\la_1,\dots, \la_h)$, and consider the tensor product 
\[U(\bla')=U(\mathfrak{n}_-^X)\otimes V(\la_h)\otimes \cdots \otimes
  V(\la_1).\]  Recall that by Proposition \ref{prop:tildeTaction},
$U(\bla')$ is categorified  by the category of finitely generated $\tilde{T}^\bla$-modules.  
Define $\lambda_\bR=\sum_{k=1}^h \la_k$, and $\mu_{\mathscr{S}}=\lambda_\bR-\al_{\mathscr{S}}$, where
\begin{equation}\label{eq:S-root}
    \al_{\mathscr{S}}=\sum_{i=1}^k\sum_{s\in\mathscr{S}_i}\al_{i,s+\bar{i}}.
\end{equation}

\begin{Theorem}
  There is an equivalence of categories from $\GTc(\pi, \bR, \mathscr{S})$ to the category of nilpotent representations of the  KLRW algebra $\tilde{T}^\bla_{\mu_{\mathscr{S}}}$.  
\end{Theorem}

\begin{proof}
We have:
\begin{align*}
\GTc(\pi, \bR, \mathscr{S}) &\stackrel{\ref{QMViso}}{\cong}    Y^{N\omega_1}_\mu(\bR)\operatorname{-wtmod}_{\mathscr S} \\
    &\stackrel{\ref{th:non-integral}}{\cong} \tilde{P}^{\bR}_{\mu_{\mathscr{S}}}\operatorname{-mod}_{\operatorname{nil}} \\
     &\stackrel{\ref{ex:Nomega1}}{\cong} \tilde{T}^\bla_{\mu_{\mathscr{S}}}\operatorname{-mod}_{\operatorname{nil}} \qedhere
\end{align*}

\end{proof}


As promised in the introduction, we will leave most detailed
discussion of the consequences of this observation to later work, but
will note that the results of our earlier sections resolve several
important questions about \GT modules.

The first clear consequence is that $\GTc(\pi, \bR, \mathscr{S})$ is a
Deligne tensor product of the categories $\GTc(\pi_x, \bR_x,
\mathscr{S}_x)$ where we decompose our data according to the cosets of
$\C/2\Z$ (i.e. it is the representation category of the tensor product
of the endomorphism algebras of projective generators in these categories).  Thus, all questions about the structure of simples,
extensions, weight multiplicities, etc. can be derived from the
integral case.

Perhaps of more direct interest, Lemma \ref{lem:tildeP-crystal} gives a classification of simple \GT modules:
\begin{Corollary}\label{cor:GZ-bijection}
  There are canonical bijections between: \begin{enumerate}
      \item the set of simple objects of $\GTc(\pi, \bR, \mathscr{S})$;
      \item the elements of weight $\mu_{\mathscr{S}}$ in the tensor product crystal $\B(\bla)\otimes  \B(\infty)$ for $\mathfrak{sl}_n^{X}$.  
  \end{enumerate} 
\end{Corollary}


\begin{Example}
Consider the example $\pi=(1,...,1)$; in this case, we have that
$W(\pi)=U(\mathfrak{gl}_N)$.  Let us explain how our result applies in
this case
to modules extended from integral maximal ideals of $\Gamma$, since
this is the most difficult case to attack with conventional methods.
First we fix an integral central character $U(\mathfrak{gl}_N)$.
Following our conventions, this corresponds to a choice of distinct
integers $r_1,...,r_h$ such that $r_i \not\equiv (N \mod 2)$, and
positive multiplicities $g_1,...,g_h$ such that $\sum g_i =N$.  Let
$\mathfrak{m} $ be the corresponding ideal of $Z(\mathfrak{gl}_N)$;
note that there is a finite dimensional module with this central
character if and only if $g_i=1$ for all $i$.  

Then we have a canonical bijection between integral irreducible $\GTc$ modules with central character $\mathfrak{m}$ and the zero weight space of the $\mathfrak{sl}_N$-crystal
$$
\B(g_1\omega_1)\otimes\cdots\otimes\B(g_h\omega_1)\otimes \B(\infty)
$$
The reader might naturally wonder if this is related to the appearance
of tensor product crystals in the structure of category $\cO$, for
example as in \cite[Th. 4.4 \& 4.5]{BK}.  While not unrelated, this is
not the same as the crystal structure  induced by translation
functors;  rather, it is Howe dual in an appropriate sense (much as in
\cite{Webweb}).  
\end{Example}

Furthermore, the weight multiplicities of these simple modules have a natural interpretation.  Associated to each metric longitude, we have a projective $\tilde{T}e(\bS)$, and thus a class in $\tilde{p}(\bS)\in K^0(\tilde{T}^\bla)\cong U(\bla)$.
\begin{Corollary}\label{cor:GZ-weights}
  The weight multiplicity of a \GT weight corresponding to $\bS$ in a simple is the coefficient of the corresponding canonical basis vector when $\tilde{p}(\bS)$ is expanded in the canonical basis of $U(\bla)$.
\end{Corollary}
One can use this together with analysis of canonical bases to understand a number of special cases of $\GTc(\pi, \bR, \mathscr{S})$, such as nonsingular or 1-singular modules or non-critical modules, but this result also shows why concrete classification of \GT modules is so hard to do by hand: it essentially requires construction of these canonical bases, and thus an understanding of Kazhdan-Lusztig theory in type A which one cannot expect to achieve so explicitly.  
\excise{
\subsection{\GTc\: modules of $U(\mathfrak{gl}_N)$}
\label{sec:GZ-interesting-case}

\bcom{Copied from above: 
Under these identifications, the central character of a finite dimensional
module corresponds to integral parameters $x_1> \cdots >x_N$ with
$x_i-x_{i+1}\in 2\Z$.  The characters of $Z(\mathfrak{gl}_{N-1})$ that
appear in the spectrum of this finite dimensional are given by
$(y_1,\cdots, y_{N-1})$ which are integers of opposite parity from the
$x_i$ and which interlace them according to the inequalities
$x_1>y_1>x_2>y_2>\cdots >y_{N-1}>x_N$.  The attentive reader can use
this to deduce the appearance of \GT patterns in our conventions.  

In the case when $\mu=0$ we thus use this isomorphism to identify the subalgebra of $Y_0^{N\omega_1}$ generated by
$A_k^{(r)}$ for all $r$ with the center $Z(\mathfrak{gl}_k)\subset
U(\mathfrak{gl}_N)$ for $k<N$, and similarly for $R_{1}^{(s)}$ with the
center $Z(\mathfrak{gl}_N)$ via the maps:
\begin{align}
A_k(u)&=\operatorname{cdet}(\delta_{i,j}(u-\tfrac{k}{2}+i-1)-E_{i,j})_{i,j=1,...,k}
  \label{eq:A-char-corr}\\
R_{N-1}(u)&=\operatorname{cdet}(\delta_{i,j}(u-\tfrac{N}{2}+i-1)-E_{i,j})_{i,j=1,...,N}\label{eq:R-char-corr}
\end{align}
}
\ocom{ Can we stick with the conventions of this paper, and fix $R_1$?
  As it reads now, the reader would have no idea why $R_{N-1}$
  suddenly appears instead of $R_1$.}
\bcom{Yes. I had said at some point that someone else should fix them,
  since obviously I'm confused.}

The most interesting case is when $Y^\la_0\cong U(\mathfrak{gl}_N)$. In this case, we have $A_i(u)$ is the Capelli determinant of the top left $i\times i$-matrix:
\acom{It's a bit of a mess...  

Using the diagram just after equation (4.4) in \cite{QMV}, plus Lemmas 4.8 and 4.9, we get that {\bf if $\hbar =1$},
$$
A_i^{\text{QMV}}(u) = u^{-i} \operatorname{cdet} 
\begin{pmatrix} 
(u-\tfrac{i-1}{2}) - E_{1,1} & - E_{1,2} & \cdots & - E_{1,i} \\ 
- E_{2,1} &  (u-\tfrac{i-3}{2}) - E_{2,2}  & \cdots & -E_{2,i} \\ 
\vdots & \vdots & \ddots  & \vdots \\
-E_{i,1} & -E_{i,2} & \cdots &  (u+\tfrac{i-1}{2}) - E_{i,i}
\end{pmatrix}
$$
for $i=1,\ldots,n-1$, and also that
$$
R(u) = \operatorname{cdet}
\begin{pmatrix}
(u-\tfrac{n}{2}) - E_{1,1} & - E_{1,2} & \cdots & - E_{1,n} \\
- E_{2,1} & (u-\tfrac{n}{2}+1) - E_{2,2} & \cdots & - E_{2,n} \\
\vdots & \vdots & \ddots & \vdots \\
- E_{n,1} & -E_{n,2} & \cdots & (u-\tfrac{n}{2}+(n-1) ) - E_{n,n}
\end{pmatrix}
$$
Rescaling to get $\hbar =2$ we end up rescaling $A_i^{\text{current}}(u) = A_i^{\text{QMV}}(\frac{1}{2} u)$, so I think that really
$$
A_i^{\text{current}}(u) = \left(\frac{u}{2}\right)^{-i} \operatorname{cdet} 
\begin{pmatrix} 
\frac{u-i+1}{2} - E_{1,1} & - E_{1,2} & \cdots & - E_{1,i} \\ 
- E_{2,1} &  \frac{u-i+3}{2} - E_{2,2}  & \cdots & -E_{2,i} \\ 
\vdots & \vdots & \ddots  & \vdots \\
-E_{i,1} & -E_{i,2} & \cdots &  \frac{u+i-1}{2} - E_{i,i}
\end{pmatrix}
$$
Here, I mean the column determinant
$$
\operatorname{cdet}(A) = \sum_{\sigma \in S_n} (-1)^\sigma A_{\sigma(1), 1}\cdots A_{\sigma(n), n}
$$
}
\bcom{I'm totally confused about how this could possibly by right.  This says that $A_i^{(1)}=-E_{11}+\cdots -E_{ii}$ which is an {\bf anti-dominant} coweight, so we can't be sending $E_i^{(1)}$ to $E_i$.  Did I do something wrong?}
\ocom{Ok, so we have this correspondence:
$$
(\B(\bla)\otimes  \B(\infty))_{\mu_{\mathscr{S}}} \leftrightarrow \bigcup_\mathfrak{m}\{\text{\GTc\: modules of $U(\mathfrak{gl}_N)$ in the fiber of }\mathfrak{m} \}
$$
where the union is over lifts of $\mathscr{S}$ to weights $\mathfrak{m}$ .
}

}

\bibliography{./monbib}
\bibliographystyle{amsalpha}

\end{document}